\documentclass[10pt,amstex,amssymb]{amsart}
\usepackage{amsmath,amsthm,amsfonts,amssymb,amscd}
\usepackage[latin1]{inputenc}
\usepackage[all]{xy}
\usepackage[dvips]{graphicx}
\usepackage{xcolor}
\usepackage{comment}
\def\vr{{\varphi}}
\newcommand\dint{\displaystyle\int}

\newtheorem{theorem}{Theorem}

\newtheorem{Theorem}{Theorem}[section]

\newtheorem{Lemma}{Lemma}[section]

\newtheorem{Proposition}{Proposition}[section]
\newtheorem{Definition}{Definition}[section]
\newtheorem{Remark}{Remark}[section]
\newtheorem{Example}{Example}[section]

\title[On singular Frobenius  for second order]{On singular Frobenius  for second order linear partial differential equations}
\author{V. Le\'on}
\address{V. Le\'on. ILACVN - CICN, Universidade Federal of the Integração Latino-Americano, Parque tecnológico from Itaip\'u, Foz do  Iguaçu-PR, 85867-970 - Brazil}
\email{victor.leon@unila.edu.br}

\author{B. Sc\'ardua}
\address{B. Sc\'ardua. Instituto from Matem\'atica - Universidade Federal do the Rio from Janeiro,
CP. 68530-Rio from Janeiro-RJ, 21945-970 - Brazil}
\email{bruno.scardua@gmail.com}


\begin{document}
\begin{abstract}
The main subject of this paper is the study of analytic second order linear partial differential equations.
 We aim to solve the classical equations and some more, in the real or complex analytical case. This is done by introducing methods inspired by the method of  Frobenius method for second order linear ordinary differential equations. We introduce a notion of  Euler type partial differential equation. To such a PDE we associate an indicial conic, which is an affine plane curve of degree two. Then comes the concept of regular singularity and finally convergence theorems, which must necessarily take into account the type of PDE (parabolic, elliptical or hyperbolic) and a nonresonance condition. This condition gives  a new geometric interpretation of the original condition between the roots of the original Frobenius theorem for second order ODEs. The interpretation is something like, a certain reticulate has or not vertices on the indexical conic. Finally, we retrieve the solution of all the classical PDEs  by this method (heat diffusion, wave propagation and Laplace equation), and also increase the class of those that have explicit algorithmic solution to far beyond those admitting separable variables. The last part of the paper is dedicated to the construction of PDE models
 for the classical ODEs like Airy, Legendre, Laguerre, Hermite and Chebyshev by two different means. One model is based on the requirement that the restriction of the PDE to lines through the origin must be the classical ODE model. The second is based on the idea of having symmetries on the PDE model and imitating the ODE model.
We study these PDEs and obtain their solutions, obtaining for the framework of PDEs some of the classical 
results, like existence of polynomial solutions (Laguerre, Hermite and Chebyshev polynomials).

\end{abstract}

\maketitle
\tableofcontents

\section{Introduction}
One of the most applicable fields in mathematics is the subject of partial differential equations.
Indeed, a number of phenomena as heat diffusion, waves propagation and electromagnetic forces are
modeled by these equations. All  these equations belong to the  class of second order linear equations. This class plays
a special role in the theory. Indeed, it includes the above mentioned and several other modern problems as nuclear reactions and atomic models.
Since the work of Euler, Lagrange, Bernoulli and Fourier, among others, the techniques  for solving such equations are based on reducing, at least partially, the original PDE to a system of ordinary differential equations and trying to solve these ODEs. This is done by the use of special transforms as Fourier or Laplace transforms. Another possibility is the use of the theory of distributions and operators as the heat kernel.
All this works pretty well, but has some limitations. One of the first to show up is the fact that the equations must have constant coefficients or, at least, the majority of these. Another restriction of the classical methods is the fact that not all PDEs can have separated variables. Indeed, already some simple perturbations of
classical equations fall outside of the class of equations admitting this separation.
Our aim is to bring to the framework of partial differential equations some of the techniques used in
the study of ordinary differential equations.
Before going through this, let us recall the ordinary differential case.

The classical {\em method of Frobenius} is a very useful tool in finding solutions of
a homogeneous second order linear ordinary differential equations with analytic coefficients.
These are equations that write in the form
$a(x) y^{\prime \prime} + b(x)y^\prime  + c(x)y=0$ for some real analytic functions $a(x), b(x), c(x)$ at some point
$x_0\in \mathbb R$. It is well known that if $x_0$ is an ordinary point, i.e., $a(x_0)\ne 0$ then there
are two linearly independent
solutions $y_1(x), y_2(x)$ of the ODE, admitting  power series expansions converging in some common neighborhood of
$x_0$. This is a consequence of the classical theory of ODE and also shows that the solution space of this ODE
has dimension two, i.e., any solution is of the form $c_1 y_1(x) + c_2 y_2(x)$ for some constants $c_1, c_2 \in \mathbb R$.
Second order linear homogeneous differential equations appear in many concrete
problems in natural sciences, as physics,
chemistry, meteorology and even biology. Thus solving such equations is an important task.
The existence of solutions for the case of an ordinary point is not enough for most of the applications.
Indeed, most of the relevant equations are connected to the singular (non-ordinary) case. We can mention
Bessel equation $x^2 y^{\prime \prime}  + x y ^\prime + (x^2 - \nu^2) y=0$, whose range of applications goes from
heat conduction, to the model of the hydrogen atom. This equation has  the origin $x=0$ as a singular point.
Another remarkable equation is the {\it Laguerre equation} $x y ^{\prime \prime} + (\nu +1 -x) y^\prime + \lambda y=0$ where
$\lambda, \nu \in \mathbb R$ are parameters. This equation is quite relevant  in quantum mechanics, since
it appears in the modern quantum mechanical description of the hydrogen atom.
All these are examples of equations with a {\it regular singular point}.
The classical Frobenius methods for second order ODE is found originally found in
\cite{Frobenius} and, more recently, in \cite{Boyce,C} and for higher order in \cite{Leon-Scardua1}.

 The above brief description of the method of Frobenius already suggests that there may exist a
  version of this result for partial differential equations.

 We shall work with the class of second order analytic linear homogeneous partial differential equations. Such an PDE is of the form
\[
a(x,y)\frac{\partial^2 u}{\partial x^2}+b(x,y)\frac{\partial^2 u}{\partial x\partial y}+c(x,y)\frac{\partial^2 u}{\partial y^2}+d_1(x,y)\frac{\partial z}{\partial x}+d_2(x,y)\frac{\partial z}{\partial y}+d_3(x,y)z=f(x,y)
\]
where the coefficients $a(x,y), b(x,y), c(x,y), d_1(x,y), d_2(x,y), d_3(x,y)$ and $f(x,y)$ are complex analytic functions defined in some
open subset  $U\subset \mathbb K^2$ where $\mathbb K \in \{\mathbb R, \mathbb C\}$.

This is a very meaningful class when it comes to classical natural phenomena.
We shall first introduce a simple model for these equations, which is a kind of Euler PDE of second order.
This will be a
PDE of the form
\begin{equation}\label{eq1intro}
Ax^2\frac{\partial^2 z}{\partial x^2}+Bxy\frac{\partial^2 z}{\partial x\partial y}+Cy^2\frac{\partial^2 z}{\partial y^2}+Dx\frac{\partial z}{\partial x}+Ey\frac{\partial z}{\partial y}+Fz=0,\;x>0,y>0
\end{equation}
where $A,B,C,D,E,F\in\mathbb{K}$.

To this equation we shall associate an indicial polynomial which defines a conic in the affine place.
This conic $\mathcal C\subset \mathbb C^2$ is given by
\[
P(r,s)=Ar^2+Brs+Cs^2+(D-A)r+(E-C)s+F=0
\]

This conic plays the role of the indicial equation for Euler ordinary differential equations of order two. Then we shall introduce  a notion of regular singularity for such equations, which is a natural adaptation of the
notion of regular singularity imported from the theory of second order linear ordinary differential equations.
We are able to prove:

\begin{theorem}
\label{Theorem:EulerEDP}{ Consider the following Euler type PDE
\begin{equation}\label{eq1ThmEulerEDP}
Ax^2\frac{\partial^2 z}{\partial x^2}+Bxy\frac{\partial^2 z}{\partial x\partial y}+Cy^2\frac{\partial^2 z}{\partial y^2}+Dx\frac{\partial z}{\partial x}+Ey\frac{\partial z}{\partial y}+Fz=0,\;(x,y) \in \mathbb K^2
\end{equation}
where $A,B,C,D,E,F\in\mathbb{K}$. Given
$(r,s)\in \mathcal C^2$ we have the following equivalence:
\begin{enumerate}
\item The point $(r,s)$ belongs to the indicial conic $\mathcal C$.
 \item We have a
 complex monomial solution of the form $x^ry^s$.
 \end{enumerate}
 If $\mathbb K=\mathbb R$ the  indicial conic  is of one of the following  types:
\begin{enumerate}
\item[{\rm(a)}] Ellipse, if $B^2-4AC<0$.
\item[{\rm(b)}] Parabola, if $B^2-4AC=0$
\item[{\rm(c)}] Hyperbole, if $B^2-4AC>0$.
\end{enumerate}}
\end{theorem}

This however does not assure the existence of monomial polynomial or monomial rational solutions, since the
points of the indicial conic may be all non-integral. Using some basic Linear Algebra we prove:

\begin{theorem}
\label{Theorem:firstintegralconicnormalform}
Let $E$ be a PDE of Euler type. Then we have the following possibilities:

\begin{enumerate}

\item $E$ is elliptic $\implies$ there are two linearly independent degree one polynomial solutions.

\item $E$ is hyperbolic $\implies$ there is one  polynomial solution $q(x,y)$ and one solution of the form
$1/q(x,y)$, where $q(x,y)$ is a degree one polynomial.

\item $E$ is parabolic $\implies$ there are infinitely many linearly independent
degree $n+ n^2$ polynomial solutions of the form $p(x,y)^n q(x,y)^{n^2}, n \in \mathbb N$ where
$p(x,y), q(x,y)$ are degree one polynomials.

\item $E$ is of degenerate type (1) $\implies$ there are infinitely many linearly independent
degree $n$ polynomial solutions of the form $p(x,y)^n$ and infinitely many solutions of the form $q(x,y)^m$ where $p(x,y), q(x,y)$ are
degree one polynomials intersecting transversely.

\item $E$ is of degenerate type  (2) $\implies$ there are infinitely many $r+m$ degree polynomial solutions of the form $p(x,y)^r q(x,y)^m$

\end{enumerate}
\end{theorem}

Let us consider a second order linear homogeneous PDE of the form
\[
A(x,y)\frac{\partial^2 z}{\partial x^2}+B(x,y)\frac{\partial^2 z}{\partial x\partial y}+C(x,y)\frac{\partial^2 z}{\partial y^2}+D(x,y)\frac{\partial z}{\partial x}+E(x,y)\frac{\partial z}{\partial y}+F(x,y)z=0,
\]
where the coefficients $A,B,C,D$ and $E$ are analytic functions at some point $(x_0,y_0)\in \mathbb R^2$.
{\rm The point $(x_0,y_0)$ is called {\it ordinary} if some of the coefficients $A, B$ and  $C$ does not vanish at $(x_0,y_0)$.
Otherwise, if $A(x_0,y_0)=B(x_0,y_0)=C(x_0,y_0)=0$,  it will be called a {\it singular point}. A singular point will be called a {\it regular singularity} if the PDE can be put into the form
\[
Ax^2\frac{\partial^2 z}{\partial x^2}+Bxy\frac{\partial^2 z}{\partial x\partial y}+Cy^2\frac{\partial^2 z}{\partial y^2}+xa(x,y)\frac{\partial z}{\partial x}+yb(x,y)\frac{\partial z}{\partial y}+c(x,y)z=0
\]
with $A, B, C\in \mathbb R$ constants, $a(x,y), b(x,y), c(x,y)$ analytic functions,
after some division by coefficients and change of coordinates centered at $(x_0,y_0)$.

The indicial conic of the PDE will be the indicial conic of the corresponding Euler PDE.
If the PDE has real analytic coefficients we shall then say that it is {\it parabolic, elliptic} or
{\it hyperbolic} according to the indicial conic.

 The notion of regular singular point above gives rise to a version for this framework of PDE of the classical method of Frobenius for finding solutions via power series. The picture is not so straightforward, since we are dealing with a degree two affine curve instead of a one variable degree polynomial. First we shall introduce a notion of non-resonance, which extends and gives geometric sense to the main obstruction regarding the roots of the indicial equation in the classical Frobenius theorem for ordinary differential equations.
 A point in the indicial conic $(r_0,s_0)\in \mathcal C$ is {\it resonant} if there is some non-trivial positive translation of $(r_0,s_0)$ by
integral shift $(r_0 +q_1, s_0 +q_2), \, q_1, q_2\in \mathbb Z_+$ which also lies on the indicial conic.
This can be seen as follows: consider the reticulate $r_0\vee s_0 \subset\mathbb C^2$ centered at $(r_0,q_0)$. This means   the set of all points of the form $ (r_0 +q_1,s_0 +q_2)$ where $q_1,q_2\in \mathbb Z$. The positive part of the reticulate is the set of points of the form $ (r_0 +q_1,s_0 +q_2)$ where $q_1,q_2\in \mathbb N\cup\{0\}$.  Then, a point $(r_0,s_0)\in \mathcal C$ of the indicial conic is resonant if there is some vertex of the positive part of the reticulate that lies over the indicial conic.

 We then prove the following Frobenius type result the following existence constructive results:

\begin{theorem}[parabolic real analytic]
\label{Theorem:realparabolic}
A parabolic real analytic second order linear homogeneous PDE, with a regular singularity, admits a convergent Frobenius type solution for each pair $(r,s)\in \mathbb R^2$ in the indicial conic.
Indeed, there is a solution of the form $x^ r y ^s \sum\limits_{Q \in \mathbb N^2} a_q x^{q_1}y^{q_2}$
where the series converges in some neighborhood of the origin, the coefficients $a_Q$ are  real and obtained recursively from the PDE.
\end{theorem}
Theorem~\ref{Theorem:realparabolic} is a consequence of the more general Theorem~\ref{Theorem:paraboliccomplex} which holds for complex PDEs.
Versions of Theorem~\ref{Theorem:realparabolic} for the elliptic and hyperbolic case are also proved. The case of elliptic
equations is stated below and follows from the more general statement in Theorem~\ref{Theorem:ellipticcomplex} for complex PDEs:
\begin{theorem}[elliptic real analytic]
\label{Theorem:realelliptic}
An elliptic second order linear homogeneous real analytic PDE
having a regular singular point, admits  Frobenius type solutions. More precisely,
consider the real analytic second order linear PDE
$$L[z]:=Ax^2\frac{\partial^2 z}{\partial x^2}+Cy^2\frac{\partial^2 z}{\partial y^2}+xa(x,y)\frac{\partial z}{\partial x}+yb(x,y)\frac{\partial z}{\partial y}+c(x,y)z=0$$
where  $A{C}>0$, $a(x,y),b(x,y)$ and $c(x,y)$ analytic in $\Delta[(0,0),(R,R)]$, $R>0$. Let $(r_0,s_0)\in \mathbb R^2$ be a non-resonant real point of the indicial conic $\mathcal C \subset \mathbb C^2$. Then
$L[z]=0$ admits a convergent Frobenius type solution with initial monomial $x^{r_0} y^{s_0}$.
\end{theorem}

Next we address the hyperbolic case and obtain as a result of the complex case stated in Theorem~\ref{Theorem:hyperboliccomplex} we obtain:
\begin{theorem}[hyperbolic real analytic]
\label{Theorem:realhyperbolic}
A hyperbolic real analytic second order linear homogeneous complex analytic PDE
having a regular singular point, admits  Frobenius type solutions. More precisely,
consider the second order PDE

$$L[z]:=Ax^2\frac{\partial^2 z}{\partial x^2}+Cy^2\frac{\partial^2 z}{\partial y^2}+xa(x,y)\frac{\partial z}{\partial x}+yb(x,y)\frac{\partial z}{\partial y}+c(x,y)z=0$$

where  ${A}{C}<0$, $a(x,y),b(x,y)$ and $c(x,y)$ real analytic  in $\Delta[(0,0),(R,R)]$, $R>0$.
Let $(r_0,s_0)\in \mathbb R^2$ be a non-resonant real point of the indicial conic. Then
$(\mathcal E)$ admits a convergent Frobenius type solution with initial monomial $x^{r_0} y^{s_0}$.
\end{theorem}

Our results are a first step in the reintroduction
of techniques from ordinary differential equations in the study of classical problems
involving partial differential equations. Our solutions are constructive and computationally
viable. We give a number of examples illustrating the range of these techniques. A further study of these examples is found in a forthcoming work.

\section{Second order linear equations}

Throughout this paper $\mathbb K$ will denote the field of complex numbers $\mathbb C$ or the field of real numbers $\mathbb R$, according to the context.
A {\it second order linear PDE} in two (real) variables is an expression of the form
\[
a(x,y)\frac{\partial^2 u}{\partial x^2}+b(x,y)\frac{\partial^2 u}{\partial x\partial y}+c(x,y)\frac{\partial^2 u}{\partial y^2}+d_1(x,y)\frac{\partial z}{\partial x}+d_2(x,y)\frac{\partial z}{\partial y}+d_3(x,y)z=f(x,y)
\]
where the coefficients $a(x,y), b(x,y), c(x,y), d_1(x,y), d_2(x,y), d_3(x,y)$ and $f(x,y)$ are functions defined in some
open subset  $U\subset \mathbb R^2$.

The class of differentiability of the PDE is the common class of differentiability of its coefficients.

{\em We shall be working with analytic PDEs}, i.e., the coefficients will be assumed to be analytic in their domain of definition. If $\mathbb K=\mathbb R$ this means that they are real analytic. If $\mathbb K=\mathbb C$ this means that they are complex analytic (holomorphic).
The PDE is {\it homogeneous} if $f(x,y)=0$.
The simplest model for these PDEs is the one with constant coefficients, given by
\begin{equation}\label{eq14}
A\frac{\partial^2 z}{\partial x^2}+B\frac{\partial^2 z}{\partial x\partial y}+C\frac{\partial^2 z}{\partial y^2}+D\frac{\partial z}{\partial x}+E\frac{\partial z}{\partial y}+Fz=0.
\end{equation}

\subsection{Classification of real PDEs}
Let us assume that $\mathbb K=\mathbb R$.
It is well-known that the sign of the term $ac-b^2\in \mathbb R$ does not depend on the choice of the coordinates
$(x,y)\in U\subset \mathbb R^2$ above. This allows a partial classification of such PDEs as follows:
\noindent The equation is {\it elliptic} if $ac-b^2>0$ everywhere in $U$.

\noindent The equation is {\it hyperbolic} if $ac -b^2 <0$ everywhere in $U$.

\noindent The equation is {\it parabolic} if $ac=b^2$ everywhere in $U$.

Notice that a PDE may change its behavior in different regions of the domain $U$.

An example of a hyperbolic equation is the {\it wave equation} $u_{xx}= \lambda ^2 u_{tt}$.
 By its turn the {\it heat diffusion equation} $u_{xx} = \lambda^2 u_t$ is parabolic.  Finally, {\it Laplace's equation} $u_{xx} + u_{yy}=0$ is elliptic.

\subsection{Complex analytic PDEs}

We shall consider PDEs with complex analytic coefficients.

\begin{Definition}
{\rm A second order complex analytic linear PDE in two  variables is an expression of the form
\[
a(x,y)\frac{\partial^2 u}{\partial x^2}+b(x,y)\frac{\partial^2 u}{\partial x\partial y}+c(x,y)\frac{\partial^2 u}{\partial y^2}+d_1(x,y)\frac{\partial z}{\partial x}+d_2(x,y)\frac{\partial z}{\partial y}+d_3(x,y)z=f(x,y)
\]
where the coefficients $a(x,y), b(x,y), c(x,y), d_1(x,y), d_2(x,y), d_3(x,y)$ and $f(x,y)$ are complex analytic functions defined in some
open subset  $U\subset \mathbb C^2$.
The PDE is homogeneous if $f(x,y)=0$.}
\end{Definition}
\section{Euler equation for second order PDEs}

In this section we build the basis of the method of Frobenius for PDE. The first step
is to introduce the notion of Euler equation in this framework. Let us get started.

\subsection{Euler partial differential equation of second order}
We shall consider the PDE of the form
\begin{equation}\label{eq1}
Ax^2\frac{\partial^2 z}{\partial x^2}+Bxy\frac{\partial^2 z}{\partial x\partial y}+Cy^2\frac{\partial^2 z}{\partial y^2}+Dx\frac{\partial z}{\partial x}+Ey\frac{\partial z}{\partial y}+Fz=0,\;x>0,y>0
\end{equation}
where $A,B,C,D,E,F\in\mathbb{K}$.
This will be called an {\it Euler partial differential equation of second order} in $\mathbb K^2$.
Notice that $(B xy)^2 - 4 (Ax^2) (Cy^2)= (B^2 - 4 AC) x^2y^2$. Thus we have the following possibilities for
the Euler PDE, according to the classical terminology for linear PDEs:
Denote by $\mathbb K^*=\mathbb K \setminus \{0\}$.
Assume that the equation is {\it real}, ie., $\mathbb K=\mathbb R$, it has real coefficients. In this case the  Euler equation is classified as:

\noindent {\it Hyperbolic} if $B^2 - 4AC>0$,

\noindent {\it Parabolic} if $B^2 =4AC$,

\noindent {\it Elliptic} if  $B^2 -4AC<0$.

We point-out that this coincides with the classical notion for PDEs given above if we consider $U=(\mathbb R^*)^2$.

Now we turn our attention to the Euler equation
above in $\mathbb K^2$ where $\mathbb K \in \{\mathbb R, \mathbb C\}$.

\begin{Definition}[indicial conic]
{\rm The {\it indicial conic} associate to the Euler equation above is defined as the {\em complex} affine plane curve
$\mathcal C:=\{(r,s) \in \mathbb C^2; P(r,s)=0\}$ where:
\[
P(r,s)=Ar^2+Brs+Cs^2+(D-A)r+(E-C)s+F.
\]
In the case of a real PDE, ie., when the coefficients of the PDE are real, we define the  {\it real trace of the indicial conic} (or {\it the real conic} for short) as the intersection
$\mathcal C(\mathbb R):=\mathcal C \cap \mathbb R^2$. The real conic  is the real affine plane curve
$\{(r,s)\in \mathbb R^2, \, P(r,s)=0\}$.}
\end{Definition}

It is clear from the above that, in case $\mathbb K=\mathbb R$,  the Euler equation is hyperbolic, parabolic or elliptic according to the
fact that the corresponding real conic $\mathcal C(\mathbb R) \subset \mathbb R^2$ is a hyperbole, a parabola or an ellipse.

\begin{Definition}[monomial solutions]
{\rm A {\it monomial solution} of the Euler equation is a solution of the form $\vr=x^r y^s$ where
$r, s\in \mathbb C$. The solution is {\it real} if $(r,s)\in \mathbb R^2$. }
\end{Definition}

\begin{proof}[Proof of Theorem~\ref{Theorem:EulerEDP}]
Let $\vr=x^ry^s$ and consider the differential operator $L[z]=Ax^2\frac{\partial^2 z}{\partial x^2}+Bxy\frac{\partial^2 z}{\partial x\partial y}+Cy^2\frac{\partial^2 z}{\partial y^2}+Dx\frac{\partial z}{\partial x}+Ey\frac{\partial z}{\partial y}+Fz$. Then $\vr$  is solution of (\ref{eq1ThmEulerEDP}) if and only if $L[\vr]=0$. Substituting we have
$$L[\vr]=L[x^r y^s]=Ax^2(r(r-1)x^{r-2}y^s)+Bxy(rsx^{r-1}y^{s-1})+Cy^2(s(s-1)x^{r}y^{s-2})+Dx(rx^{r-1}y^s)$$
$$
\\+Ey(sx^ry^{s-1})+Fx^ry^s$$
$$=x^ry^s\big[Ar(r-1)+Brs+Cs(s-1)+rD+sE+F\big].$$
Therefore
$\vr=x^r y^s$ is a solution of (\ref{eq1ThmEulerEDP}) if and only if
\begin{equation}\label{eq2'}
Ar^2+Brs+Cs^2+(D-A)r+(E-C)s+F=0.
\end{equation}

The last part, regarding the classification of the indicial conic, is well-known from Analytic Geometry and Linear Algebra in dimension two.
 \end{proof}

Since the PDE is linear homogeneous, any linear combination of monomial solutions is also a solution.
Indeed, from the merely formal point of view, not all solutions of the Euler equation above are of the
``monomial form". Indeed, since the PDE is linear, any linear combination or a series $\vr=\sum\limits_{n=0} c_n x^{r_n}y^{s_n}$ where $c_n \in \mathbb K$ and $(r_n,s_n)\in \mathcal C$ are points in the indicial conic, is also a solution of the
Euler equation. Later on we shall discuss the restrictions imposed by regularity of the solutions and
boundary or initial conditions.

Let us now assume that we are in the real case $(x,y)\in \mathbb R^2$. Assume that we have a point
of the indicial conic $(r,s)\in \mathcal C\subset \mathbb C^2$   which we write in real and imaginary part  $r=r_1 + i r_2, s= s_1 + i s_2$ where $i ^2=-1$. Then given $x,y\in \mathbb R_+$ we have $x^r=x^{r_1}[\cos(r_2 \ln x) + i \sin (r_2 \ln x)]$  and $y^s=y^{s_1}[\cos(s_2 \ln y) + i \sin (s_2 \ln y)]$.
Thus we get
$
x^r y ^s = x^{r_1} y^{s_1} [\cos(r_2 \ln x)\cos(s_2 \ln y) - \sin(r_2 \ln x) \sin (s_2 \ln y)] + ix^{r_1}y^{s_1}[\cos(r_2 \ln x) \sin (r_2 \ln y)+  \sin (r_2 \ln x)\cos(s_2 \ln y)].
$
Hence we obtain two linearly independent real solutions
\[
u_1(x,y)=x^{r_1} y^{s_1} [\cos(r_2 \ln x)\cos(s_2 \ln y) - \sin(r_2 \ln x) \sin (s_2 \ln y)]
 \]
 and
 \[
 u_2(x,y)=x^{r_1}y^{s_1}[\cos(r_2 \ln x) \sin (r_2 \ln y)+  \sin (r_2 \ln x)\cos(s_2 \ln y)],
  \]
  analytic defined for $x>0,y>0$.

\begin{Remark}[Euler type $PDE \ne (ODE)^2$ of Euler type]{\rm
Consider the following PDE
\begin{equation}\label{eqremark1}
Ax^2\frac{\partial^2 z}{\partial x^2}+Bxy\frac{\partial^2 z}{\partial x\partial y}+Cy^2\frac{\partial^2 z}{\partial y^2}+Dx\frac{\partial z}{\partial x}+Ey\frac{\partial z}{\partial y}+Fz=0,\;x>0,y>0
\end{equation}
where $A,B,C,D,E,F\in\mathbb{R}$. The method of separation of variables can be used to obtain the solution of equation (\ref{eqremark1}). Assuming that
$$z(x,y)=X(x)Y(y)$$
and substituting for $z$ in equation (\ref{eqremark1}), we obtain
$$Ax^2X^{\prime\prime}(x)Y(y)+BxyX^\prime(x)Y^\prime(y)+Cy^2X(x)Y^{\prime\prime}(y)+DxX^\prime(x)Y(y)+EyX(x)Y^\prime(y)+FX(x)Y(y)=0.$$
If $B=0$ we obtain
$$[Ax^2X^{\prime\prime}(x)+DxX^\prime(x)+FX(x)]Y(y)+[Cy^2Y^{\prime\prime}(y)+EyY^\prime(y)]X(x)=0$$
equivalently
\begin{equation}\label{eqremark2}
\frac{Ax^2X^{\prime\prime}(x)+DxX^\prime(x)+FX(x)}{X(x)}=-\frac{Cy^2Y^{\prime\prime}(y)+EyY^\prime(y)}{Y(y)}
\end{equation}
in which the variables are separated; that is, the left side depends only on $x$ and the
right side only on $y$. If we call this separation constant $\lambda$, then equation (\ref{eqremark2}) becomes
$$\frac{Ax^2X^{\prime\prime}(x)+DxX^\prime(x)+FX(x)}{X(x)}=-\frac{Cy^2Y^{\prime\prime}(y)+EyY^\prime(y)}{Y(y)}=\lambda$$
Hence we obtain the following two ordinary differential equations for $X(x)$ and $Y(y)$:
\begin{equation}\label{eqremark3}
Ax^2X^{\prime\prime}(x)+DxX^\prime(x)+(F-\lambda)X(x)=0,
\end{equation}
\begin{equation}\label{eqremark4}
Cy^2Y^{\prime\prime}(y)+EyY^\prime(y)+\lambda Y(y)=0.
\end{equation}
Note that (\ref{eqremark3}) and (\ref{eqremark4}) are Euler equations. Therefore the solutions of (\ref{eqremark3}) and (\ref{eqremark4}) are of the form $x^{r(\lambda)}$ and $y^{s(\lambda)}$, where the indexes $r(\lambda), s(\lambda)$ depend on the separation constant $\lambda$. If we consider the two
indicial equations they are given by
\[
Ar(r-1) + Dr + F - \lambda =0, \, \, C r (r-1) + Er + \lambda=0.
\]
Eliminating $\lambda$ in the system above we obtain the indicial equation of the Euler PDE.
The dependence on the parameter $\lambda$ makes inviable the obtaining of solutions for the PDE separately
via classical Euler ordinary differential equations. It also suggests that more complicate situations,
like PDEs which are of the form $Ax^2\frac{\partial^2 z}{\partial x^2}+Bxy\frac{\partial^2 z}{\partial x\partial y}+Cy^2\frac{\partial^2 z}{\partial y^2}+Dx\frac{\partial z}{\partial x}+Ey\frac{\partial z}{\partial y}+Fz + h.o.t.=0$ (where $h.o.t.$ stands for higher order terms), are not always treatable via the classical ordinary differential equations Frobenius methods.

}
\end{Remark}

\begin{Definition}[analytic, meromorphic and real monomial solution]
{\rm A monomial solution $z=x^r y^s$ of an Euler equation
will be called {\it analytic} (no matter the PDE is real or complex) (respectively, {\it meromorphic}) if $(r,s)\in \mathbb N$ (respectively, $(r,s)\in  \mathbb Z$). The solution is
{\it real} if $r, s \in \mathbb R$.}
\end{Definition}
Let us consider the complex analytic case for the PDE. According to Theorem~\ref{Theorem:EulerEDP} an analytic (respectively, meromorphic) monomial solution $z=x^r y^s$ corresponds to a natural (respectively,  integral)
point $(r,s)\in \mathbb N^2$ (respectively, $(r,s)\in \mathbb Z^2$) in the indicial conic $\mathcal C$ of the Euler equation. A real  solution corresponds to a real point of the conic.

An affine change of coordinates $T(x,y)=(u,v), \, u=a x + by + e , v= cx + dy + g$ in $\mathbb R^2$,
takes an Euler equation $\mathcal E$ into an Euler equation $T_*(\mathcal E)$.
It is also clear that the indicial conics $\mathcal C(\mathcal E)$ of $\mathcal C(T_*\mathcal E)$ are
related by $\mathcal C(T_*\mathcal E)=T(\mathcal C(\mathcal E))$. In particular, these conics are of a same type
(hyperbole, parabola or ellipse).
Another easy to see fact, partially consequence of the above, is that {\em given any conic $\mathcal C\subset \mathbb R^2$ there is an Euler equation $\mathcal E$ which has this conic as indicial conic.}

Let us now consider a real Euler PDE. The indicial conic is given by a real equation and after a real  affine change of coordinates in $\mathbb R^2$ we have the following
possible {\it normal forms} for the real trace of  $\mathcal C$:

\begin{enumerate}
\item $C: r^2 + s^2 =1$ (elliptic type)

\item $C: r^2 - s^2 =1$ (hyperbolic type)

\item $C: s=r^2$ (parabolic type)
\item $C: rs=0$ (degenerate type 1)

\item $C: r(r-1)=0$ (degenerate type 2)
\end{enumerate}
Based on this we shall call a PDE of Euler type, complex or real, with coefficients in the
field $\mathbb K$, as
{\it parabolic, elliptic, hyperbolic, degenerate type 1, degenerate type 2}, according to the
normal form above of its indicial conic after some affine change of coordinates with coefficients
in the field $\mathbb K$.

\begin{proof}[Proof of Theorem~\ref{Theorem:firstintegralconicnormalform}]
The proof is based on two steps. First,  by hypothesis, the original Euler PDE $\mathcal E$ has its
indicial conic $\mathcal C(\mathcal E)$ taken into this normal form $T(\mathcal C(\mathcal E))$ by some affine change of coordinates say
$(\xi, \eta)=T(x,y)=(a + bx + cy, \tilde a + \tilde b x + \tilde c y)=(p(x,y),q(x,y))$ with coefficients $a,b,c,\tilde a, \tilde b, \tilde c \in \mathbb K$. This same transformation $T$  takes the original PDE $\mathcal E$ into a PDE $T(\mathcal E)$ with indicial conic $\mathcal C(T(\mathcal E))=T (\mathcal C(\mathcal E))$ in the normal form. Given then a monomial solution $\vr=x^r y^s$ of $T(\mathcal E)$ the composition $\vr \circ T=(p(x,y))^r(q(x,y))^s$ is a first integral for the   original PDE $\mathcal E$.
Second, now we just need to consider the normal form and obtain
a monomial solution $x^r y^s$ with suitable exponents $r,s\in \mathbb Z$, according to the model.
For instance, the elliptic model admits two point $(r,s)\in \{(0,1), (1,0)\}$ and therefore
monomial solutions of the form $x$ and $y$.

\end{proof}

\subsection{Examples of Euler equations}
 Next we shall give a list of Euler equations and their corresponding indicial conics. This list contains
examples of all the possible cases for the indicial conic.

\noindent{\bf 1- Elliptic case}:
\[
4x^2\frac{\partial^2 z}{\partial x^2}+9y^2\frac{\partial^2 z}{\partial y^2}-36x\frac{\partial z}{\partial x}+45y\frac{\partial z}{\partial y}+100z=0
 \rightarrow \frac{(r-5)^2}{9}+\frac{(s+2)^2}{4}=1.
 \]

\noindent{\bf 2- Degenerate type 1}:
\[
36x^2\frac{\partial^2 z}{\partial x^2}+9y^2\frac{\partial^2 z}{\partial y^2}-72x\frac{\partial z}{\partial x}+15y\frac{\partial z}{\partial y}+82z=0.
\rightarrow s+\frac{1}{3}=\pm 2i\big(r-\frac{3}{2}\big).
\]

\noindent{\bf 3- Elliptic case}:

\[
9x^2\frac{\partial^2 z}{\partial x^2}+4y^2\frac{\partial^2 z}{\partial y^2}+27x\frac{\partial z}{\partial x}-5y\frac{\partial z}{\partial y}+25z=0.
\rightarrow 9\big(r+1\big)^2+4\big(s-\frac{9}{8}\big)^2=\frac{81}{16}-16=-\frac{175}{16}<0.
\]
The solutions are then of the form
$$r=-1\pm i\frac{5\sqrt{7}}{12}\cos\alpha,\;\;\\;\;\;s=\frac{9}{8}\pm i\frac{5\sqrt{7}}{8}\sin \alpha\;\;\;\alpha\in\mathbb{R}.$$

\noindent{\bf 4- Elliptic case}:
\[
x^2\frac{\partial^2 z}{\partial x^2}+y^2\frac{\partial^2 z}{\partial y^2}-x\frac{\partial z}{\partial x}-y\frac{\partial z}{\partial y}-z=0.
\rightarrow    (r-1)^2+(s-1)^2=1.
\]
\noindent{\bf 5- Hyperbolic case}:
\[
x^2\frac{\partial^2 z}{\partial x^2}-2y^2\frac{\partial^2 z}{\partial y^2}+7x\frac{\partial z}{\partial x}+2y\frac{\partial z}{\partial y}+9z=0\rightarrow r^2-2s^2+6r+4s+9=0
\]

and then by
\[
(s-1)^2-\frac{(r+3)^2}{2}=1.
\]

\noindent{\bf 6- Hyperbolic case}:
\begin{equation}\label{eq8}
9x^2\frac{\partial^2 z}{\partial x^2}-16y^2\frac{\partial^2 z}{\partial y^2}+99x\frac{\partial z}{\partial x}-144y\frac{\partial z}{\partial y}-31z=0.
\end{equation}
has indicial conic  given by
$$9r^2-16s^2+90r-128s-31=0 $$
and therefore by
$$3(r+5)=\mp4(s+4).$$

\noindent{\bf 7- Parabolic case}:
\begin{equation}\label{eq9}
2y^2\frac{\partial^2 z}{\partial y^2}+5x\frac{\partial z}{\partial x}+10y\frac{\partial z}{\partial y}-7z=0.
\end{equation}
has indicial conic  given by
$$2s^2+5r+8s-7=0$$
and therefore by
$$(s+2)^2=-\frac{5(r-3)}{2}.$$

\noindent{\bf 8- Parabolic case}:
\begin{equation}\label{eq10}
x^2\frac{\partial^2 z}{\partial x^2}+2xy\frac{\partial^2 z}{\partial x\partial y}+y^2\frac{\partial^2 z}{\partial y^2}+x\frac{\partial z}{\partial x}+y\frac{\partial z}{\partial y}-z=0.
\end{equation}
has indicial conic  given by
$$r^2+2rs+s^2-1=0$$
and therefore by
$r+s=\pm 1\}$.

\noindent{\bf 9- Parabolic case}:

\begin{equation}\label{eq11}
3y^2\frac{\partial^2 z}{\partial y^2}+10y\frac{\partial z}{\partial y}-6z=0.
\end{equation}
has indicial conic  given by
$$3s^2+7s-6=(3s-2)(s+3)=0.$$

\noindent{\bf 10- Parabolic case}:

\begin{equation}\label{eq12}
3y^2\frac{\partial^2 z}{\partial y^2}+y\frac{\partial z}{\partial y}+z=0.
\end{equation}
has indicial conic  given by
$$3s^2-2s+1=0$$
and therefore by
$$s=\frac{1}{3}\pm i\frac{\sqrt{2}}{3}.$$

\subsection{Euler coordinates}

We shall introduce for our framework of PDEs, a change of variables, inspired in Euler's trick for ordinary differential equations. Start with an Euler equation
\begin{equation}\label{eq1Eulertrick}
Ax^2\frac{\partial^2 z}{\partial x^2}+Bxy\frac{\partial^2 z}{\partial x\partial y}+Cy^2\frac{\partial^2 z}{\partial y^2}+Dx\frac{\partial z}{\partial x}+Ey\frac{\partial z}{\partial y}+Fz=0,\;x>0,y>0
\end{equation}
where $A,B,C,D,E,F\in\mathbb{R}$. We consider the {\it Euler coordinates} $(u,v)\in \mathbb R^2$ which are given by  $x=e^u$ and $y=e^v$. Suppose that $\varphi$ is solution of (\ref{eq1Eulertrick}) and we define $\tilde{\varphi}(u,v)=\varphi(e^u,e^v)$ hence we have
\begin{equation}\label{eq13}
\begin{array}{c} \frac{\partial\varphi}{\partial x}\big(e^u,e^v\big)=e^{-u}\frac{\partial\tilde{\varphi}}{\partial u}(u,v),\,
\frac{\partial\varphi}{\partial y}\big(e^u,e^v\big)=e^{-v}\frac{\partial\tilde{\varphi}}{\partial v}(u,v),\,
\frac{\partial^2\varphi}{\partial x^2}\big(e^u,e^v\big)=e^{-2u}\big(\frac{\partial^2\tilde{\varphi}}{\partial u^2}(u,v)-\frac{\partial\tilde{\varphi}}{\partial u}(u,v)\big),\,\\
\\
\frac{\partial^2\varphi}{\partial x \partial y}\big(e^u,e^v\big)=e^{-u}e^{-v}\frac{\partial^2\tilde{\varphi}}{\partial u \partial v}(u,v),\,
\frac{\partial^2\varphi}{\partial y^2}\big(e^u,e^v\big)=e^{-2v}\big(\frac{\partial^2\tilde{\varphi}}{\partial v^2}(u,v)-\frac{\partial\tilde{\varphi}}{\partial v}(u,v)\big),\end{array}
\end{equation}

Given that  $\varphi$ is solution of (\ref{eq1Eulertrick}) we obtain

 \[
 Ae^{2u}\frac{\partial^2 \varphi}{\partial x^2}(e^u,e^v)+Be^ue^v\frac{\partial^2 \varphi}{\partial x\partial y}(e^u,e^v)+Ce^{2v}\frac{\partial^2 \varphi}{\partial y^2}(e^u,e^v)+De^u\frac{\partial \varphi}{\partial x}(e^u,e^v)+Ee^v\frac{\partial \varphi}{\partial y}(e^u,e^v)+F\varphi(e^u,e^v)=0,
 \]

 and by (\ref{eq13}) we have

$$ A\frac{\partial^2 \tilde{\varphi}}{\partial u^2}(u,v)+B\frac{\partial^2 \tilde{\varphi}}{\partial u\partial v}(u,v)+C\frac{\partial^2 \tilde{\varphi}}{\partial v^2}(u,v)+D\frac{\partial \tilde{\varphi}}{\partial u}(u,v)+E\frac{\partial \tilde{\varphi}}{\partial v}(u,v)+F\tilde{\varphi}(u,v)=0.$$
Thus $\tilde{\varphi}$ satisfies the equation:
$$A\frac{\partial^2 \tilde{z}}{\partial u^2}+B\frac{\partial^2 \tilde{z}}{\partial u\partial v}+C\frac{\partial^2 \tilde{z}}{\partial v^2}+(D-A)\frac{\partial \tilde{z}}{\partial u}+(E-C)\frac{\partial \tilde{z}}{\partial v}+F\tilde{z}=0.$$

Summarizing and extending the formal procedure also to the complex case we have:

\begin{Proposition}
\label{Proposition:Euler}
The introduction of Euler coordinates $x=e^u, y= e^v$ transforms the Euler equation
\begin{equation}\label{eq1PropEulertrick}
Ax^2\frac{\partial^2 z}{\partial x^2}+Bxy\frac{\partial^2 z}{\partial x\partial y}+Cy^2\frac{\partial^2 z}{\partial y^2}+Dx\frac{\partial z}{\partial x}+Ey\frac{\partial z}{\partial y}+Fz=0,
\end{equation}
where $A,B,C,D,E,F\in\mathbb{K}$ into the following PDE with constant coefficients
$$A\frac{\partial^2 \tilde{z}}{\partial u^2}+B\frac{\partial^2 \tilde{z}}{\partial u\partial v}+C\frac{\partial^2 \tilde{z}}{\partial v^2}+(D-A)\frac{\partial \tilde{z}}{\partial u}+(E-C)\frac{\partial \tilde{z}}{\partial v}+F\tilde{z}=0.$$
\end{Proposition}

\section{Classical equations: heat diffusion, wave propagation and Laplace equation}
As mentioned in the Introduction, the classical PDE equations which are classically solved by
separation of variables, can also be solved by our methods. Let us study the classical  heat, wave and Laplace equations.
\subsection{Heat diffusion equation}
 Consider
\begin{equation}\label{eqcalor}
     a^2\frac{\partial^2 z}{\partial x^2} =\frac{\partial z}{\partial y},\;\;x>0,y>0
     \end{equation}
where $a$ is a positive constant. Equation (\ref{eqcalor}) is of the form (\ref{eq14}) with $A=a^2$, $E=-1$ and $B=C=D=F=0$. According to Proposition~\ref{Proposition:Euler},  equation (\ref{eqcalor}) is transformed, by putting $x=\ln u, y=\ln v$, in
\begin{equation}\label{eqcalor2}
a^2u^2\frac{\partial^2 \tilde{z}}{\partial u^2}+a^2u\frac{\partial \tilde{z}}{\partial u}-v\frac{\partial \tilde{z}}{\partial v}=0.
\end{equation}
This equation has parabolic indicial conic of the form $a^2r^2-s=0$,  hence the solutions of (\ref{eqcalor2}) are of the form
$$u^rv^{a^2r^2}$$for any $r$. Back to the original variables the solution of (\ref{eqcalor}) are of the form

$$e^{rx}e^{a^2r^2y}$$for any $r$.

\subsubsection{Boundary conditions} {\rm In equation (\ref{eqcalor}) we consider boundary conditions of type $z(0,y)=0$, $z(L,y)=0$ for every $y>0$. In order to have $z(x,y)=e^{rx}e^{a^2r^2y}$ as solutions, verifying these  boundary conditions, we must have the following. Write  $r=\alpha+i\beta$
and then
$$\begin{array}{l c l} z(x,y)&=&e^{rx}e^{a^2r^2y}=e^{\alpha x+a^2(\alpha^2-\beta^2)y}e^{i[\beta x+2\alpha\beta a^2y]}\\&&\\&=&
e^{\alpha x+a^2(\alpha^2-\beta^2)y}\cos(\beta x+2\alpha\beta a^2y)+ie^{\alpha x+a^2(\alpha^2-\beta^2)y}\sin(\beta x+2\alpha\beta a^2y)
 \end{array}$$
It is enough to consider the imaginary part of this solution $w(x,y)=\mbox{Im}(z(x,y))=e^{\alpha x+a^2(\alpha^2-\beta^2)y}\sin(\beta x+2\alpha\beta a^2y)$. The boundary condition then implies
$0=w(0,y)=e^{a^2(\alpha^2-\beta^2)y}\sin(\beta x+ 2\alpha\beta a^2y)$for every $y>0$. We can choose $\alpha=0$. Thence
$$w(x,y)=e^{-a^2\beta^2y}\sin(\beta x)$$
$$0=w(L,y)=e^{-a^2\beta^2y}\sin(\beta L)$$for every $y>0$. This implies
$$\sin(\beta L)=0$$
and then $$\beta L=n\pi,\;\;n=1,2,3,\ldots$$
therefore the solutions are of the form
$$w_n(x,y)=e^{-\frac{a^2n^2\pi^2}{L^2}y}\sin\big(\frac{n\pi x}{L}\big).$$
This is what we obtain by the classical methods of separation of variables.

\subsection{Wave propagation equation}
 Consider
\begin{equation}\label{eqonda}
     a^2\frac{\partial^2 z}{\partial x^2} =\frac{\partial^2 z}{\partial y^2},\;x>0,y>0
 \end{equation}
where $a$ is a positive constant. Equation (\ref{eqonda}) is of the form (\ref{eq14}) with $A=a^2$, $C=-1$ and $B=D=E=F=0$. Equation (\ref{eqonda}) is transformed, by putting $x=\ln u, y=\ln v$, into
\begin{equation}\label{eqonda2}
a^2u^2\frac{\partial^2 \tilde{z}}{\partial u^2}-v^2 \frac{\partial^2 \tilde{z}}{\partial v^2}+a^2u\frac{\partial \tilde{z}}{\partial u}-v\frac{\partial \tilde{z}}{\partial v}=0
\end{equation}
This equation has indicial conic of the form $a^2r^2-s^2=0$ hence the solutions of (\ref{eqonda2}) are of the form
$$u^rv^{\pm ar}$$for any $r$. Back to the original variables a solution of (\ref{eqonda}) are of the form

$$e^{rx}e^{\pm ary}$$for any $r$.
\subsubsection{Boundary conditions} If in equation (\ref{eqonda}) we consider boundary conditions of type $z(0,y)=0$, $z(L,y)=0$ for every $y>0$. In order to have $z(x,y)=e^{rx}e^{\pm ary}$ as solution verifying these boundary conditions we must have the following: $r=\alpha+i\beta$
and then
$$\begin{array}{l c l} z(x,y)&=&e^{rx}e^{\pm ary}=e^{\alpha x\pm a\alpha y}e^{i[\beta x\pm\beta ay]}\\&&\\&=&e^{\alpha x\pm a\alpha y}\big[\cos(\beta x\pm\beta ay)+i\sin(\beta x\pm\beta ay)\big]
 \end{array}$$
 Choosing $\alpha=0$ we obtain:
$$\begin{array}{l c l} z(x,y)&=&\cos(\beta x\pm\beta ay)+i\sin(\beta x\pm\beta ay)\\&&\\
&=& \cos(\beta x)\cos(\beta a y)\mp\sin(\beta x)\sin(\beta a y)+i[\sin(\beta x)\cos(\beta a y)\pm\cos(\beta x)\sin(\beta a y)]
 \end{array}$$

In this case we may consider the two parcels:

The first
$$u(x,y)=\sin(\beta x)\sin(\beta a y)$$
observe that this parcel verifies as boundary conditions if
$$0=u(L,y)=\sin(\beta L)\sin(\beta a y)$$for every $y>0$. This implies
$$\sin(\beta L)=0$$
and then $$\beta L=n\pi,\;\;n=1,2,3,\ldots$$
therefore the solutions are of the form
$$u_n(x,y)=\sin\big(\frac{n\pi x}{L}\big)\sin\big(\frac{n\pi a y}{L}\big).$$
The second
$$v(x,y)=\sin(\beta x)\cos(\beta a y)$$
observe that this parcel verifies as boundary conditions if
$$0=v(L,y)=\sin(\beta L)\cos(\beta a y)$$for every $y>0$. This implies
$$\sin(\beta L)=0$$
and then $$\beta L=n\pi,\;\;n=1,2,3,\ldots$$
therefore the solutions are of the form
$$v_n(x,y)=\sin\big(\frac{n\pi x}{L}\big)\cos\big(\frac{n\pi a y}{L}\big).$$
Once again, these are the solutions obtained by the separation of variables method.

\subsection{Laplace equation}
 Consider
\begin{equation}\label{eqlaplace}
     \frac{\partial^2 z}{\partial x^2}+\frac{\partial^2 z}{\partial y^2}=0,\;x>0,y>0
 \end{equation}
Equation (\ref{eqlaplace}) is of the form (\ref{eq14}) with $A=1$, $C=1$ and $B=D=E=F=0$. According to Proposition~\ref{Proposition:Euler},  equation (\ref{eqlaplace}) is transformed, by putting $x=\ln u, y=\ln v$, in
\begin{equation}\label{eqlaplace2}
u^2\frac{\partial^2 \tilde{z}}{\partial u^2}+v^2 \frac{\partial^2 \tilde{z}}{\partial v^2}+u\frac{\partial \tilde{z}}{\partial u}+v\frac{\partial \tilde{z}}{\partial v}=0.
\end{equation}
This equation has indicial conic of the form $r^2+s^2=0$ hence the solutions of (\ref{eqlaplace2}) are of the form
$$u^rv^{\pm ir}$$for any $r$. Back to the original variables, the solutions of (\ref{eqlaplace}) are of the form

$$e^{rx}e^{\pm iry}$$for any $r$.
\subsubsection{Boundary conditions:} If in equation (\ref{eqlaplace}) we consider boundary conditions of type $z(0,y)=0$, $z(L,y)=0$ for every $y>0$ and $z(x,0)=0$ for every $x>0$. For  $z(x,y)=e^{rx}e^{\pm iry}$ to be a solution verifying these boundary conditions we must have the following: $r=\alpha+i\beta$

$$\begin{array}{l c l} z(x,y)&=&e^{rx}e^{\pm iry}=e^{\alpha x\mp \beta y}e^{i[\beta x\pm\alpha y]}\\&&\\&=&e^{\alpha x\mp \beta y}\big[\cos(\beta x\pm\alpha y)+i\sin(\beta x\pm\alpha y)\big]
 \end{array}$$
 Choosing $\alpha=0$ we obtain:
$$z(x,y)=e^{\mp \beta y}\cos(\beta x)+ie^{\mp \beta y}\sin(\beta x)$$
In this case we can consider two parcels:

The first
$$u(x,y)=e^{\beta y}\sin(\beta x)$$
observe that this parcel verifies as boundary conditions if
$$0=u(L,y)=e^{\beta y}\sin(\beta L)$$for every $y>0$. This implies
$$\sin(\beta L)=0$$
and then $$\beta L=n\pi,\;\;n=1,2,3,\ldots$$
therefore the solutions are of the form
$$u_n(x,y)=e^{\frac{n\pi y}{L}}\sin\big(\frac{n\pi x}{L}\big).$$
The second
$$v(x,y)=e^{-\beta y}\sin(\beta x)$$
observe that this parcel verifies as boundary conditions if
$$0=u(L,y)=e^{-\beta y}\sin(\beta L)$$for every $y>0$. This implies
$$\sin(\beta L)=0$$
and then $$\beta L=n\pi,\;\;n=1,2,3,\ldots$$
therefore the solutions are of the form
$$v_n(x,y)=e^{-\frac{n\pi y}{L}}\sin\big(\frac{n\pi x}{L}\big).$$
Notice that, putting

$$w_n(x,y)=\frac{u_n(x,y)-v_n(x,y)}{2}=\sinh\big(\frac{n\pi y}{L}\big)\sin\big(\frac{n\pi x}{L}\big).$$

\section{Rational Euler equations: existence of monomial meromorphic  solutions}

Let us now investigate the existence of meromorphic or rational solutions for a given Euler equation with rational numbers are coefficients. This may be useful if we want to obtain monomial solutions which
have analytic behavior, possibly admitting poles.
 As we have already observe,
the problem of finding monomial holomorphic or meromorphic solutions of a given Euler equation  is related to
the existence of  integral  or rational points of the indicial conic. This last is related to the existence of solutions of Diophantine equations.
Let us  present families of examples, of each type for the indicial conic type above, exhibiting rational solutions. Such families are constructed based on the existence results from \cite{Hua} and \cite{Matthews}.
\begin{Theorem}[existence of meromorphic monomial solutions]
\label{Theorem:meromorphicmonomialsol}
The following families of Euler PDEs with integer coefficients admit meromorphic monomial solutions:
\begin{itemize}
\item[{\rm (elliptic)}]
\begin{equation}\label{eq34}
\mathcal Ell: \, A^2x^2\frac{\partial^2 z}{\partial x^2}+C^2y^2\frac{\partial^2 z}{\partial y^2}+3A^2x\frac{\partial z}{\partial x}+3C^2y\frac{\partial z}{\partial y}+(C^2+A^2-A^2C^2)z=0,\;x>0,y>0
\end{equation}
where $A,C\in\mathbb{Z}$ such that $AC>0$.

\item[{\rm(parabolic)}]
\begin{equation}\label{eq37}
\mathcal Par: \,
Ax^2\frac{\partial^2 z}{\partial x^2}+Bxy\frac{\partial^2 z}{\partial x\partial y}+Cy^2\frac{\partial^2 z}{\partial y^2}+3Ax\frac{\partial z}{\partial x}+(C+B)y\frac{\partial z}{\partial y}+Az=0,\;x>0,y>0
\end{equation}
where $A,B,C\in\mathbb{Z}$ such that $B^2-4AC=0$.

\item[{\rm(hyperbolic)}]
\begin{equation}\label{eq40}
\mathcal Hyp: \,
Ax^2\frac{\partial^2 z}{\partial x^2}+Bxy\frac{\partial^2 z}{\partial x\partial y}+Ax\frac{\partial z}{\partial x}+By\frac{\partial z}{\partial y}-Az=0,\;x>0,y>0
\end{equation}
where $A,B\in\mathbb{Z}$ such that $B\neq0$.
\end{itemize}

 Indeed, on each case, there are integral points of the corresponding indicial conic.

\end{Theorem}
\begin{proof}
Let us start with the first elliptic case.
The corresponding indicial conic is

\begin{equation}\label{eq35}
A^2r^2+C^2s^2+2A^2r+2C^2s+C^2+A^2-A^2C^2=0.
\end{equation}

Completing the square in equation (\ref{eq35}) we have
$$A^2[r^2+2r+1]+C^2[s^2+2s+1]=A^2C^2 $$
$$A^2(r+1)^2+C^2(s+1)^2=A^2C^2 $$
\begin{equation}\label{eq36}
 (Ar+A)^2+(Cs+C)^2=(AC)^2.
\end{equation}
Equation (\ref{eq36}) admits entire solutions in the following cases:
$$ Ar+A=\pm AC \;\;\;\mbox{ and }\;\;\;Cs+C=0$$ or
$$ Ar+A=0 \;\;\;\mbox{ and }\;\;\;Cs+C=\pm AC.$$
 In these cases $r$ and $s$ have integral values:
$$ r=-1\pm C \;\;\;\mbox{ and }\;\;\;s=-1$$ or
$$ r=-1 \;\;\;\mbox{ and }\;\;\;s=-1\pm A.$$

Now, for the second (parabolic) case the indicial conic is given by
\begin{equation}\label{eq38}
Ar^2+Brs+Cs^2+2Ar+Bs+A=0.
\end{equation}

Multiplying equation (\ref{eq38}) by $4A$ and using that $B^2=4AC$ we have
$$4A^2r^2+4ABrs+4ACs^2+8A^2r+4ABs+4A^2=0 $$
$$4A^2r^2+4ABrs+B^2s^2+8A^2r+4ABs+4A^2=0 $$
$$(2Ar+Bs)^2+4A(2Ar+Bs)+4A^2=0$$
$$(2Ar+Bs+2A)^2=0$$
\begin{equation}\label{eq39}
2Ar+Bs=-2A.
\end{equation}
Let $h=\mbox{gcd}\{2A,B\}$, since $h$ divides $-2A$ we have that equation (\ref{eq39}) admits entire solutions.

Finally we consider the hyperbolic case. In this case we have the indicial conic given by
\begin{equation}\label{eq41}
Ar^2+Brs+Bs-A=0.
\end{equation}

Multiplying equation (\ref{eq41}) by $B^4$ we have
$$AB^4r^2+B^5rs+B^5s-B^4A=0 $$
$$A(B^2r)^2+B(B^2r)(B^2s)+B^3(B^2s)-B^4A=0$$
by putting
$$B^2r=R-B^2\;\;\;\;\mbox{ and }\;\;\;\;B^2s=S+2AB$$
 we have
$$A(R-B^2)^2+B(R-B^2)(S+2AB)+B^3(S+2AB)-B^4A=0$$
$$A(R^2-2B^2R+B^4)+B(RS-B^2S+2RAB-2AB^3)+B^3S+2AB^4-B^4A=0$$
$$AR^2-2AB^2R+AB^4+BRS-B^3S+2RAB^2-2AB^4+B^3S+2AB^4-B^4A=0$$
$$AR^2+BRS=0$$
\begin{equation}\label{eq42}
R(AR+Bs)=0.
\end{equation}
Equation (\ref{eq42}) admits entire solutions in the following cases:
$$ R=0 \;\;\;\mbox{ or }\;\;\;AR+BS=0$$
back to the original variables
$$ B^2r+B^2=0 \;\;\;\mbox{ or }\;\;\;A(B^2r+B^2)+B(B^2s-2AB)=0$$
 then $$ r=-1,\;s\mbox{any integer}$$
 or
 $$AB^2r+AB^2+B^3s-2AB^2=0$$
 $$AB^2r+B^3s=AB^2$$
\begin{equation}\label{eq43}
Ar+Bs=A,
\end{equation}
be a $h=\mbox{gcd}\{A,B\}$, since $h$ divides $A$ we have that equation (\ref{eq43}) admits entire solutions.

\end{proof}

\section{Frobenius method for second order linear PDEs}
In this section we pave the way to our main results Theorems~\ref{Theorem:paraboliccomplex}, \ref{Theorem:ellipticcomplex} and \ref{Theorem:hyperboliccomplex}, as well as their real counterparts
Theorems~\ref{Theorem:realparabolic}, \ref{Theorem:realelliptic} and \ref{Theorem:realhyperbolic}.
Let us consider a second order linear homogeneous PDE of the form
\[
A(x,y)\frac{\partial^2 z}{\partial x^2}+B(x,y)\frac{\partial^2 z}{\partial x\partial y}+C(x,y)\frac{\partial^2 z}{\partial y^2}+D(x,y)\frac{\partial z}{\partial x}+E(x,y)\frac{\partial z}{\partial y}+F(x,y)z=0,
\]
where the coefficients $A,B,C,D$ and $E$ are analytic functions at some point $(x_0,y_0)\in \mathbb K^2$.

\subsection{Regular singularities}
We recall the following definition from the Introduction:

\begin{Definition}[regular singularity]
{\rm The point $(x_0,y_0)$ is called {\it ordinary} if some of the coefficients $A, B$ and  $C$ does not vanish at $(x_0,y_0)$.
Otherwise, if $A(x_0,y_0)=B(x_0,y_0)=C(x_0,y_0)=0$,  it will be called a {\it singular point}. A singular point will be called a {\it regular singularity} if the PDE can be put into the form
\[
Ax^2\frac{\partial^2 z}{\partial x^2}+Bxy\frac{\partial^2 z}{\partial x\partial y}+Cy^2\frac{\partial^2 z}{\partial y^2}+xa(x,y)\frac{\partial z}{\partial x}+yb(x,y)\frac{\partial z}{\partial y}+c(x,y)z=0
\]
with $A, B, C\in \mathbb R$ constants, $a(x,y), b(x,y), c(x,y)$ analytic functions,
after some division by coefficients and change of coordinates centered at $(x_0,y_0)$.

}
\end{Definition}

We look initially for solutions of the form $\psi=x^r y^s + h.o.t.$. We write $a(x,y)=a(0,0) + x^ny^m a_1(x,y)$ where
$a_1(x,y)$ is analytic and $n+ m \geq 1$. Similarly we write $b(x,y) = b(0,0) + x^\ell + y^k b_1(x,y)$ and
$c(x,y)= c(0,0) + x^t y^s c_1(x,y)$.

Then we can rewrite the PDE as
\[
Ax^2\frac{\partial^2 z}{\partial x^2}+Bxy\frac{\partial^2 z}{\partial x\partial y}+Cy^2\frac{\partial^2 z}{\partial y^2}+xa(0,0) \frac{\partial z}{\partial x}+yb(0,0)\frac{\partial z}{\partial y}+c(0,0)z
\]
\[
+x^{n+1} y^ m a_1(x,y)\frac{\partial z}{\partial x}+x^\ell y^{k+1}b_1(x,y)\frac{\partial z}{\partial y}+x^t y^s c_1(x,y)z=0
\]

Thus the PDE has a first jet which is a PDE of Euler type given by
\[
Ax^2\frac{\partial^2 z}{\partial x^2}+Bxy\frac{\partial^2 z}{\partial x\partial y}+Cy^2\frac{\partial^2 z}{\partial y^2}+xa(0,0) \frac{\partial z}{\partial x}+yb(0,0)\frac{\partial z}{\partial y}+c(0,0)z=0.
\]

The higher order part is seen as an increment and the original PDE as a deformation of the Euler PDE above.

Based on this we define:

\begin{Definition}
{\rm

The {\it indicial conic}  is the complex affine curve  $\mathcal C\subset \mathbb C^2$  given by the zeros of the polynomial
$$P(r,s)=Ar^2+Brs+Cs^2+r(a(0,0)-A)+s(b(0,0)-C)+c(0,0).$$
In case the PDE has real coefficients we can define the {\it real trace of the indicial conic} (or {\it real indicial conic} for short) as $\mathcal C(\mathbb R):=\mathcal C \cap \mathbb R^2$.

}
\end{Definition}

The above definition of regular singularity is motivated by the following

\begin{Lemma}
Let
\[
A(x,y)x^2\frac{\partial^2 z}{\partial x^2}+2B(x,y) xy\frac{\partial^2 z}{\partial x\partial y}+C(x,y) y^2\frac{\partial^2 z}{\partial y^2}+xa(x,y)\frac{\partial z}{\partial x}+yb(x,y)\frac{\partial z}{\partial y}+c(x,y)z=0
\]
be given with $A(x,y), B(x,y), C(x,y), a(x,y), b(x,y), c(x,y)$ of class $C^r$ in  a neighborhood of the origin. For any change of coordinates $(x,y)\mapsto(\tilde x,\tilde y)$, of class $C^r$, that preserves the coordinate axes,
the new PDE is still of the above form. More precisely, if we put $(\tilde x,\tilde y)=(x.\xi_1(x,y), y.\eta_1(x,y))$, then in the new coordinates we have
\[
\tilde A(\tilde x,\tilde y)\tilde x^2\frac{\partial^2 z}{\partial \tilde x^2}+2\tilde B(\tilde x,\tilde y) \tilde x\tilde y\frac{\partial^2 z}{\partial \tilde x\partial \tilde y}+\tilde C (\tilde x,\tilde y)\tilde y^2\frac{\partial^2 z}{\partial \tilde y^2}+ \tilde x \tilde a(\tilde x, \tilde y)\frac{\partial z}{\partial \tilde x}+\tilde y \tilde b(\tilde x, \tilde y)\frac{\partial z}{\partial \tilde y}+\tilde c(\tilde x,\tilde y)z=0
\]

for some $\tilde A(\tilde x,\tilde y), \tilde B(\tilde x,\tilde y), \tilde C(\tilde x,\tilde y), \tilde a(\tilde x,\tilde y), \tilde b(\tilde x,\tilde y), \tilde c(\tilde x,\tilde y)$ of class $C^r$ at the origin.

Moreover, we have 
$$\begin{pmatrix}
\tilde A \tilde x^2 & \tilde B \tilde x \tilde y\\
\tilde B \tilde x \tilde y& \tilde C \tilde y^2
\end{pmatrix} = J^t \begin{pmatrix} A x^2 & Bxy \\ Bxy & C y^2 \end{pmatrix} J
$$
where $J$ is the Jacobian matrix of the change of coordinates. In particular, 
$(\tilde B^2 - \tilde A\tilde C)(\tilde x \tilde y)^2 = |\det J|^2 ( B^2 -   A  C)(xy)^2$.
\end{Lemma}

\begin{proof}
We start with a PDE of the form
\[
a(x,y)\frac{\partial^2 u}{\partial x^2}+2b(x,y)\frac{\partial^2 u}{\partial x\partial y}+c(x,y)\frac{\partial^2 u}{\partial y^2}+d_1(x,y)\frac{\partial z}{\partial x}+d_2(x,y)\frac{\partial z}{\partial y}+d_3(x,y)z=f(x,y)
\]
and perform a change of coordinates
$(x,y)\mapsto (\xi(x,y), \eta(x,y))$.
Then we obtain a PDE of the form
\[
A(\xi,\eta)\frac{\partial^2 u}{\partial \xi^2}+2B(\xi,\eta)\frac{\partial^2 u}{\partial \xi\partial \eta}+C(\xi,\eta)\frac{\partial^2 u}{\partial \eta ^2}+e_1(\xi,\eta)\frac{\partial z}{\partial \xi}+e_2(\xi,\eta)\frac{\partial z}{\partial \eta}+e_3(\xi,\eta)z=g(\xi,\eta)
\]
where we have $A=a\xi_ x^2 + 2 b\xi_x \xi _y + c \xi_y ^2, \, B=a \xi_x \eta_x + b(\eta_x \xi_y + \eta_y \xi _x) + c \xi_y \eta_y, \, C= a \eta_x ^2 + 2b \eta_x \eta_y + c \eta_y ^2$. In terms of matrices we have
$$\begin{pmatrix}
A & B \\
B & C
\end{pmatrix} = J^t \begin{pmatrix} a & b \\ b & c \end{pmatrix} J
$$
where $J=\begin{pmatrix} \xi_x & \xi _y \\ \eta_x & \eta_y\end{pmatrix}$ is the jacobian matrix of the change of coordinates.

\end{proof}

\begin{Remark}[Behavior of solutions in straight lines]
\label{Remark:restriction}
{\rm Consider the following PDE
\begin{equation}\label{eq44}
Ax^2\frac{\partial^2 z}{\partial x^2}+Bxy\frac{\partial^2 z}{\partial x\partial y}+Cy^2\frac{\partial^2 z}{\partial y^2}+Dx\frac{\partial z}{\partial x}+Ey\frac{\partial z}{\partial y}+Fz=0,
\end{equation}
where $A,B,C\in\mathbb{C}^*,D,E,F\in\mathbb{C}$. Let $u$ be a solution of (\ref{eq44}). Hence we have
\begin{equation}\label{eq45}
Ax^2\frac{\partial^2 u}{\partial x^2}+Bxy\frac{\partial^2 u}{\partial x\partial y}+Cy^2\frac{\partial^2 u}{\partial y^2}+Dx\frac{\partial u}{\partial x}+Ey\frac{\partial u}{\partial y}+Fu=0.
\end{equation}
Consider the line $y=tx$. We want to understand the behavior of the solutions restricted to this line.
For this sake we evaluate $(x,tx)$ in equation (\ref{eq45})
$$Ax^2\frac{\partial^2 u}{\partial x^2}(x,tx)+Btx^2\frac{\partial^2 u}{\partial x\partial y}(x,tx)+Ct^2x^2\frac{\partial^2 u}{\partial y^2}(x,tx)+Dx\frac{\partial u}{\partial x}(x,tx)+Etx\frac{\partial u}{\partial y}(x,tx)+Fu(x,tx)=0.$$
We define $\tilde{u}(x)=u(x,tx)$, taking derivatives we have

$$\tilde{u}^\prime(x)=\frac{\partial u}{\partial x}(x,tx)+t\frac{\partial u}{\partial y}(x,tx)$$

$$\tilde{u}^{\prime\prime}(x)=\frac{\partial^2 u}{\partial x^2}(x,tx)+ 2t\frac{\partial^2 u}{\partial x \partial y}(x,tx)+t^2\frac{\partial^2 u}{\partial y^2}(x,tx).$$
Thus observe that $\tilde{u}$ is solution of the second order Euler equation
$$x^2z^{\prime\prime}+x\big(\frac{D}{A}\big)z^\prime+\big(\frac{F}{A}\big)z=0$$se $A=C=\frac{B}{2}$ and $D=E$. Observe that in this case equation (\ref{eq44}) would be of parabolic type.

}
\end{Remark}

\subsection{Resonances}

We shall now introduce the notion of resonance for a second order linear PDE with a regular singular point. For this sake, we shall consider the second order PDE
\begin{equation}\label{eq16}
({\mathcal E})\, \, \,   L[z]:=Ax^2\frac{\partial^2 z}{\partial x^2}+Bxy\frac{\partial^2 z}{\partial x\partial y}+Cy^2\frac{\partial^2 z}{\partial y^2}+xa(x,y)\frac{\partial z}{\partial x}+yb(x,y)\frac{\partial z}{\partial y}+c(x,y)z=0
\end{equation}
where $A,B, C\in\mathbb{R}, \, \mathbb C$, $a(x,y),b(x,y)$ and $c(x,y)$ real analytic or holomorphic in $\Delta[(0,0),(R,R)]\subset\mathbb R^2 , \, \mathbb C^2$, $R>0$.

The indicial conic $\mathcal C\subset  \mathbb C^2$ is then given by the zeros of the polynomial
$$P(r,s)=Ar^2+Brs+Cs^2+r(a(0,0)-A)+s(b(0,0)-C)+c(0,0).$$

\begin{Definition}
{\rm
A pair $(r,s)\in \mathcal C$ is
called {\it resonant} (with respect to the PDE $\mathcal E$) if there is $Q=(q_1,q_2)\in \mathbb Z_+^2\setminus \{ (0,0)\}$ such that
\[
P(q_1+r,q_2+s)=0.
\]
Since $P(r,s)=0$ this is equivalent to
\[
Aq_1(2r + q_1 -1) + B(q_1 s + rq_2 + q_1 q_2) + Cq_2(2s + q_2-1) +q_1 a(0,0) + q_2 b(0,0)=0.
\]
This last is a degree one equation in $r$ and $s$.
Let us introduce the {\it set of resonant points of} $\mathcal E$ as $\mathcal R(\mathcal E):=\mathcal R \cap \mathcal C$ where
$$\mathcal{R}=\bigcup\limits_{|Q|=1}^\infty R_Q$$ and
$$R_Q:= \lbrace(r,s)\in\mathcal{C};\;P(q_1+r,q_2+s)=0\rbrace.$$

A point $(r_0,s_0)\in \mathcal C$ is {\it non-resonant} if $(r_0,s_0)\not\in\mathcal R \cap \mathcal C$.
}
\end{Definition}

\begin{Lemma}
$\mathcal R(\mathcal E)$ is nowhere dense in $\mathcal C$ and in $\mathbb R^2, \mathbb C^2$.
In particular, the set of non-resonant points is dense in $\mathcal C$.
\end{Lemma}
\begin{proof}
The set of resonant points of $\mathcal E$ is the intersection of the indicial conic with
a countable number of straight lines in $\mathbb R^2, \mathbb C^2$. By Baire's category theorem this set has empty interior. It also has zero Lebesgue measure and contains no interior points even when looked inside $\mathcal C C$.
\end{proof}

\begin{Remark}
{\rm Let us point-out a few facts about the notion of resonance we have introduced above:

\begin{enumerate}
\item Geometrically, a point $(r_0,s_0)\in \mathcal C$ is resonant if there is some non-trivial translation of $(r_0,s_0)$ by
integral shift $(r_0 +q_1, s_0 +q_2), \, q_1, q_2\in \mathbb N$ which also lies on the indicial conic.
This can be seen as follows: consider the reticulate $r_0\vee s_0 \subset\mathbb C^2$ centered at $(r_0,q_0)$. This means   the set of all points of the form $ (r_0 +q_1,s_0 +q_2)$ where $q_1,q_2\in \mathbb Z$. The positive part of the reticulate is the set of points of the form $ (r_0 +q_1,s_0 +q_2)$ where $q_1,q_2\in \mathbb N\cup\{0\}$.  Then, a point $(r_0,s_0)\in \mathcal C$ of the indicial conic is resonant if there is some vertex of the positive part of the reticulate that lies over the indicial conic.

\item Consider a  second order linear analytic ODE of the form
$a x^2 u^{\prime \prime} + xb(x) u^\prime + c(x)u=0$, with a regular singularity.
From the classical method of Frobenius, we know that the indicial equation is of the form
$P(r)=a r (r-1) + b(0)r + c(0)=0$. If we choose the root $r_0$ with greater real part then there is a solution of the form $u(x)=x^{r_0} \sum\limits_{q=0}^\infty a_q x^q$. We looking for other solutions,  the exceptional case occurs when  there is another root $r_1$ of the indicial equation which is of the form
$r_1=r_0 -q_0$ for some $q_0 \in \mathbb N$. This means that the indicial polynomial $P(r)$ has a root $r_1$ and another
root of the form $r_1+q$. This is exactly the notion of resonance we have just introduced above for the case of PDEs.

\end{enumerate}
}
\end{Remark}

\subsection{Frobenius type solutions}
We consider
a second order linear homogeneous PDE of the form
\[
A(x,y)\frac{\partial^2 z}{\partial x^2}+B(x,y)\frac{\partial^2 z}{\partial x\partial y}+C(x,y)\frac{\partial^2 z}{\partial y^2}+D(x,y)\frac{\partial z}{\partial x}+E(x,y)\frac{\partial z}{\partial y}+F(x,y)z=0,
\]
where the coefficients $A,B,C,D$ and $E$ are analytic functions at some point $(x_0,y_0)\in \mathbb R^2$. Assume that $(x_0,y_0)$ is a non-singular point or a regular singularity of the PDE.
Let be given $(r_0,s_0)\in \mathbb R^2$  a point of the indicial conic.

\begin{Definition}
{\rm A {\it  Frobenius type solution} of the PDE above is an expression
$$\psi(x,y)=x^{r_0}y^{s_0}\sum^\infty_{|Q|=0}d_Qx^{q_1}y^{q_2},\;\;\;\;d_{0,0}=1,$$
which satisfies the PDE from the formal point of view, ie., $L[\psi](x,y)=0$ where
$L[z]=A(x,y)\frac{\partial^2 z}{\partial x^2}+B(x,y)\frac{\partial^2 z}{\partial x\partial y}+C(x,y)\frac{\partial^2 z}{\partial y^2}+D(x,y)\frac{\partial z}{\partial x}+E(x,y)\frac{\partial z}{\partial y}+F(x,y)z$. Moreover, the coefficients $D_Q$ are obtained by means of recurrences associate to the equation.
 The solution is called of {\it convergent type} if  the series $\sum^\infty_{|Q|=0}d_Qx^{q_1}y^{q_2},\;\;\;\;d_{0,0}=1$ is convergent in the bidisc $\Delta[(0,0),(R,R)]$. We shall say that the solution is {\it real} if the exponent $(r,s)\in \mathbb R^2$ and coefficients
$d_Q$ of the power series are all real.
}
\end{Definition}

\begin{Example}{\rm Consider the second order PDE given by
\begin{equation}\label{eq46}
x^2\frac{\partial^2 z}{\partial x^2}+2xy\frac{\partial^2 z}{\partial x\partial y}+y^2\frac{\partial^2 z}{\partial y^2}-\frac{\partial z}{\partial x}-\frac{\partial z}{\partial y}-\frac{1}{2}z=0
\end{equation}
The origin is not a regular singularity for  (\ref{eq46}). Let us check for the existence of
formal solutions. This equation admits as solution  a formal power  series
\begin{equation}\label{eq47}
\sum^{\infty}_{|Q|=0} a_QX^Q,
\end{equation}
where the coefficients $a_Q$ satisfy the following recurrence formula
\begin{equation}\label{eq48}
C_Q=\big[|Q|^2-|Q|-\frac{1}{2}\big]a_Q,\;\;\;\mbox{ for every }|Q|=0,1,2,\ldots,
\end{equation}
where $C_{q_1,q_2}=(1+q_1)a_{1+q_1,q_2}+(1+q_2)a_{q_1,1+q_2}$. But is not convergent, since  if we consider
\begin{equation}\label{eq50}
\gamma_n=\sum^\infty_{n=0}a_nx^n
\end{equation}
where $a_n=\sum_{|Q|=n}a_Q$. We see that by (\ref{eq48}) the coefficients $a_n$ satisfy the following recorrência
\begin{equation}\label{eq49}
\big(n^2-n-\frac{1}{2}\big)a_n=(n+1)a_{n+1},\;\;\;\mbox{ for every }n=0,1,2,\ldots,
\end{equation}
 If $a_{00}\neq0$ then $a_0\neq0$ and then by the quotient test to  (\ref{eq49}) and (\ref{eq50}), we have that
$$ \left|\frac{a_{n+1}x^{n+1}}{a_nx^n}\right|=\left|\frac{n^2-n-\frac{1}{2}}{n+1}\right|\cdot|x|\to\infty,$$
as $k\to\infty$, if $|x|\neq0$. Thus, the series converges only for $x=0$, and therefore
 (\ref{eq47}) does not define a  convergent solution of (\ref{eq46}).}
\end{Example}

\subsection{Pre-canonical forms}

We start with a PDE of the form
\[
a(x,y)\frac{\partial^2 u}{\partial x^2}+2b(x,y)\frac{\partial^2 u}{\partial x\partial y}+c(x,y)\frac{\partial^2 u}{\partial y^2}+d_1(x,y)\frac{\partial z}{\partial x}+d_2(x,y)\frac{\partial z}{\partial y}+d_3(x,y)z=f(x,y)
\]
and perform a change of coordinates
$(x,y)\mapsto (\xi(x,y), \eta(x,y))$.
Then we obtain a PDE of the form
\[
A(\xi,\eta)\frac{\partial^2 u}{\partial \xi^2}+2B(\xi,\eta)\frac{\partial^2 u}{\partial \xi\partial \eta}+C(\xi,\eta)\frac{\partial^2 u}{\partial \eta ^2}+e_1(\xi,\eta)\frac{\partial z}{\partial \xi}+e_2(\xi,\eta)\frac{\partial z}{\partial \eta}+e_3(\xi,\eta)z=g(\xi,\eta)
\]
where we have $A=a\xi_ x^2 + 2 b\xi_x \xi _y + c \xi_y ^2, \, B=a \xi_x \eta_x + b(\eta_x \xi_y + \eta_y \xi _x) + c \xi_y \eta_y, \, C= a \eta_x ^2 + 2b \eta_x \eta_y + c \eta_y ^2$. In terms of matrices we have
$$\begin{pmatrix}
A & B \\
B & C
\end{pmatrix} = J^t \begin{pmatrix} a & b \\ b & c \end{pmatrix} J
$$
where $J=\begin{pmatrix} \xi_x & \xi _y \\ \eta_x & \eta_y\end{pmatrix}$ is the jacobian matrix of the change of coordinates.

\noindent{\bf Parabolic case}. In this case we have $ac=b^2$ and therefore we can write
$A=(p\xi_x + q \xi_y)^2$ and $C=(p\eta_x + q \eta_y)^2$ for some functions $p(x,y), q(x,y)$. Let now $\xi$ be a solution of the equation $p\xi_x + q \xi_y=0$. Then we have $A=0$ and, since $AC = B^2$ we obtain $B=0$.
This gives the following {\it parabolic canonical form}, after re-renaming the variables $(\xi, \eta)$ as $(x,y)$:

\[
C(x,y)\frac{\partial^2 u}{\partial y ^2}+e_1(x,y)\frac{\partial z}{\partial x}+e_2(x,y)\frac{\partial z}{\partial y}+e_3(x,y)z=g(x,y).
\]

 \noindent{\bf Hyperbolic case}. In this case $ac<b^2$. We can split in factors $A=(p_1 \xi_x + q_1 \xi_y)
 (p_2 \xi_x + q_2 \xi_y)$ and $C=(p_1 \eta_x + q_1 \eta_y)(p_2 \eta_x + q_2 \eta_y)$. Then we choose $\xi$ and $\eta$ satisfying $p_1 \xi_x + q_1 \xi_y= p_2 \eta_x + q_2 \eta_y=0$. Then we obtain $A=C=0$.
 In this case, after re-renaming the coordinates back to $(x,y)$ we obtain the following
 {\it hyperbolic canonical form}:

 \[
2B(x,y)\frac{\partial^2 u}{\partial x\partial y}+e_1(x,y)\frac{\partial z}{\partial x}+e_2(x,y)\frac{\partial z}{\partial y}+e_3(x,y)z=g(x,y)
\]

\noindent{\bf Elliptic case}. In this case $ac> b^2$. We cannot not make $A$ or $C$ equal to zero, but we can
obtain $A=C$ and $B=0$ so that the {\it elliptic canonical form} is:

\[
A(x,y)\frac{\partial^2 u}{\partial x^2}+A(x,y)\frac{\partial^2 u}{\partial y ^2}+e_1(x,y)\frac{\partial z}{\partial x}+e_2(x,y)\frac{\partial z}{\partial y}+e_3(x,y)z=g(x,y)
\]

We point-out that, usually the canonical forms are a bit simpler. Indeed, in the parabolic case
it is observed that $C$ does not vanish and therefore the equation may be divided by $C$ and we ay assume that $C=1$. Similarly, in hyperbolic case it is assumed that $B=1$ and in the elliptic case that we have $A=C=1$.
Nevertheless, we shall not work with these simpler canonical forms. Indeed, we will not be considering exactly the same situation. We shall use the above "pre-canonical forms" adapted to our framework.

\subsubsection{Preparation lemma}

\begin{Lemma}
Consider the following PDE
\begin{equation}\label{eqlem1}
A(x)x^2\frac{\partial^2 z}{\partial x^2}+B(x,y)xy\frac{\partial^2 z}{\partial x\partial y}+C(y)y^2\frac{\partial^2 z}{\partial y^2}+D(x,y)x\frac{\partial z}{\partial x}+E(x,y)y\frac{\partial z}{\partial y}+F(x,y)z=0,
\end{equation}
where $A,C$ holomorphic in $\Delta[0,R]$ and $B,D,E,F$ holomorphic in $\Delta[(0,0),(R,R)]$, $R>0$. Assume that $B(x,y)=2\sqrt{A(x)}\sqrt{C(y)}$ then there suitable change of coordinates so that the equation (\ref{eqlem1}) has the form
\begin{equation}\label{eqlem2}
A(0)\xi^2\frac{\partial^2 \tilde{z}}{\partial \xi^2}+B(0,0)\xi\eta\frac{\partial^2 \tilde
{z}}{\partial \xi\partial \eta}+C(0)\eta^2\frac{\partial^2\tilde{z}}{\partial \eta^2}+\tilde{D}(\xi,\eta)\xi\frac{\partial \tilde{z}}{\partial \xi}+\tilde{E}(\xi,\eta)\eta\frac{\partial \tilde{z}}{\partial \eta}+\tilde{F}(\xi,\eta)\tilde{z}=0,
\end{equation}
where $\tilde{D},\tilde{E},\tilde{F}$ holomorphic in $\Delta[(0,0),(\tilde{R},\tilde{R})],$ $\tilde{R}>0$.
\end{Lemma}
\begin{proof} Let us consider a change of coordinates of the form  $\xi=x f(x)$ and $\eta=yg(y)$.  Let $\tilde{u}(\xi,\eta)=u(x,y)$, where $u$ is solution of (\ref{eqlem1}) thus we have

$$\frac{\partial u}{\partial x}=(f(x)+xf^\prime(x))\frac{\partial \tilde{u}}{\partial \xi}$$
$$\frac{\partial u}{\partial y}=(g(y)+yg^\prime(y))\frac{\partial \tilde{u}}{\partial \eta}$$
$$\frac{\partial^2 u}{\partial x^2}=(2f^\prime(x)+xf^{\prime\prime}(x))\frac{\partial \tilde{u}}{\partial \xi}+(f(x)+xf^\prime(x))^2\frac{\partial^2 \tilde{u}}{\partial \xi^2}$$
$$\frac{\partial^2 u}{\partial x \partial y}=(f(x)+xf^\prime(x))(g(y)+yg^\prime(y))\frac{\partial \tilde{u}}{\partial \xi \partial\eta}$$
$$\frac{\partial^2 u}{\partial y^2}=(2g^\prime(y)+yg^{\prime\prime}(y))\frac{\partial \tilde{u}}{\partial \eta}+(g(y)+yg^\prime(y))^2\frac{\partial^2 \tilde{u}}{\partial \eta^2}.$$
Since $u$ is solution of (\ref{eqlem1}) we have
$$A(x)x^2\frac{\partial^2 u}{\partial x^2}+B(x,y)xy\frac{\partial^2 u}{\partial x\partial y}+C(y)y^2\frac{\partial^2 u}{\partial y^2}+D(x,y)x\frac{\partial u}{\partial x}+E(x,y)y\frac{\partial u}{\partial y}+F(x,y)u=0
$$
therefore
$$\begin{array}{c}A(x)x^2(f(x)+xf^\prime(x))^2\frac{\partial^2 \tilde{u}}{\partial \xi^2}+B(x,y)xy(f(x)+xf^\prime(x))(g(y)+yg^\prime(y))\frac{\partial \tilde{u}}{\partial \xi \partial \eta}\\
\\
+C(y)y^2(g(y)+yg^\prime(y))^2\frac{\partial^2 \tilde{u}}{\partial \eta^2}+\big(A(x)x^2(2f^\prime(x)+xf^{\prime\prime}(x))+D(x,y)x(f(x)+xf^\prime(x))\big)\frac{\partial \tilde{u}}{\partial \xi}\\
\\
+\big(C(y)y^2(2g^\prime(y)+yg^{\prime\prime}(y))+E(x,y)y(g(y)+yg^\prime(y))\big)\frac{\partial \tilde{u}}{\partial \eta}+F(x,y)\tilde{u}=0.
\end{array}$$
We want to find holomorphic functions  $f$ and $g$ in a neighborhood of he origin, so that
$$A(x)x^2(f(x)+xf^\prime(x))^2=A(0)\xi^2=A(0)(xf(x))^2$$ and

$$C(y)y^2(g(y)+yg^\prime(y))^2=C(0)\eta^2=C(0)(yg(y))^2.$$Hence

$$\frac{f(x)+xf^\prime(x)}{f(x)}=\pm\frac{\sqrt{A(0)}}{\sqrt{A(x)}}$$ and
$$\frac{g(y)+yg^\prime(y)}{g(y)}=\pm\frac{\sqrt{C(0)}}{\sqrt{C(y)}}.$$ Therefore

$$f(x)=\frac{k_1}{x}\exp\big(\pm\dint^x\frac{\sqrt{A(0)}}{\sqrt{A(s)}}ds\big)$$
 and
 $$g(y)=\frac{k_2}{y}\exp\big(\pm\dint^y\frac{\sqrt{C(0)}}{\sqrt{C(s)}}ds\big)$$
where $k_1,k_2$ are constant.
Now observe that
$$B(x,y)xy(f(x)+xf^\prime(x))(g(y)+yg^\prime(y))=\pm B(x,y)\xi\eta\frac{\sqrt{A(0)}}{\sqrt{A(x)}}\frac{\sqrt{C(0)}}{\sqrt{C(y)}}=\pm B(0,0)\xi\eta.$$

Thus, it is enough to choose
$$f(x)=\frac{1}{x}\exp\big(\dint^x\frac{\sqrt{A(0)}}{\sqrt{A(s)}}ds\big)$$
 and
 $$g(y)=\frac{1}{y}\exp\big(\dint^y\frac{\sqrt{C(0)}}{\sqrt{C(s)}}ds\big)$$
obtaining the desired change. Under these conditions we have that  $\tilde{u}$ is solution of equation (\ref{eqlem2}).
\end{proof}

\section{Parabolic equations}


Now we shall prove the existence and study the convergence of ``formal" solutions.
More precisely, we shall state and prove the effectiveness of a method of Frobenius for
such PDEs.
The first idea would be to put the PDE into a  simpler form. Nevertheless,
this method requires the use of affine transformations in the changes of coordinates.
This procedure has two main cons: One is the difficulty in order to find the affine transformations on each case. The second is the fact that  affine transformations very much change the expression of the solution.
This would therefore make the procedure less effective in computational terms, since this would be
more on the theoretical point of view. Somehow this would slide a bit from our original purpose.
Let us therefore analyze case by case, without performing any change of coordinates on the given
PDE.

We shall now deal with the parabolic  case. For this sake, we shall consider the second order PDE
\begin{equation}\label{eq16parab}
({\mathcal E})\, \, \,   L[z]:=Ax^2\frac{\partial^2 z}{\partial x^2}+Bxy\frac{\partial^2 z}{\partial x\partial y}+Cy^2\frac{\partial^2 z}{\partial y^2}+xa(x,y)\frac{\partial z}{\partial x}+yb(x,y)\frac{\partial z}{\partial y}+c(x,y)z=0
\end{equation}
where $A,B, C\in\mathbb{R}^*, \, \mathbb C^2$, $a(x,y),b(x,y)$ and $c(x,y)$ real analytic or holomorphic in $\Delta[(0,0),(R,R)]\subset\mathbb R^2 , \, \mathbb C^2$, $R>0$. The parabolic type is given by the condition
$B^2= 4AC$.

\begin{Definition}
{\rm We must say that the  PDE is {\it parabolic of real type} if $B=2\sqrt{A}\sqrt{C}$ and $\mbox{Re}(\sqrt{A}\overline{\sqrt{C}})>0$.
}
\end{Definition}
We observe that, though there is an ambiguity in the choices of $\sqrt{A}$ and $\sqrt{C}$, this does not
affect the condition $\mbox{Re}(\sqrt{A}\overline{\sqrt{C}})>0$. Indeed, the fact that $B=2\sqrt{A}\sqrt{C}$
shows that the condition  $\mbox{Re}(\sqrt{A}\overline{\sqrt{C}})>0$ does not depend on the choices of $\sqrt{A}$ and $\sqrt{C}$.

 In the real analytic case this condition is automatically fulfilled.

We state the full  complex and real analytic case:

\begin{Theorem}[parabolic complex]
\label{Theorem:paraboliccomplex}
A parabolic of real type complex analytic PDE, with a regular singularity, admits a convergent Frobenius type solution.
Indeed, consider the complex analytic second order PDE
\begin{equation}\label{eq16Thmparab}
L[z]:=Ax^2\frac{\partial^2 z}{\partial x^2}+Bxy\frac{\partial^2 z}{\partial x\partial y}+Cy^2\frac{\partial^2 z}{\partial y^2}+xa(x,y)\frac{\partial z}{\partial x}+yb(x,y)\frac{\partial z}{\partial y}+c(x,y)z=0
\end{equation}
 with $A,B,C\in\mathbb{C}^*$ such that $B=2\sqrt{A}\sqrt{C}$ and $\mbox{Re}(\sqrt{A}\overline{\sqrt{C}})>0$, $a(x,y),b(x,y)$ and $c(x,y)$ holomorphic in $\Delta[(0,0),(R,R)]$, $R>0$. Let $(r_0,s_0)\in \mathbb{C}^2$ be a non-resonant point of the indicial conic. Then there exists a Frobenius type convergent solution of (\ref{eq16Thmparab}) with initial monomial
$x^{r_0}y^{s_0}$ and  converging  in the bidisc $\Delta[(0,0),(R,R)]$. Moreover, the solution is real provided that the point $(r_0,s_0)$ and the coefficients of the PDE are real.
\end{Theorem}
\begin{proof}
Let $\varphi$ be a solution  of (\ref{eq16Thmparab}) of the form
\begin{equation}\label{eq17}
\varphi(x,y)=x^ry^s\sum^{\infty}_{|Q|=0}d_QX^Q
\end{equation}
where $d_{0,0}\neq0$. Given that $a(x,y),b(x,y)$ and $c(x,y)$ are holomorphic in $\Delta[(0,0),(R,R)]$ we have that
\begin{equation}\label{eq18}
a(x,y)=\sum^\infty_{|Q|=0}a_QX^Q,\;\;\;b(x,y)=\sum^\infty_{|Q|=0}b_QX^Q\;\;\;\mbox{ and }\;\;\;c(x,y)=\sum^\infty_{|Q|=0}c_QX^Q
\end{equation}
for every $(x,y)\in \Delta[(0,0),(R,R)]$.

Then
$$\frac{\partial\varphi}{\partial x}=\sum^\infty_{|Q|=0}(q_1+r)d_Qx^{q_1+r-1}y^{q_2+s} $$
$$\frac{\partial\varphi}{\partial y}=\sum^\infty_{|Q|=0}(q_2+s)d_Qx^{q_1+r}y^{q_2+s-1} $$
$$\frac{\partial^2\varphi}{\partial x^2}=\sum^\infty_{|Q|=0}(q_1+r)(q_1+r-1)d_Qx^{q_1+r-2}y^{q_2+s} $$
$$\frac{\partial^2\varphi}{\partial x \partial y}=\sum^\infty_{|Q|=0}(q_1+r)(q_2+s)d_Qx^{q_1+r-1}y^{q_2+s-1} $$
$$\frac{\partial^2\varphi}{\partial y^2}=\sum^\infty_{|Q|=0}(q_2+s)(q_2+s-1)d_Qx^{q_1+r}y^{q_2+s-2} $$
and from this we have
$$Ax^2\frac{\partial^2\varphi}{\partial x^2}=x^ry^s\sum^\infty_{|Q|=0}A(q_1+r)(q_1+r-1)d_QX^Q$$
$$Bxy\frac{\partial^2\varphi}{\partial x \partial y}=x^ry^s\sum^\infty_{|Q|=0}B(q_1+r)(q_2+s)d_QX^Q$$
$$Cy^2\frac{\partial^2\varphi}{\partial y^2}=x^ry^s\sum^\infty_{|Q|=0}C(q_2+s)(q_2+s-1)d_QX^Q$$

$$\begin{array}{l c l}xa(x,y)\frac{\partial \varphi}{\partial x}&=&x^ry^s\big(\sum^\infty_{|Q|=0}a_QX^Q\big)\big(\sum^\infty_{|Q|=0}(q_1+r)d_QX^Q\big)\\&&\\&=& x^ry^s\big(\sum^\infty_{|Q|=0}\tilde{a}_QX^Q\big)\end{array}$$ where $\tilde{a}_Q=\sum^{q_1}_{i=0}\sum^{q_2}_{j=0}(i+r)a_{q_1-i,q_2-j}d_{ij}$.

$$\begin{array}{l c l}yb(x,y)\frac{\partial\varphi}{\partial y}&=&x^ry^s\big(\sum^\infty_{|Q|=0}b_QX^Q\big)\big(\sum^\infty_{|Q|=0}(q_2+s)d_QX^Q\big)\\&&\\&=& x^ry^s\big(\sum^\infty_{|Q|=0}\tilde{b}_QX^Q\big)\end{array}$$ where $\tilde{b}_Q=\sum^{q_1}_{i=0}\sum^{q_2}_{j=0}(j+s)b_{q_1-i,q_2-j}d_{ij}$.

$$\begin{array}{l c l}c(x,y)\varphi&=&x^ry^s\big(\sum^\infty_{|Q|=0}c_QX^Q\big)\big(\sum^\infty_{|Q|=0}d_QX^Q\big)\\&&\\&=& x^ry^s\big(\sum^\infty_{|Q|=0}\tilde{c}_QX^Q\big)\end{array}$$ where $\tilde{c}_Q=\sum^{q_1}_{i=0}\sum^{q_2}_{j=0}c_{q_1-i,q_2-j}d_{ij}$.

Given that $\varphi$ is solution of (\ref{eq16Thmparab}) we have
$$Ax^2\frac{\partial^2 \varphi}{\partial x^2}+Bxy\frac{\partial^2 \varphi}{\partial x\partial y}+Cy^2\frac{\partial^2 \varphi}{\partial y^2}+xa(x,y)\frac{\partial \varphi}{\partial x}+yb(x,y)\frac{\partial \varphi}{\partial y}+c(x,y)\varphi=0$$
Therefore
$$x^ry^s\sum^\infty_{|Q|=0}\big(\big[A(q_1+r)(q_1+r-1)+B(q_1+r)(q_2+s)+C(q_2+s)(q_2+s-1)\big]d_Q+\tilde{a}_Q+\tilde{b}_Q+\tilde{c}_Q\big)X^Q=0 $$
hence we have
$$\big[A(q_1+r)(q_1+r-1)+B(q_1+r)(q_2+s)+C(q_2+s)(q_2+s-1)\big]d_Q+\tilde{a}_Q+\tilde{b}_Q+\tilde{c}_Q=0,\;\;\;|Q|=0,1,2,\ldots$$
Using the definitions of $\tilde{a}_Q$, $\tilde{b}_Q$ and $\tilde{c}_Q$ we write the equation anterior as
$$\begin{array}{c}\big[A(q_1+r)(q_1+r-1)+B(q_1+r)(q_2+s)+C(q_2+s)(q_2+s-1)\big]d_Q+\\
\\ \sum^{q_1}_{i=0}\sum^{q_2}_{j=0}\big[(i+r)a_{q_1-i,q_2-j}+(j+s)b_{q_1-i,q_2-j}+c_{q_1-i,q_2-j}\big]d_{i,j}=0\end{array}$$
equivalently
$$\begin{array}{c}\big[A(q_1+r)(q_1+r-1)+B(q_1+r)(q_2+s)+C(q_2+s)(q_2+s-1)+(q_1+r)a_{0,0}+(q_2+s)b_{0,0}+c_{0,0}\big]d_Q\\
\\ +\sum^{q_1-1}_{i=0}\sum^{q_2}_{j=0}\big[(i+r)a_{q_1-i,q_2-j}+(j+s)b_{q_1-i,q_2-j}+c_{q_1-i,q_2-j}\big]d_{i,j}\\
\\+\sum^{q_2-1}_{j=0}\big[(q_1+r)a_{0,q_2-j}+(j+s)b_{0,q_2-j}+c_{0,q_2-j}\big]d_{q_1,j} =0\end{array}$$

Para $|Q|=0$ we have
\begin{equation}\label{eq19} Ar(r-1)+Brs+Cs(s-1)+ra_{0,0}+sb_{0,0}+c_{0,0}=0
\end{equation}
provided that $d_{0,0}\neq0$. The second degree two variables polynomial $P$ given by $$P(r,s)=Ar^2+Brs+Cs^2+r(a_{0,0}-A)+s(b_{0,0}-C)+c_{0,0}$$
is called indicial conic associate to  equation (\ref{eq16Thmparab}).
We see that
\begin{equation}\label{eq20}
P(q_1+r,q_2+s)d_Q+e_Q=0,\;\;\;\;\;|Q|=1,2,\ldots
\end{equation}
where
\begin{equation}\label{eq21}
\begin{array}{c c l}e_Q&=&\sum^{q_1-1}_{i=0}\sum^{q_2}_{j=0}\big[(i+r)a_{q_1-i,q_2-j}+(j+s)b_{q_1-i,q_2-j}+c_{q_1-i,q_2-j}\big]d_{i,j}\\&&\\&&
+\sum^{q_2-1}_{j=0}\big[(q_1+r)a_{0,q_2-j}+(j+s)b_{0,q_2-j}+c_{0,q_2-j}\big]d_{q_1,j},\;\;\;\;\;|Q|=1,2,\ldots\end{array}
\end{equation}
Observe that $e_Q$ is a linear combination of $d_{0,0},d_{1,0},d_{0,1},\ldots,d_{n-1,0},d_{0,n-1}$, whose coefficients are given only in terms of the already known functions $a,b,c$, $r$ and $s$. Letting  $r$, $s$ and $d_{0,0}$ undetermined , as of now we solve  equations (\ref{eq20}) and (\ref{eq21}) in terms of $d_{0,0}$, $r$ and $s$. These solutions are defined by $D_Q(r,s)$, and a $e_Q$ corresponding to $E_Q(r,s)$.
Thus
$$E_{1,0}(r,s)=(ra_{1,0}+sb_{1,0}+c_{1,0})d_{0,0}\;\;\;\;\;\;D_{1,0}(r,s)=-\frac{E_{1,0}(r,s)}{P(1+r,s)},$$
$$E_{0,1}(r,s)=(ra_{0,1}+sb_{0,1}+c_{0,1})d_{0,0}\;\;\;\;\;\;D_{0,1}(r,s)=-\frac{E_{0,1}(r,s)}{P(r,1+s)},$$
 and in general:
 \begin{equation}\label{eq22}
\begin{array}{c c l}E_Q(r,s)&=&\sum^{q_1-1}_{i=0}\sum^{q_2}_{j=0}\big[(i+r)a_{q_1-i,q_2-j}+(j+s)b_{q_1-i,q_2-j}+c_{q_1-i,q_2-j}\big]D_{i,j}(r,s)\\&&\\&&
+\sum^{q_2-1}_{j=0}\big[(q_1+r)a_{0,q_2-j}+(j+s)b_{0,q_2-j}+c_{0,q_2-j}\big]D_{q_1,j}(r,s),\;\;|Q|=1,2,\ldots\end{array}
\end{equation}

\begin{equation}\label{eq23}
D_Q(r,s)=-\frac{E_Q(r,s)}{P(q_1+r,q_2+s)},\;\;\;\;\;|Q|=1,2,\ldots
\end{equation}

The coefficients $D_Q$ obtained in this way, are rational functions of $r$ and $s$, and the only points where they are not defined, are those points $r$ and $s$ for which $P(q_1+r,q_2+s)=0$ for some $|Q|=1,2,\ldots$.  Let us define $\varphi$ by:
\begin{equation}\label{eq24}
\varphi((x,y),(r,s))=d_{0,0}x^ry^s+x^ry^s\sum^{\infty}_{|Q|=1}D_Q(r,s)x^{q_1}y^{q_2}.
\end{equation}
 If the series (\ref{eq24}) converges in $\Delta[(0,0),(R,R)]$, then we have:
\begin{equation}\label{eq25}
L(\varphi)((x,y),(r,s))=d_{0,0}P(r,s)x^ry^s.
\end{equation}
We have the following situation: If a $\varphi$ given by (\ref{eq17}) is solution of (\ref{eq16Thmparab}), then $(r,s)$ must be a point of the indicial conic $P$, and then os $d_Q$ ($|Q|=1,2,\ldots$) are uniquely determined in terms of $d_{0,0}$, $r$ and $s$ by os $D_Q(r,s)$ from (\ref{eq23}), provided that $P(q_1+r,q_2+s)\neq0$, $|Q|=1,2,\ldots$.

Conversely, if $(r,s)$ is a root of $P$ and if os $D_Q(r,s)$ can be determined (i.e., $P(q_1+r,q_2+s)\neq0$ for $|Q|=1,2,\ldots$) then the function $\varphi$ given by $\varphi(x,y)=\varphi((x,y),(r,s))$ is solution of (\ref{eq16Thmparab}) for every choice of $d_{0,0}$, provided that the series (\ref{eq24}) is convergent.

By hypothesis $(r_0,s_0)$ is a point of the indicial conic $P$ such that $(r_0,s_0)\notin\mathcal{R}$, then $P(q_1+r_0,q_2+s_0)\neq0$ for every $|Q|=1,2,\ldots$. Thus, $D_Q(r_0,s_0)$ there exists for every $|Q|=1,2,\ldots$, and putting $d_{0,0}=D_{0,0}(r_0,s_0)=1$ we have that the function $\psi$ given by
\begin{equation}\label{eq26}
\psi(x,y)=x^{r_0}y^{s_0}\sum^\infty_{|Q|=0}D_Q(r_0,s_0)x^{q_1}y^{q_2},\;\;\;\;D_{0,0}(r_0,s_0)=1,
\end{equation}
 is a solution of (\ref{eq16Thmparab}), provided that the series is convergent.

We need to show that the series (\ref{eq26}) converges in the bidisc $\Delta[(0,0),(R,R)]$ where the coefficients $D_Q(r_0,s_0)$ are given recursively  by
\begin{equation}\label{eq27}
\begin{array}{c} D_{0,0}(r_0,s_0)=1,\\
\\
\begin{array}{c} P(q_1+r_0,q_2+s_0)D_Q(r_0,s_0)=-\sum^{q_1-1}_{i=0}\sum^{q_2}_{j=0}\big[(i+r_0)a_{q_1-i,q_2-j}+(j+s_0)b_{q_1-i,q_2-j}+c_{q_1-i,q_2-j}\big]D_{i,j}(r_0,s_0)\\ \\
-\sum^{q_2-1}_{j=0}\big[(q_1+r_0)a_{0,q_2-j}+(j+s_0)b_{0,q_2-j}+c_{0,q_2-j}\big]D_{q_1,j}(r_0,s_0),\;|Q|=1,2,\ldots\end{array}
\end{array}
\end{equation}
Observe that
$$P(q_1+r_0,q_2+s_0)=(\sqrt{A}q_1+\sqrt{C}q_2)^2+2q_1A\big[r_0+\frac{B}{2A}s_0+\frac{a(0,0)-A}{2A}\big]+2q_2C\big[s_0+\frac{B}{2C}r_0+\frac{b(0,0)-C}{2C}\big]$$
therefrom
$$\Vert P(q_1+r_0,q_2+s_0)\Vert\geq  \Vert \sqrt{A}q_1+\sqrt{C}q_2\Vert^2-2q_1\Vert A\Vert\left\Vert r_0+\frac{B}{2A}s_0+\frac{a(0,0)-A}{2A}\right\Vert-2q_2\Vert C\Vert\left\Vert s_0+\frac{B}{2C}r_0+\frac{b(0,0)-C}{2C}\right\Vert.
$$
Given that
$$\Vert \sqrt{A}q_1+\sqrt{C}q_2\Vert^2=\Vert A\Vert q_1^2+\Vert C\Vert q_2^2+2q_1q_2\mbox{Re}(\sqrt{A}\overline{\sqrt{C}})$$
be a $\alpha=\min\{\Vert A\Vert,\Vert C\Vert,\mbox{Re}(\sqrt{A}\overline{\sqrt{C}})\}>0$ we have
$$\Vert \sqrt{A}q_1+\sqrt{C}q_2\Vert^2\geq  \alpha (q_1^2+q_2^2+2q_1q_2)$$
$$\Vert \sqrt{A}q_1+\sqrt{C}q_2\Vert^2\geq  \alpha (q_1+q_2)^2.$$
Let  $\theta=\max\{\left\Vert r_0+\frac{B}{2A}s_0+\frac{a(0,0)-A}{2A}\right\Vert,\left\Vert s_0+\frac{B}{2C}r_0+\frac{b(0,0)-C}{2C}\right\Vert\}$ and $\beta=\max\{\Vert A \Vert,\Vert C \Vert\}$ hence we have
$$\Vert P(q_1+r_0,q_2+s_0)\Vert \geq \alpha(q_1+q_2)^2-2\theta\beta(q_1+q_2).$$
Consequently:
\begin{equation}\label{eq27'}
\Vert P(q_1+r_0,q_2+s_0)\Vert\geq \alpha (q_1+q_2)\big[(q_1+q_2)-\frac{2\theta\beta}{\alpha} \big].
\end{equation}

Let $\rho$ be any number that satisfies the inequality $0<\rho<R$. Given that the series defined in (\ref{eq18}) are convergent for $(x,y)=(\rho,\rho)$ there exists a constant $M>0$ such that

\begin{equation}\label{eq28}
\Vert a_Q\Vert\rho^{|Q|}\leq M,\;\;\;\;\;\Vert b_Q\Vert\rho^{|Q|}\leq M\;\;\;\;\;\Vert c_Q\Vert\rho^{|Q|}\leq M\;\;\;\;\;|Q|=0,1,2,\ldots
\end{equation}
Using (\ref{eq27'}) and (\ref{eq28}) in (\ref{eq27})we obtain
\begin{equation}\label{eq28a}
\begin{array}{c c l}\alpha|Q| \big[|Q|-\frac{2\theta\beta}{\alpha} \big]\Vert D_Q(r_0,s_0)\Vert&\leq& M\sum^{q_1-1}_{j=0}\sum^{q_2}_{j=0}[i+j+\Vert r_0\Vert+\Vert s_0\Vert+1]\rho^{i+j-|Q|}\Vert D_{i,j}(r_0,s_0)\Vert\\&&\\&&+M\sum^{q_2-1}_{j=0}[q_1+\Vert r_0\Vert+\Vert s_0\Vert+1]\rho^{j-q_2}\Vert D_{q_1,j}(r_0,s_0)\Vert.\end{array}
\end{equation}
Summing up all terms of norm $|Q|=n$ in (\ref{eq28a}) we have
\begin{equation}\label{eq28b}
\alpha n \big[n-\frac{2\theta\beta}{\alpha} \big]\sum_{|Q|=n}\Vert D_Q(r_0,s_0)\Vert \leq 2M\sum^{n-1}_{k=0}\big(k+\Vert r_0\Vert+\Vert s_0\Vert+1\big)\rho^{k-n}\big(\sum_{|Q|=k}\Vert D_Q(r_0,s_0)\Vert\big).
\end{equation}

Consider $\tilde{\psi}(x)=\sum^\infty_{|Q|=0}\Vert D_Q(r_0,s_0)\Vert x^{|Q|}$ the formal power series of nonnegative power in the variable $x$. We will show that $\tilde{\psi}$ converges in $D_R[0]$,
this will imply that $\psi$ converges in the bidisc $\Delta[(0,0),(R,R)]$. Let $d_n=\sum_{|Q|=n}\Vert D_Q(r_0,s_0)\Vert$ and then
$$\tilde{\psi}(x)=\sum^\infty_{n=0}\big(\sum_{|Q|=n}\Vert D_Q(r_0,s_0)\Vert \big)x^{n}=\sum^\infty_{n=0}d_nx^n.$$
Let $n_0$ be a natural number such that $n_0>\frac{2\theta\beta}{\alpha}$ and let us define $g_0,g_1,\ldots$ of the following way:
$$g_0=\Vert D_{0,0}(r_0,s_0)\Vert=1,\,\;\;g_n=\sum_{|Q|=n}\Vert D_Q(r_0,s_0)\Vert,\;\;(n=1,2,\ldots,n_0-1)$$e
\begin{equation}\label{eq28c}
\alpha n \big[n-\frac{2\theta\beta}{\alpha} \big]g_n=2M\sum^{n-1}_{k=0}\big(k+\Vert r_0\Vert+\Vert s_0\Vert+1\big)\rho^{k-n}g_k
\end{equation}
for $n=n_0,n_0+1,\ldots$.
Then, comparing the definition of $g_n$ with (\ref{eq28b}), we conclude that
\begin{equation}\label{eq28d}
d_n\leq g_n,\;\;\;\;\;\;n=0,1,2,\ldots
\end{equation}
Thus we will show that the series
\begin{equation}\label{eq28e}
\sum^\infty_{n=0}g_nx^n
\end{equation}
converges for $|x|<\rho$.

Replacing $n$ by $n+1$ in (\ref{eq28c}) we have:
$$\rho\alpha (n+1) \big[n+1-\frac{2\theta\beta}{\alpha} \big]g_{n+1}=\big(\alpha n \big[n-\frac{2\theta\beta}{\alpha} \big]+2M(n+\Vert r_0\Vert+\Vert s_0\Vert+1)\big)g_n$$
for $n=n_0,n_0+1,\ldots$. Thus
$$\left|\frac{g_{n+1}x^{n+1}}{g_{n}x^{n}}\right|=\frac{\alpha n \big[n-\frac{2\theta\beta}{\alpha} \big]+2M(n+\Vert r_0\Vert+\Vert s_0\Vert+1)}{\rho\alpha (n+1) \big[n+1-\frac{2\theta\beta}{\alpha} \big]}|x|$$ converges to  $|x|/\rho$ as $n\to\infty$. Thus according to the quotient test, the series (\ref{eq28e}) converges in $|x|<\rho$. Using (\ref{eq28d}) and by the comparison criteria, we conclude that the series
$$\sum^\infty_{n=0}d_nx^n,\;\;\;\;d_0=1,$$ converges in   $|x|<\rho$.  But given that $\rho$ is any number that satisfies the inequality $0<\rho<R$, we have already proved that this series converges for $|x|<R$.
\end{proof}

\section{Non-parabolic equations}

We shall now deal with the non-parabolic cases, in particular the elliptic and hyperbolic cases. For this sake, we shall consider the second order PDE
$$({\mathcal E})\, \, \,  Ax^2\frac{\partial^2 z}{\partial x^2}+Cy^2\frac{\partial^2 z}{\partial y^2}+xa(x,y)\frac{\partial z}{\partial x}+yb(x,y)\frac{\partial z}{\partial y}+c(x,y)z=0$$
where $A,C\in\mathbb{K}^*$, $a(x,y),b(x,y)$ and $c(x,y)$  analytic  in $\Delta[(0,0),(R,R)]\subset\mathbb K^2,\,  R>0$.

\begin{Definition}
{\rm The above PDE will be  classified as {\it elliptic} or {\it hyperbolic} according to
$\mbox{Re}(A \bar{C})>0$ or $\mbox{Re}(A\bar{C})<0$. }
\end{Definition}
The definition above extends the definition in the case of real PDEs.

\subsection{Non-parabolic: elliptic case}

We shall now state what we understand as the version of Frobenius method for elliptic complex PDEs:
\begin{Theorem}[elliptic complex]
\label{Theorem:ellipticcomplex}{
An elliptic second order linear homogeneous real analytic PDE
having a regular singular point, admits  Frobenius type solutions. More precisely,
consider the second order linear PDE

\begin{equation}\label{eq55}
L[z]:=Ax^2\frac{\partial^2 z}{\partial x^2}+Cy^2\frac{\partial^2 z}{\partial y^2}+xa(x,y)\frac{\partial z}{\partial x}+yb(x,y)\frac{\partial z}{\partial y}+c(x,y)z=0
\end{equation}
where $A,C\in\mathbb{C}^*$ with $\mbox{Re}(A\overline{C})>0$, $a(x,y),b(x,y)$ and $c(x,y)$ analytic in $\Delta[(0,0),(R,R)]$, $R>0$. Let $(r_0,s_0)$ be a non-resonant point of the indicial conic $\mathcal C \subset \mathbb C^2$. Then
$(\mathcal E)$ admits a convergent Frobenius type solution with initial monomial $x^{r_0} y^{s_0}$.
Moreover, the solution is real provided that the point $(r_0,s_0)$ and the coefficients of the
PDE are real.}
\end{Theorem}
\begin{proof}
Let $\varphi$ be a solution  of (\ref{eq55}) of the form
\begin{equation}\label{eq56}
\varphi(x,y)=x^ry^s\sum^{\infty}_{|Q|=0}d_QX^Q
\end{equation}
where $d_{0,0}\neq0$. Given that $a(x,y),b(x,y)$ and $c(x,y)$ are holomorphic in $\Delta[(0,0),(R,R)]$ we have that
\begin{equation}\label{eq57}
a(x,y)=\sum^\infty_{|Q|=0}a_QX^Q,\;\;\;b(x,y)=\sum^\infty_{|Q|=0}b_QX^Q\;\;\;\mbox{ and }\;\;\;c(x,y)=\sum^\infty_{|Q|=0}c_QX^Q
\end{equation}
for every $(x,y)\in \Delta[(0,0),(R,R)]$.

Then
$$\frac{\partial\varphi}{\partial x}=\sum^\infty_{|Q|=0}(q_1+r)d_Qx^{q_1+r-1}y^{q_2+s} $$
$$\frac{\partial\varphi}{\partial y}=\sum^\infty_{|Q|=0}(q_2+s)d_Qx^{q_1+r}y^{q_2+s-1} $$
$$\frac{\partial^2\varphi}{\partial x^2}=\sum^\infty_{|Q|=0}(q_1+r)(q_1+r-1)d_Qx^{q_1+r-2}y^{q_2+s} $$
$$\frac{\partial^2\varphi}{\partial y^2}=\sum^\infty_{|Q|=0}(q_2+s)(q_2+s-1)d_Qx^{q_1+r}y^{q_2+s-2} $$
and from this we have
$$Ax^2\frac{\partial^2\varphi}{\partial x^2}=x^ry^s\sum^\infty_{|Q|=0}A(q_1+r)(q_1+r-1)d_QX^Q$$
$$Cy^2\frac{\partial^2\varphi}{\partial y^2}=x^ry^s\sum^\infty_{|Q|=0}C(q_2+s)(q_2+s-1)d_QX^Q$$

$$\begin{array}{l c l}xa(x,y)\frac{\partial \varphi}{\partial x}&=&x^ry^s\big(\sum^\infty_{|Q|=0}a_QX^Q\big)\big(\sum^\infty_{|Q|=0}(q_1+r)d_QX^Q\big)\\&&\\&=& x^ry^s\big(\sum^\infty_{|Q|=0}\tilde{a}_QX^Q\big)\end{array}$$ where $\tilde{a}_Q=\sum^{q_1}_{i=0}\sum^{q_2}_{j=0}(i+r)a_{q_1-i,q_2-j}d_{ij}$.

$$\begin{array}{l c l}yb(x,y)\frac{\partial\varphi}{\partial y}&=&x^ry^s\big(\sum^\infty_{|Q|=0}b_QX^Q\big)\big(\sum^\infty_{|Q|=0}(q_2+s)d_QX^Q\big)\\&&\\&=& x^ry^s\big(\sum^\infty_{|Q|=0}\tilde{b}_QX^Q\big)\end{array}$$ where $\tilde{b}_Q=\sum^{q_1}_{i=0}\sum^{q_2}_{j=0}(j+s)b_{q_1-i,q_2-j}d_{ij}$.

$$\begin{array}{l c l}c(x,y)\varphi&=&x^ry^s\big(\sum^\infty_{|Q|=0}c_QX^Q\big)\big(\sum^\infty_{|Q|=0}d_QX^Q\big)\\&&\\&=& x^ry^s\big(\sum^\infty_{|Q|=0}\tilde{c}_QX^Q\big)\end{array}$$ where $\tilde{c}_Q=\sum^{q_1}_{i=0}\sum^{q_2}_{j=0}c_{q_1-i,q_2-j}d_{ij}$.

Given that $\varphi$ is solution of (\ref{eq55}) we have
$$Ax^2\frac{\partial^2 \varphi}{\partial x^2}+Cy^2\frac{\partial^2 \varphi}{\partial y^2}+xa(x,y)\frac{\partial \varphi}{\partial x}+yb(x,y)\frac{\partial \varphi}{\partial y}+c(x,y)\varphi=0$$
Therefore
$$x^ry^s\sum^\infty_{|Q|=0}\big(\big[A(q_1+r)(q_1+r-1)+C(q_2+s)(q_2+s-1)\big]d_Q+\tilde{a}_Q+\tilde{b}_Q+\tilde{c}_Q\big)X^Q=0 $$
hence we have
$$\big[A(q_1+r)(q_1+r-1)+C(q_2+s)(q_2+s-1)\big]d_Q+\tilde{a}_Q+\tilde{b}_Q+\tilde{c}_Q=0,\;\;\;|Q|=0,1,2,\ldots$$
Using the definitions of $\tilde{a}_Q$, $\tilde{b}_Q$ and $\tilde{c}_Q$ we write the equation anterior as
$$\begin{array}{c}\big[A(q_1+r)(q_1+r-1)+C(q_2+s)(q_2+s-1)\big]d_Q+\\
\\ \sum^{q_1}_{i=0}\sum^{q_2}_{j=0}\big[(i+r)a_{q_1-i,q_2-j}+(j+s)b_{q_1-i,q_2-j}+c_{q_1-i,q_2-j}\big]d_{i,j}=0\end{array}$$
equivalently
$$\begin{array}{c}\big[A(q_1+r)(q_1+r-1)+C(q_2+s)(q_2+s-1)+(q_1+r)a_{0,0}+(q_2+s)b_{0,0}+c_{0,0}\big]d_Q\\
\\ +\sum^{q_1-1}_{i=0}\sum^{q_2}_{j=0}\big[(i+r)a_{q_1-i,q_2-j}+(j+s)b_{q_1-i,q_2-j}+c_{q_1-i,q_2-j}\big]d_{i,j}\\
\\+\sum^{q_2-1}_{j=0}\big[(q_1+r)a_{0,q_2-j}+(j+s)b_{0,q_2-j}+c_{0,q_2-j}\big]d_{q_1,j} =0\end{array}$$

Para $|Q|=0$ we have
\begin{equation}\label{eq58} Ar(r-1)+Cs(s-1)+ra_{0,0}+sb_{0,0}+c_{0,0}=0
\end{equation}
provided that $d_{0,0}\neq0$. The second degree two variables polynomial $P$ given by $$P(r,s)=Ar^2+Cs^2+r(a_{0,0}-A)+s(b_{0,0}-C)+c_{0,0}$$
is the indicial conic associate to  equation (\ref{eq55}).
We see that
\begin{equation}\label{eq59}
P(q_1+r,q_2+s)d_Q+e_Q=0,\;\;\;\;\;|Q|=1,2,\ldots
\end{equation}
where
\begin{equation}\label{eq60}
\begin{array}{c c l}e_Q&=&\sum^{q_1-1}_{i=0}\sum^{q_2}_{j=0}\big[(i+r)a_{q_1-i,q_2-j}+(j+s)b_{q_1-i,q_2-j}+c_{q_1-i,q_2-j}\big]d_{i,j}\\&&\\&&
+\sum^{q_2-1}_{j=0}\big[(q_1+r)a_{0,q_2-j}+(j+s)b_{0,q_2-j}+c_{0,q_2-j}\big]d_{q_1,j},\;\;\;\;\;|Q|=1,2,\ldots\end{array}
\end{equation}
Observe that $e_Q$ is a linear combination of $d_{0,0},d_{1,0},d_{0,1},\ldots,d_{n-1,0},d_{0,n-1}$, whose coefficients are given only in terms of the already known functions $a,b,c$, $r$ and $s$. Letting  $r$, $s$ and $d_{0,0}$ undetermined , as of now we solve  equations (\ref{eq59}) and (\ref{eq60}) in terms of $d_{0,0}$, $r$ and $s$. These solutions are defined by $D_Q(r,s)$, and a $e_Q$ corresponding to $E_Q(r,s)$.
Thus
$$E_{1,0}(r,s)=(ra_{1,0}+sb_{1,0}+c_{1,0})d_{0,0}\;\;\;\;\;\;D_{1,0}(r,s)=-\frac{E_{1,0}(r,s)}{P(1+r,s)},$$
$$E_{0,1}(r,s)=(ra_{0,1}+sb_{0,1}+c_{0,1})d_{0,0}\;\;\;\;\;\;D_{0,1}(r,s)=-\frac{E_{0,1}(r,s)}{P(r,1+s)},$$
 and in general:
 \begin{equation}\label{eq61}
\begin{array}{c c l}E_Q(r,s)&=&\sum^{q_1-1}_{i=0}\sum^{q_2}_{j=0}\big[(i+r)a_{q_1-i,q_2-j}+(j+s)b_{q_1-i,q_2-j}+c_{q_1-i,q_2-j}\big]D_{i,j}(r,s)\\&&\\&&
+\sum^{q_2-1}_{j=0}\big[(q_1+r)a_{0,q_2-j}+(j+s)b_{0,q_2-j}+c_{0,q_2-j}\big]D_{q_1,j}(r,s),\;\;|Q|=1,2,\ldots\end{array}
\end{equation}

\begin{equation}\label{eq62}
D_Q(r,s)=-\frac{E_Q(r,s)}{P(q_1+r,q_2+s)},\;\;\;\;\;|Q|=1,2,\ldots
\end{equation}

The coefficients $D_Q$ obtained in this way, are rational functions of $r$ and $s$, and the only points where they are not defined, are those points $r$ and $s$ for which $P(q_1+r,q_2+s)=0$ for some $|Q|=1,2,\ldots$.  Let us define $\varphi$ by:
\begin{equation}\label{eq63}
\varphi((x,y),(r,s))=d_{0,0}x^ry^s+x^ry^s\sum^{\infty}_{|Q|=1}D_Q(r,s)x^{q_1}y^{q_2}.
\end{equation}
 If the series (\ref{eq63}) converges in $\Delta[(0,0),(R,R)]$, then we have:
\begin{equation}\label{eq64}
L(\varphi)((x,y),(r,s))=d_{0,0}P(r,s)x^ry^s.
\end{equation}
We have the following situation: If a $\varphi$ given by (\ref{eq56}) is solution of (\ref{eq55}), then $(r,s)$ must be a point of the indicial conic $P$, and then os $d_Q$ ($|Q|=1,2,\ldots$) are uniquely determined in terms of $d_{0,0}$, $r$ and $s$ by os $D_Q(r,s)$ from (\ref{eq62}), provided that $P(q_1+r,q_2+s)\neq0$, $|Q|=1,2,\ldots$.

Conversely, if $(r,s)$ is a root of $P$ and if os $D_Q(r,s)$ can be determined (i.e., $P(q_1+r,q_2+s)\neq0$ for $|Q|=1,2,\ldots$) then the function $\varphi$ given by $\varphi(x,y)=\varphi((x,y),(r,s))$ is solution of (\ref{eq64}) for every choice of $d_{0,0}$, provided that the series (\ref{eq63}) is convergent.
By hypothesis $(r_0,s_0)$ is a point of the indicial conic $P$ such that $(r_0,s_0)\notin\mathcal{R}$, then $P(q_1+r_0,q_2+s_0)\neq0$ for every $|Q|=1,2,\ldots$. Thus, $D_Q(r_0,s_0)$ there exists for every $|Q|=1,2,\ldots$, and putting $d_{0,0}=D_{0,0}(r_0,s_0)=1$ we have that the function $\psi$ given by
\begin{equation}\label{eq65}
\psi(x,y)=x^{r_0}y^{s_0}\sum^\infty_{|Q|=0}D_Q(r_0,s_0)x^{q_1}y^{q_2},\;\;\;\;D_{0,0}(r_0,s_0)=1,
\end{equation}
 is a solution of (\ref{eq55}), provided that the series is convergent.

We need to show that the series (\ref{eq65}) converges in the bidisc $\Delta[(0,0),(R,R)]$,
 where os $D_{Q}(r_0,s_0)$ are given recursively  by
 \begin{equation}\label{eq66}
 D_{0,0}(r_0,s_0)=1,
\end{equation}
\[
P(q_1+r_0,q_2+s_0)D_Q(r_0,s_0)=-\sum^{q_1-1}_{i=0}\sum^{q_2}_{j=0}
\big[(i+r_0)a_{q_1-i,q_2-j}+(j+s_0)b_{q_1-i,q_2-j}+c_{q_1-i,q_2-j}\big]D_{i,j}(r_0,s_0)
\]
\[
-\sum^{q_2-1}_{j=0}\big[(q_1+r_0)a_{0,q_2-j}+(j+s_0)b_{0,q_2-j}+
c_{0,q_2-j}\big]D_{q_1,j}(r_0,s_0),\;|Q|=1,2,\ldots
\]
Observe that
$$P(q_1+r_0,q_2+s_0)=Aq_1^2+Cq_2^2+2q_1A\big[r_0+\frac{a(0,0)-A}{2A}\big]+2q_2C\big[s_0+\frac{b(0,0)-C}{2C}\big]$$
therefrom
$$\Vert P(q_1+r_0,q_2+s_0)\Vert\geq  \Vert Aq_1^2+Cq_2^2\Vert-2q_1\Vert A\Vert\left\Vert r_0+\frac{a(0,0)-A}{2A}\right\Vert-2q_2\Vert C\Vert\left\Vert s_0+\frac{b(0,0)-C}{2C}\right\Vert.
$$
Given that
$$\Vert Aq_1^2+Cq_2^2\Vert^2=\Vert A\Vert^2 q_1^4+\Vert C\Vert^2 q_2^4+2q_1^2q_2^2\mbox{Re}(A\overline{C})$$
be a $\alpha=\min\{\Vert A\Vert^2,\Vert C\Vert^2,\mbox{Re}(A\overline{C})\}>0$ we have
$$\Vert Aq_1^2+Cq_2^2\Vert^2\geq  \alpha (q_1^4+q_2^4+2q_1^2q_2^2)$$
$$\Vert Aq_1^2+Cq_2^2\Vert^2\geq  \alpha (q_1^2+q_2^2)^2\geq \alpha\frac{(q_1+q_2)^4}{4}$$
$$\Vert Aq_1^2+Cq_2^2\Vert\geq \sqrt{\alpha}\frac{(q_1+q_2)^2}{2}.$$
Let  $\theta=\max\{\left\Vert r_0+\frac{a(0,0)-A}{2A}\right\Vert,\left\Vert s_0+\frac{b(0,0)-C}{2C}\right\Vert\}$ and $\beta=\max\{\Vert A \Vert,\Vert C \Vert\}$ hence we have
$$\Vert P(q_1+r_0,q_2+s_0)\Vert \geq \sqrt{\alpha}\frac{(q_1+q_2)^2}{2}-2\theta\beta(q_1+q_2).$$
Consequently:
\begin{equation}\label{eq67}
\Vert P(q_1+r_0,q_2+s_0)\Vert \geq \frac{\sqrt{\alpha}}{2}(q_1+q_2)\big[(q_1+q_2)-\frac{4\theta\beta}{\sqrt{\alpha}}\big].
\end{equation}

Let $\rho$ be any number that satisfies the inequality $0<\rho<R$. Given that the series defined in (\ref{eq57}) are convergent for $(x,y)=(\rho,\rho)$ there exists a constant $M>0$ such that

\begin{equation}\label{eq68}
\Vert a_Q\Vert\rho^{|Q|}\leq M,\;\;\;\;\;\Vert b_Q\Vert\rho^{|Q|}\leq M\;\;\;\;\;\Vert c_Q\Vert\rho^{|Q|}\leq M\;\;\;\;\;|Q|=0,1,2,\ldots
\end{equation}
Using (\ref{eq67}) and (\ref{eq68}) in (\ref{eq66}) we obtain
\begin{equation}\label{eq69}
\begin{array}{c c l}\frac{\sqrt{\alpha}}{2}|Q|\big[|Q|-\frac{4\theta\beta}{\sqrt{\alpha}}\big]\Vert D_Q(r_0,s_0)\Vert&\leq& M\sum^{q_1-1}_{j=0}\sum^{q_2}_{j=0}[i+j+\Vert r_0\Vert+\Vert s_0\Vert+1]\rho^{i+j-|Q|}\Vert D_{i,j}(r_0,s_0)\Vert\\&&\\&&+M\sum^{q_2-1}_{j=0}[q_1+\Vert r_0\Vert+\Vert s_0\Vert+1]\rho^{j-q_2}\Vert D_{q_1,j}(r_0,s_0)\Vert.\end{array}
\end{equation}
Summing up all terms of norm $|Q|=n$ in (\ref{eq69}) we have
\begin{equation}\label{eq70}
\frac{\sqrt{\alpha}}{2}n\big[n-\frac{4\theta\beta}{\sqrt{\alpha}}\big]\sum_{|Q|=n}\Vert D_Q(r_0,s_0)\Vert\leq 2M\sum^{n-1}_{k=0}\big(k+\Vert r_0\Vert+\Vert s_0\Vert+1\big)\rho^{k-n}\big(\sum_{|Q|=k}\Vert D_Q(r_0,s_0)\Vert\big).
\end{equation}

Consider $\tilde{\psi}(x)=\sum^\infty_{|Q|=0}\Vert D_Q(r_0,s_0)\Vert x^{|Q|}$ the formal power series
of nonnegative power in the variable $x$. We will show that $\tilde{\psi}$ converges in $D_R[0]$, this will imply that $\psi$ converges in the bidisc $\Delta[(0,0),(R,R)]$. Let $d_n=\sum_{|Q|=n}\Vert D_Q(r_0,s_0)\Vert$ and then
$$\tilde{\psi}(x)=\sum^\infty_{n=0}\big(\sum_{|Q|=n}\Vert D_Q(r_0,s_0)\Vert\big)x^{n}=\sum^\infty_{n=0}d_nx^n.$$
Let $n_0$ be a natural number such that $n_0>\frac{4\theta\beta}{\sqrt{\alpha}}$ and let us define $g_0,g_1,\ldots$ of the following way:
$$g_0=\Vert D_{0,0}(r_0,s_0)\Vert=1,\,\;\;g_n=\sum_{|Q|=n}\Vert D_Q(r_0,s_0)\Vert,\;\;(n=1,2,\ldots,n_0-1)$$ and
\begin{equation}\label{eq71}
\frac{\sqrt{\alpha}}{2}n\big[n-\frac{4\theta\beta}{\sqrt{\alpha}}\big]g_n=2M\sum^{n-1}_{k=0}\big(k+\Vert r_0\Vert+\Vert s_0\Vert+1\big)\rho^{k-n}g_k
\end{equation}
for $n=n_0,n_0+1,\ldots$.
Then, comparing the definition of $g_n$ with (\ref{eq70}), we conclude that
\begin{equation}\label{eq72}
d_n\leq g_n,\;\;\;\;\;\;n=0,1,2,\ldots
\end{equation}
Thus we will show that the series
\begin{equation}\label{eq73}
\sum^\infty_{n=0}g_nx^n
\end{equation}
converges for $|x|<\rho$.
Replacing $n$ by $n+1$ in (\ref{eq71}) we have:
$$\rho\frac{\sqrt{\alpha}}{2}(n+1)\big[n+1-\frac{4\theta\beta}{\sqrt{\alpha}}\big]g_{n+1}=\big(\frac{\sqrt{\alpha}}{2}n\big[n-\frac{4\theta\beta}{\sqrt{\alpha}}\big]+2M(n+\Vert r_0\Vert+\Vert s_0\Vert+1)\big)g_n$$
for $n=n_0,n_0+1,\ldots$. Thus
$$\left|\frac{g_{n+1}x^{n+1}}{g_{n}x^{n}}\right|=\frac{\frac{\sqrt{\alpha}}{2}n\big[n-\frac{4\theta\beta}{\sqrt{\alpha}}\big]+2M(n+\Vert r_0\Vert+\Vert s_0\Vert+1)}{\rho\frac{\sqrt{\alpha}}{2}(n+1)\big[n+1-\frac{4\theta\beta}{\sqrt{\alpha}}\big]}|x|$$ converges to  $|x|/\rho$ as $n\to\infty$. Thus according to the quotient test, the series (\ref{eq73}) converges in $|x|<\rho$. Using (\ref{eq72}) and by the comparison criteria, we conclude that the series
$$\sum^\infty_{n=0}d_nx^n,\;\;\;\;d_0=1,$$ converges in   $|x|<\rho$.  But given that $\rho$ is any number that satisfies the inequality $0<\rho<R$, we have already proved that this series converges for $|x|<R$.
\end{proof}

\subsection{Non-parabolic: hyperbolic case}

We present now our version of Frobenius method for second order linear complex PDEs of hyperbolic type:
\begin{Theorem}
[non parabolic convergent II]
\label{Theorem:hyperboliccomplex}{
A hyperbolic second order linear homogeneous complex analytic PDE
having a regular singular point, admits  Frobenius type solutions. More precisely,
consider the second order PDE
\begin{equation}\label{eqh1}
L[z]:=Ax^2\frac{\partial^2 z}{\partial x^2}+Cy^2\frac{\partial^2 z}{\partial y^2}+xa(x,y)\frac{\partial z}{\partial x}+yb(x,y)\frac{\partial z}{\partial y}+c(x,y)z=0
\end{equation}
where $A,C\in\mathbb{C}^*$ with $\mbox{Re}({A}\overline{{C}})<0$, $a(x,y),b(x,y)$ and $c(x,y)$ holomorphic in $\Delta[(0,0),(R,R)]$, $R>0$.
Let $(r_0,s_0)$ be a non-resonant point of the indicial conic. Then
$(\mathcal E)$ admits a convergent Frobenius type solution with initial monomial $x^{r_0} y^{s_0}$.
Moreover, the solution is real provided that the point $(r_0,s_0)$ and the coefficients of the
PDE are real.}
\end{Theorem}
\begin{proof}
Let $\varphi$ be a solution  of (\ref{eqh1}) of the form
\begin{equation}\label{eqh2}
\varphi(x,y)=x^ry^s\sum^{\infty}_{|Q|=0}d_QX^Q
\end{equation}
where $d_{0,0}\neq0$. Given that $a(x,y),b(x,y)$ and $c(x,y)$ are holomorphic in $\Delta[(0,0),(R,R)]$ we have that
\begin{equation}\label{eqh3}
a(x,y)=\sum^\infty_{|Q|=0}a_QX^Q,\;\;\;b(x,y)=\sum^\infty_{|Q|=0}b_QX^Q\;\;\;\mbox{ and }\;\;\;c(x,y)=\sum^\infty_{|Q|=0}c_QX^Q
\end{equation}
for every $(x,y)\in \Delta[(0,0),(R,R)]$.

Then
$$\frac{\partial\varphi}{\partial x}=\sum^\infty_{|Q|=0}(q_1+r)d_Qx^{q_1+r-1}y^{q_2+s} $$
$$\frac{\partial\varphi}{\partial y}=\sum^\infty_{|Q|=0}(q_2+s)d_Qx^{q_1+r}y^{q_2+s-1} $$
$$\frac{\partial^2\varphi}{\partial x^2}=\sum^\infty_{|Q|=0}(q_1+r)(q_1+r-1)d_Qx^{q_1+r-2}y^{q_2+s} $$
$$\frac{\partial^2\varphi}{\partial y^2}=\sum^\infty_{|Q|=0}(q_2+s)(q_2+s-1)d_Qx^{q_1+r}y^{q_2+s-2} $$
and from this we have
$$Ax^2\frac{\partial^2\varphi}{\partial x^2}=x^ry^s\sum^\infty_{|Q|=0}A(q_1+r)(q_1+r-1)d_QX^Q$$
$$Cy^2\frac{\partial^2\varphi}{\partial y^2}=x^ry^s\sum^\infty_{|Q|=0}C(q_2+s)(q_2+s-1)d_QX^Q$$

$$\begin{array}{l c l}xa(x,y)\frac{\partial \varphi}{\partial x}&=&x^ry^s\big(\sum^\infty_{|Q|=0}a_QX^Q\big)\big(\sum^\infty_{|Q|=0}(q_1+r)d_QX^Q\big)\\&&\\&=& x^ry^s\big(\sum^\infty_{|Q|=0}\tilde{a}_QX^Q\big)\end{array}$$ where $\tilde{a}_Q=\sum^{q_1}_{i=0}\sum^{q_2}_{j=0}(i+r)a_{q_1-i,q_2-j}d_{ij}$.

$$\begin{array}{l c l}yb(x,y)\frac{\partial\varphi}{\partial y}&=&x^ry^s\big(\sum^\infty_{|Q|=0}b_QX^Q\big)\big(\sum^\infty_{|Q|=0}(q_2+s)d_QX^Q\big)\\&&\\&=& x^ry^s\big(\sum^\infty_{|Q|=0}\tilde{b}_QX^Q\big)\end{array}$$ where $\tilde{b}_Q=\sum^{q_1}_{i=0}\sum^{q_2}_{j=0}(j+s)b_{q_1-i,q_2-j}d_{ij}$.

$$\begin{array}{l c l}c(x,y)\varphi&=&x^ry^s\big(\sum^\infty_{|Q|=0}c_QX^Q\big)\big(\sum^\infty_{|Q|=0}d_QX^Q\big)\\&&\\&=& x^ry^s\big(\sum^\infty_{|Q|=0}\tilde{c}_QX^Q\big)\end{array}$$ where $\tilde{c}_Q=\sum^{q_1}_{i=0}\sum^{q_2}_{j=0}c_{q_1-i,q_2-j}d_{ij}$.

Given that $\varphi$ is solution of (\ref{eqh1}) we have
$$Ax^2\frac{\partial^2 \varphi}{\partial x^2}+Cy^2\frac{\partial^2 \varphi}{\partial y^2}+xa(x,y)\frac{\partial \varphi}{\partial x}+yb(x,y)\frac{\partial \varphi}{\partial y}+c(x,y)\varphi=0$$
Therefore
$$x^ry^s\sum^\infty_{|Q|=0}\big(\big[A(q_1+r)(q_1+r-1)+C(q_2+s)(q_2+s-1)\big]d_Q+\tilde{a}_Q+\tilde{b}_Q+\tilde{c}_Q\big)X^Q=0 $$
hence we have
$$\big[A(q_1+r)(q_1+r-1)+C(q_2+s)(q_2+s-1)\big]d_Q+\tilde{a}_Q+\tilde{b}_Q+\tilde{c}_Q=0,\;\;\;|Q|=0,1,2,\ldots$$
Using the definitions of $\tilde{a}_Q$, $\tilde{b}_Q$ and $\tilde{c}_Q$ we write the equation anterior as
$$\begin{array}{c}\big[A(q_1+r)(q_1+r-1)+C(q_2+s)(q_2+s-1)\big]d_Q+\\
\\ \sum^{q_1}_{i=0}\sum^{q_2}_{j=0}\big[(i+r)a_{q_1-i,q_2-j}+(j+s)b_{q_1-i,q_2-j}+c_{q_1-i,q_2-j}\big]d_{i,j}=0\end{array}$$
equivalently
$$\begin{array}{c}\big[A(q_1+r)(q_1+r-1)+C(q_2+s)(q_2+s-1)+(q_1+r)a_{0,0}+(q_2+s)b_{0,0}+c_{0,0}\big]d_Q\\
\\ +\sum^{q_1-1}_{i=0}\sum^{q_2}_{j=0}\big[(i+r)a_{q_1-i,q_2-j}+(j+s)b_{q_1-i,q_2-j}+c_{q_1-i,q_2-j}\big]d_{i,j}\\
\\+\sum^{q_2-1}_{j=0}\big[(q_1+r)a_{0,q_2-j}+(j+s)b_{0,q_2-j}+c_{0,q_2-j}\big]d_{q_1,j} =0\end{array}$$
Para $|Q|=0$ we have
\begin{equation}\label{eqh4} Ar(r-1)+Cs(s-1)+ra_{0,0}+sb_{0,0}+c_{0,0}=0
\end{equation}
provided that $d_{0,0}\neq0$. The second degree two variables polynomial $P$ given by $$P(r,s)=Ar^2+Cs^2+r(a_{0,0}-A)+s(b_{0,0}-C)+c_{0,0}$$
is called indicial conic associate to  equation (\ref{eqh1}). We see that
\begin{equation}\label{eqh4'}
P(q_1+r,q_2+s)d_Q+e_Q=0,\;\;\;\;\;|Q|=1,2,\ldots
\end{equation}
where
\begin{equation}\label{eqh5}
\begin{array}{c c l}e_Q&=&\sum^{q_1-1}_{i=0}\sum^{q_2}_{j=0}\big[(i+r)a_{q_1-i,q_2-j}+(j+s)b_{q_1-i,q_2-j}+c_{q_1-i,q_2-j}\big]d_{i,j}\\&&\\&&
+\sum^{q_2-1}_{j=0}\big[(q_1+r)a_{0,q_2-j}+(j+s)b_{0,q_2-j}+c_{0,q_2-j}\big]d_{q_1,j},\;\;\;\;\;|Q|=1,2,\ldots\end{array}
\end{equation}
Observe that $e_Q$ is a linear combination of $d_{0,0},d_{1,0},d_{0,1},\ldots,d_{n-1,0},d_{0,n-1}$, whose coefficients are given only in terms of the already known functions $a,b,c$, $r$ and $s$. Letting  $r$, $s$ and $d_{0,0}$ undetermined , as of now we solve  equations (\ref{eqh4'}) and (\ref{eqh5}) in terms of $d_{0,0}$, $r$ and $s$. These solutions are defined by $D_Q(r,s)$, and a $e_Q$ corresponding to $E_Q(r,s)$.
Thus
$$E_{1,0}(r,s)=(ra_{1,0}+sb_{1,0}+c_{1,0})d_{0,0}\;\;\;\;\;\;D_{1,0}(r,s)=-\frac{E_{1,0}(r,s)}{P(1+r,s)},$$
$$E_{0,1}(r,s)=(ra_{0,1}+sb_{0,1}+c_{0,1})d_{0,0}\;\;\;\;\;\;D_{0,1}(r,s)=-\frac{E_{0,1}(r,s)}{P(r,1+s)},$$
 and in general:
 \begin{equation}\label{eqh6}
\begin{array}{c c l}E_Q(r,s)&=&\sum^{q_1-1}_{i=0}\sum^{q_2}_{j=0}\big[(i+r)a_{q_1-i,q_2-j}+(j+s)b_{q_1-i,q_2-j}+c_{q_1-i,q_2-j}\big]D_{i,j}(r,s)\\&&\\&&
+\sum^{q_2-1}_{j=0}\big[(q_1+r)a_{0,q_2-j}+(j+s)b_{0,q_2-j}+c_{0,q_2-j}\big]D_{q_1,j}(r,s),\;\;|Q|=1,2,\ldots\end{array}
\end{equation}
\begin{equation}\label{eqh7}
D_Q(r,s)=-\frac{E_Q(r,s)}{P(q_1+r,q_2+s)},\;\;\;\;\;|Q|=1,2,\ldots
\end{equation}
The coefficients $D_Q$ obtained in this way, are rational functions of $r$ and $s$, and the only points where they are not defined, are those points $r$ and $s$ for which $P(q_1+r,q_2+s)=0$ for some $|Q|=1,2,\ldots$.  Let us define $\varphi$ by:
\begin{equation}\label{eqh8}
\varphi((x,y),(r,s))=d_{0,0}x^ry^s+x^ry^s\sum^{\infty}_{|Q|=1}D_Q(r,s)x^{q_1}y^{q_2}.
\end{equation}
 If the series (\ref{eqh8}) converges in $\Delta[(0,0),(R,R)]$, then we have:
\begin{equation}\label{eqh9}
L(\varphi)((x,y),(r,s))=d_{0,0}P(r,s)x^ry^s.
\end{equation}
We have the following situation: If a $\varphi$ given by (\ref{eqh2}) is solution of (\ref{eqh1}), then $(r,s)$ must be a point of the indicial conic $P$, and then os $d_Q$ ($|Q|=1,2,\ldots$) are uniquely determined in terms of $d_{0,0}$, $r$ and $s$ by os $D_Q(r,s)$ from (\ref{eqh7}), provided that $P(q_1+r,q_2+s)\neq0$, $|Q|=1,2,\ldots$.

Conversely, if $(r,s)$ is a root of $P$ and if os $D_Q(r,s)$ can be determined (i.e., $P(q_1+r,q_2+s)\neq0$ for $|Q|=1,2,\ldots$) then the function $\varphi$ given by $\varphi(x,y)=\varphi((x,y),(r,s))$ is solution of (\ref{eqh9}) for every choice of $d_{0,0}$, provided that the series (\ref{eqh2}) is convergent.

By hypothesis $(r_0,s_0)$ is a point of the indicial conic $P$ such that $(r_0,s_0)\notin\mathcal{R}$, then $P(q_1+r_0,q_2+s_0)\neq0$ for every $|Q|=1,2,\ldots$. Thus, $D_Q(r_0,s_0)$ there exists for every $|Q|=1,2,\ldots$, and putting $d_{0,0}=D_{0,0}(r_0,s_0)=1$ we have that the function $\psi$ given by
\begin{equation}\label{eqh10}
\psi(x,y)=x^{r_0}y^{s_0}\sum^\infty_{|Q|=0}D_Q(r_0,s_0)x^{q_1}y^{q_2},\;\;\;\;D_{0,0}(r_0,s_0)=1,
\end{equation}
 is a solution of (\ref{eqh1}), if  series converges.

We need to show that the series (\ref{eqh8}) converges in the bidisc $\Delta[(0,0),(R,R)]$,
 where os $D_{Q}(r_0,s_0)$ are given recursively  by
 \begin{equation}
 \label{eqh11}
D_{0,0}(r_0,s_0)=1,
\end{equation}
\[
P(q_1+r_0,q_2+s_0)D_Q(r_0,s_0)=-\sum^{q_1-1}_{i=0}\sum^{q_2}_{j=0}
\big[(i+r_0)a_{q_1-i,q_2-j}+(j+s_0)b_{q_1-i,q_2-j}+c_{q_1-i,q_2-j}\big]
D_{i,j}(r_0,s_0)
\]
\[
-\sum^{q_2-1}_{j=0}\big[(q_1+r_0)a_{0,q_2-j}+(j+s_0)b_{0,q_2-j}+c_{0,q_2-j}\big]
D_{q_1,j}(r_0,s_0),\;|Q|=1,2,\ldots
\]

Observe that
$$P(q_1+r_0,q_2+s_0)=Aq_1^2+Cq_2^2+2q_1A\big[r_0+\frac{a(0,0)-A}{2A}\big]+2q_2C\big[s_0+\frac{b(0,0)-C}{2C}\big]$$
therefrom
$$\Vert P(q_1+r_0,q_2+s_0)\Vert\geq  \Vert Aq_1^2+Cq_2^2\Vert-2q_1\Vert A\Vert\left\Vert r_0+\frac{a(0,0)-A}{2A}\right\Vert-2q_2\Vert C\Vert\left\Vert s_0+\frac{b(0,0)-C}{2C}\right\Vert.
$$

Given that
$$\Vert Aq_1^2+Cq_2^2\Vert^2=\Vert A\Vert ^2q_1^4+\Vert C\Vert^2q_2^4+2q_1^2q_2^2\mbox{Re}(A\overline{C})$$
let  $\alpha=\min\{\Vert A\Vert^2,\Vert C\Vert^2\}>0$ and since $\mbox{Re}(A\overline{C})<0$ we have
$$\Vert Aq_1^2+Cq_2^2\Vert^2\geq  \alpha (q_1^4+q_2^4)+\mbox{Re}(A\overline{C})(q_1^2+q_2^2)$$
$$\Vert Aq_1^2+Cq_2^2\Vert^2\geq  \alpha \frac{(q_1^2+q_2^2)^2}{2}+\mbox{Re}(A\overline{C})(q_1+q_2)^2 $$
$$\Vert Aq_1^2+Cq_2^2\Vert^2\geq  \alpha \frac{(q_1+q_2)^4}{8}+\mbox{Re}(A\overline{C})(q_1+q_2)^2$$
$$\Vert Aq_1^2+Cq_2^2\Vert^2\geq  \frac{\alpha (q_1+q_2)^2}{8}\big( (q_1+q_2)^2+\frac{8}{\alpha}\mbox{Re}(A\overline{C})\big).$$
Let  $\theta=\max\{\left\Vert r_0+\frac{a(0,0)-A}{2A}\right\Vert,\left\Vert s_0+\frac{b(0,0)-C}{2C}\right\Vert\}$ and $\beta=\max\{\Vert A \Vert,\Vert C \Vert\}$. Let $n_0$ be a natural number such that $n_0^2>\frac{32\theta^2\beta^2}{\alpha}-\frac{8}{\alpha}\mbox{Re}(A\overline{C})>0$.
Thus for $|Q|=q_1+q_2\geq n_0$ we have that

$$\Vert P(q_1+r_0,q_2+s_0)\Vert \geq \frac{\sqrt{\alpha}}{2\sqrt{2}}(q_1+q_2)\sqrt{(q_1+q_2)^2+\frac{8}{\alpha}\mbox{Re}(A\overline{C})}-2\theta\beta(q_1+q_2).$$
Consequently, for $|Q|\geq n_0$ we have
\begin{equation}\label{eqh12}
\Vert P(q_1+r_0,q_2+s_0)\Vert \geq \frac{\sqrt{\alpha}}{2\sqrt{2}}|Q|\big(\sqrt{|Q|^2+\frac{8}{\alpha}\mbox{Re}(A\overline{C})} -\frac{4\sqrt{2}\theta\beta}{\sqrt{\alpha}}\big).
\end{equation}

Let $\rho$ be any number that satisfies the inequality $0<\rho<R$. Given that the series defined in (\ref{eqh3}) are convergent for $(x,y)=(\rho,\rho)$ there exists a constant $M>0$ such that

\begin{equation}\label{eqh13}
\Vert a_Q\Vert\rho^{|Q|}\leq M,\;\;\;\;\;\Vert b_Q\Vert\rho^{|Q|}\leq M\;\;\;\;\;\Vert c_Q\Vert\rho^{|Q|}\leq M\;\;\;\;\;|Q|=0,1,2,\ldots
\end{equation}
Using (\ref{eqh12}) and (\ref{eqh13}) in (\ref{eqh11}) for $|Q|\geq n_0$ we obtain
\begin{equation}\label{eqh14}
\begin{array}{l}\frac{\sqrt{\alpha}}{2\sqrt{2}}|Q|\big(\sqrt{|Q|^2+\frac{8}{\alpha}\mbox{Re}(A\overline{C})} -\frac{4\sqrt{2}\theta\beta}{\sqrt{\alpha}}\big)\Vert D_Q(r_0,s_0)\Vert\leq\\
\\ M\sum^{q_1-1}_{j=0}\sum^{q_2}_{j=0}[i+j+\Vert r_0\Vert+\Vert s_0\Vert+1]\rho^{i+j-|Q|}\Vert D_{i,j}(r_0,s_0)\Vert+M\sum^{q_2-1}_{j=0}[q_1+\Vert r_0\Vert+\Vert s_0\Vert+1]\rho^{j-q_2}\Vert D_{q_1,j}(r_0,s_0)\Vert.\end{array}
\end{equation}
Summing up all terms of norm $|Q|=n$ ($n\geq n_0$) in (\ref{eqh14}) we have
\begin{equation}\label{eqh15}
\begin{array}{l}\frac{\sqrt{\alpha}}{2\sqrt{2}}n\big(\sqrt{n^2+\frac{8}{\alpha}\mbox{Re}(A\overline{C})} -\frac{4\sqrt{2}\theta\beta}{\sqrt{\alpha}}\big)\sum_{|Q|=n}\Vert D_Q(r_0,s_0)\Vert\leq \\ \\ \hspace{5cm}
2M\sum^{n-1}_{k=0}\big(k+\Vert r_0\Vert+\Vert s_0\Vert+1\big)\rho^{k-n}\big(\sum_{|Q|=k}\Vert D_Q(r_0,s_0)\Vert\big)\end{array}.
\end{equation}

Consider $\tilde{\psi}(x)=\sum^\infty_{|Q|=0}\Vert D_Q(r_0,s_0)\Vert x^{|Q|}$ the formal power series of nonnegative power in the variable $x$. We will show that $\tilde{\psi}$ converges in $D_R[0]$, this will imply that $\psi$ converges in the bidisc $\Delta[(0,0),(R,R)]$. Let $d_n=\sum_{|Q|=n}\Vert D_Q(r_0,s_0)\Vert$ and then
$$\tilde{\psi}(x)=\sum^\infty_{n=0}\big(\sum_{|Q|=n}\Vert D_Q(r_0,s_0)\Vert\big)x^{n}=\sum^\infty_{n=0}d_nx^n.$$
Let us define $g_0,g_1,\ldots$ of the following way:
$$g_0=\Vert D_{0,0}(r_0,s_0)\Vert=1,\,\;\;g_n=\sum_{|Q|=n}\Vert D_Q(r_0,s_0)\Vert,\;\;(n=1,2,\ldots,n_0-1)$$e
\begin{equation}\label{eqh16}
\begin{array}{l}\frac{\sqrt{\alpha}}{2\sqrt{2}}n\big(\sqrt{n^2+\frac{8}{\alpha}\mbox{Re}(A\overline{C})} -\frac{4\sqrt{2}\theta\beta}{\sqrt{\alpha}}\big)g_n=2M\sum^{n-1}_{k=0}\big(k+\Vert r_0\Vert+\Vert s_0\Vert+1\big)\rho^{k-n}g_k\end{array}
\end{equation}
for $n=n_0,n_0+1,\ldots$.
Then, comparing the definition of $g_n$ with (\ref{eqh14}), we conclude that
\begin{equation}\label{eqh17}
d_n\leq g_n,\;\;\;\;\;\;n=0,1,2,\ldots
\end{equation}
Thus we will show that the series
\begin{equation}\label{eqh18}
\sum^\infty_{n=0}g_nx^n
\end{equation}
converges for $|x|<\rho$.
Replacing $n$ by $n+1$ in (\ref{eqh16}) we have:
$$\begin{array}{l}\rho\frac{\sqrt{\alpha}}{2\sqrt{2}}(n+1)\big(\sqrt{(n+1)^2+\frac{8}{\alpha}\mbox{Re}(A\overline{C})} -\frac{4\sqrt{2}\theta\beta}{\sqrt{\alpha}}\big)g_{n+1}=\\
\\ \big[\frac{\sqrt{\alpha}}{2\sqrt{2}}n\big(\sqrt{n^2+\frac{8}{\alpha}\mbox{Re}(A\overline{C})} -\frac{4\sqrt{2}\theta\beta}{\sqrt{\alpha}}\big)+2M(n+\Vert r_0\Vert+\Vert s_0\Vert+1)\big]g_n\end{array}$$
for $n=n_0,n_0+1,\ldots$. Thus
$$\left|\frac{g_{n+1}x^{n+1}}{g_{n}x^{n}}\right|=\frac{\big[\frac{\sqrt{\alpha}}{2\sqrt{2}}n\big(\sqrt{n^2+\frac{8}{\alpha}\mbox{Re}(A\overline{C})} -\frac{4\sqrt{2}\theta\beta}{\sqrt{\alpha}}\big)+2M(n+\Vert r_0\Vert+\Vert s_0\Vert+1)\big]}{\rho\frac{\sqrt{\alpha}}{2\sqrt{2}}(n+1)\big(\sqrt{(n+1)^2+\frac{8}{\alpha}\mbox{Re}(A\overline{C})} -\frac{4\sqrt{2}\theta\beta}{\sqrt{\alpha}}\big)}|x|$$ converges to  $|x|/\rho$ as $n\to\infty$. Thus by the quotient test, the series (\ref{eqh18}) converges in $|x|<\rho$. Using (\ref{eqh17}) and
by the comparison criteria, we conclude that the series
$$\sum^\infty_{n=0}d_nx^n,\;\;\;\;d_0=1,$$ converges in   $|x|<\rho$.  But given that $\rho$ is any number
that satisfies the inequality $0<\rho<R$, we have already proved that this series converges for $|x|<R$.
\end{proof}



\subsection{Non parabolic case: a sufficient convergence condition}
A general statement for a wide class of complex PDEs that includes all the real analytic
cases is given below:

\begin{Theorem}[Non parabolic convergent]
\label{Theorem:nonparabolicgeneral}
Consider the complex analytic second order PDE
\begin{equation}\label{eq3p1}
L[z]:=Ax^2\frac{\partial^2 z}{\partial x^2}+Bxy\frac{\partial^2 z}{\partial x\partial y}+Cy^2\frac{\partial^2 z}{\partial y^2}+xa(x,y)\frac{\partial z}{\partial x}+yb(x,y)\frac{\partial z}{\partial y}+c(x,y)z=0
\end{equation}
 with $A,B,C\in\mathbb{K}^*$ such that $\frac{\Vert B\Vert^2}{2}+\mbox{Re}(A\overline{C})>0$, $\mbox{Re}(A\overline{B})>0$, $\mbox{Re}(B\overline{C})>0$, $a(x,y),b(x,y)$ and $c(x,y)$ analytic in $\Delta[(0,0),(R,R)]$, $R>0$. Let $(r_0,s_0)\in \mathbb{C}^2$ be a non-resonant point of the indicial conic.
 Then there exists a convergent Frobenius type solution of (\ref{eq3p1}) with initial monomial $x^{r_0}y^{s_0}$. If $(r_0,s_0)$ and the coefficients of the PDE are real, the solution is also real.
\end{Theorem}
\begin{proof}
Let $\varphi$ be a solution  of (\ref{eq3p1}) of the form
\begin{equation}\label{eq3p2}
\varphi(x,y)=x^ry^s\sum^{\infty}_{|Q|=0}d_QX^Q
\end{equation}
where $d_{0,0}\neq0$. Given that $a(x,y),b(x,y)$ and $c(x,y)$ are holomorphic in $\Delta[(0,0),(R,R)]$ we have that
\begin{equation}\label{eq3p3}
a(x,y)=\sum^\infty_{|Q|=0}a_QX^Q,\;\;\;b(x,y)=\sum^\infty_{|Q|=0}b_QX^Q\;\;\;\mbox{ and }\;\;\;c(x,y)=\sum^\infty_{|Q|=0}c_QX^Q
\end{equation}
for every $(x,y)\in \Delta[(0,0),(R,R)]$.

Then
$$\frac{\partial\varphi}{\partial x}=\sum^\infty_{|Q|=0}(q_1+r)d_Qx^{q_1+r-1}y^{q_2+s} $$
$$\frac{\partial\varphi}{\partial y}=\sum^\infty_{|Q|=0}(q_2+s)d_Qx^{q_1+r}y^{q_2+s-1} $$
$$\frac{\partial^2\varphi}{\partial x^2}=\sum^\infty_{|Q|=0}(q_1+r)(q_1+r-1)d_Qx^{q_1+r-2}y^{q_2+s} $$
$$\frac{\partial^2\varphi}{\partial x \partial y}=\sum^\infty_{|Q|=0}(q_1+r)(q_2+s)d_Qx^{q_1+r-1}y^{q_2+s-1} $$
$$\frac{\partial^2\varphi}{\partial y^2}=\sum^\infty_{|Q|=0}(q_2+s)(q_2+s-1)d_Qx^{q_1+r}y^{q_2+s-2} $$
and from this we have
$$Ax^2\frac{\partial^2\varphi}{\partial x^2}=x^ry^s\sum^\infty_{|Q|=0}A(q_1+r)(q_1+r-1)d_QX^Q$$
$$Bxy\frac{\partial^2\varphi}{\partial x \partial y}=x^ry^s\sum^\infty_{|Q|=0}B(q_1+r)(q_2+s)d_QX^Q$$
$$Cy^2\frac{\partial^2\varphi}{\partial y^2}=x^ry^s\sum^\infty_{|Q|=0}C(q_2+s)(q_2+s-1)d_QX^Q$$

$$\begin{array}{l c l}xa(x,y)\frac{\partial \varphi}{\partial x}&=&x^ry^s\big(\sum^\infty_{|Q|=0}a_QX^Q\big)\big(\sum^\infty_{|Q|=0}(q_1+r)d_QX^Q\big)\\&&\\&=& x^ry^s\big(\sum^\infty_{|Q|=0}\tilde{a}_QX^Q\big)\end{array}$$ where $\tilde{a}_Q=\sum^{q_1}_{i=0}\sum^{q_2}_{j=0}(i+r)a_{q_1-i,q_2-j}d_{ij}$.

$$\begin{array}{l c l}yb(x,y)\frac{\partial\varphi}{\partial y}&=&x^ry^s\big(\sum^\infty_{|Q|=0}b_QX^Q\big)\big(\sum^\infty_{|Q|=0}(q_2+s)d_QX^Q\big)\\&&\\&=& x^ry^s\big(\sum^\infty_{|Q|=0}\tilde{b}_QX^Q\big)\end{array}$$ where $\tilde{b}_Q=\sum^{q_1}_{i=0}\sum^{q_2}_{j=0}(j+s)b_{q_1-i,q_2-j}d_{ij}$.

$$\begin{array}{l c l}c(x,y)\varphi&=&x^ry^s\big(\sum^\infty_{|Q|=0}c_QX^Q\big)\big(\sum^\infty_{|Q|=0}d_QX^Q\big)\\&&\\&=& x^ry^s\big(\sum^\infty_{|Q|=0}\tilde{c}_QX^Q\big)\end{array}$$ where $\tilde{c}_Q=\sum^{q_1}_{i=0}\sum^{q_2}_{j=0}c_{q_1-i,q_2-j}d_{ij}$.

Given that $\varphi$ is solution of (\ref{eq3p1}) we have
$$Ax^2\frac{\partial^2 \varphi}{\partial x^2}+Bxy\frac{\partial^2 \varphi}{\partial x\partial y}+Cy^2\frac{\partial^2 \varphi}{\partial y^2}+xa(x,y)\frac{\partial \varphi}{\partial x}+yb(x,y)\frac{\partial \varphi}{\partial y}+c(x,y)\varphi=0$$
Therefore
$$x^ry^s\sum^\infty_{|Q|=0}\big(\big[A(q_1+r)(q_1+r-1)+B(q_1+r)(q_2+s)+C(q_2+s)(q_2+s-1)\big]d_Q+\tilde{a}_Q+\tilde{b}_Q+\tilde{c}_Q\big)X^Q=0 $$
hence we have
$$\big[A(q_1+r)(q_1+r-1)+B(q_1+r)(q_2+s)+C(q_2+s)(q_2+s-1)\big]d_Q+\tilde{a}_Q+\tilde{b}_Q+\tilde{c}_Q=0,\;\;\;|Q|=0,1,2,\ldots$$
Using the definitions of $\tilde{a}_Q$, $\tilde{b}_Q$ and $\tilde{c}_Q$ we write the equation anterior as
$$\begin{array}{c}\big[A(q_1+r)(q_1+r-1)+B(q_1+r)(q_2+s)+C(q_2+s)(q_2+s-1)\big]d_Q+\\
\\ \sum^{q_1}_{i=0}\sum^{q_2}_{j=0}\big[(i+r)a_{q_1-i,q_2-j}+(j+s)b_{q_1-i,q_2-j}+c_{q_1-i,q_2-j}\big]d_{i,j}=0\end{array}$$
equivalently
$$\begin{array}{c}\big[A(q_1+r)(q_1+r-1)+B(q_1+r)(q_2+s)+C(q_2+s)(q_2+s-1)+(q_1+r)a_{0,0}+(q_2+s)b_{0,0}+c_{0,0}\big]d_Q\\
\\ +\sum^{q_1-1}_{i=0}\sum^{q_2}_{j=0}\big[(i+r)a_{q_1-i,q_2-j}+(j+s)b_{q_1-i,q_2-j}+c_{q_1-i,q_2-j}\big]d_{i,j}\\
\\+\sum^{q_2-1}_{j=0}\big[(q_1+r)a_{0,q_2-j}+(j+s)b_{0,q_2-j}+c_{0,q_2-j}\big]d_{q_1,j} =0\end{array}$$

Para $|Q|=0$ we have
\begin{equation}\label{eq3p4} Ar(r-1)+Brs+Cs(s-1)+ra_{0,0}+sb_{0,0}+c_{0,0}=0
\end{equation}
provided that $d_{0,0}\neq0$. The second degree two variables polynomial $P$ dado by $$P(r,s)=Ar^2+Brs+Cs^2+r(a_{0,0}-A)+s(b_{0,0}-C)+c_{0,0}$$
is called indicial conic associate to  equation (\ref{eq3p1}).
We see that
\begin{equation}\label{eq3p5}
P(q_1+r,q_2+s)d_Q+e_Q=0,\;\;\;\;\;|Q|=1,2,\ldots
\end{equation}
where
\begin{equation}\label{eq3p6}
\begin{array}{c c l}e_Q&=&\sum^{q_1-1}_{i=0}\sum^{q_2}_{j=0}\big[(i+r)a_{q_1-i,q_2-j}+(j+s)b_{q_1-i,q_2-j}+c_{q_1-i,q_2-j}\big]d_{i,j}\\&&\\&&
+\sum^{q_2-1}_{j=0}\big[(q_1+r)a_{0,q_2-j}+(j+s)b_{0,q_2-j}+c_{0,q_2-j}\big]d_{q_1,j},\;\;\;\;\;|Q|=1,2,\ldots\end{array}
\end{equation}
Observe that $e_Q$ is a linear combination of $d_{0,0},d_{1,0},d_{0,1},\ldots,d_{n-1,0},d_{0,n-1}$, whose coefficients are given only in terms of the already known functions $a,b,c$, $r$ and $s$. Letting  $r$, $s$ and $d_{0,0}$ undetermined, as of now we solve  equations (\ref{eq3p5}) and (\ref{eq3p6}) in terms of $d_{0,0}$, $r$ and $s$. These solutions are defined by $D_Q(r,s)$, and a $e_Q$ corresponding to $E_Q(r,s)$.
Thus
$$E_{1,0}(r,s)=(ra_{1,0}+sb_{1,0}+c_{1,0})d_{0,0}\;\;\;\;\;\;D_{1,0}(r,s)=-\frac{E_{1,0}(r,s)}{P(1+r,s)},$$
$$E_{0,1}(r,s)=(ra_{0,1}+sb_{0,1}+c_{0,1})d_{0,0}\;\;\;\;\;\;D_{0,1}(r,s)=-\frac{E_{0,1}(r,s)}{P(r,1+s)},$$
 and in general:
 \begin{equation}\label{eq3p7}
\begin{array}{c c l}E_Q(r,s)&=&\sum^{q_1-1}_{i=0}\sum^{q_2}_{j=0}\big[(i+r)a_{q_1-i,q_2-j}+(j+s)b_{q_1-i,q_2-j}+c_{q_1-i,q_2-j}\big]D_{i,j}(r,s)\\&&\\&&
+\sum^{q_2-1}_{j=0}\big[(q_1+r)a_{0,q_2-j}+(j+s)b_{0,q_2-j}+c_{0,q_2-j}\big]D_{q_1,j}(r,s),\;\;|Q|=1,2,\ldots\end{array}
\end{equation}

\begin{equation}\label{eq3p8}
D_Q(r,s)=-\frac{E_Q(r,s)}{P(q_1+r,q_2+s)},\;\;\;\;\;|Q|=1,2,\ldots
\end{equation}

The coefficients $D_Q$ obtained in this way, are rational functions of $r$ and $s$, and the only points where they are not defined, are those points $r$ and $s$ for which $P(q_1+r,q_2+s)=0$ for some $|Q|=1,2,\ldots$.  Let us define $\varphi$ by:
\begin{equation}\label{eq3p9}
\varphi((x,y),(r,s))=d_{0,0}x^ry^s+x^ry^s\sum^{\infty}_{|Q|=1}D_Q(r,s)x^{q_1}y^{q_2}.
\end{equation}
 If the series (\ref{eq3p9}) converges in $\Delta[(0,0),(R,R)]$, then we have:
\begin{equation}\label{eq3p10}
L(\varphi)((x,y),(r,s))=d_{0,0}P(r,s)x^ry^s.
\end{equation}
We have the following situation: If a $\varphi$ given by (\ref{eq3p2}) is solution of (\ref{eq3p1}), then $(r,s)$ must be a point of the indicial conic $P$, and then os $d_Q$ ($|Q|=1,2,\ldots$) are uniquely determined in terms of $d_{0,0}$, $r$ and $s$ by os $D_Q(r,s)$ from (\ref{eq3p8}), provided that $P(q_1+r,q_2+s)\neq0$, $|Q|=1,2,\ldots$.

Conversely, if $(r,s)$ is a root of $P$ and if os $D_Q(r,s)$ can be determined (i.e., $P(q_1+r,q_2+s)\neq0$ for $|Q|=1,2,\ldots$) then the function $\varphi$ given by $\varphi(x,y)=\varphi((x,y),(r,s))$ is solution of (\ref{eq3p1}) for every choice of $d_{0,0}$, provided that the series (\ref{eq3p9}) is convergent.

By hypothesis $(r_0,s_0)$ is a point of the indicial conic $P$ such that $(r_0,s_0)\notin\mathcal{R}$, then $P(q_1+r_0,q_2+s_0)\neq0$ for every $|Q|=1,2,\ldots$. Thus, $D_Q(r_0,s_0)$ there exists for every $|Q|=1,2,\ldots$, and putting $d_{0,0}=D_{0,0}(r_0,s_0)=1$ we have that the function $\psi$ given by
\begin{equation}\label{eq3p11}
\psi(x,y)=x^{r_0}y^{s_0}\sum^\infty_{|Q|=0}D_Q(r_0,s_0)x^{q_1}y^{q_2},\;\;\;\;D_{0,0}(r_0,s_0)=1,
\end{equation}
 is a solution of (\ref{eq3p1}), when the series is convergent.

We need to show that the series (\ref{eq3p11}) converges in the bidisc $\Delta[(0,0),(R,R)]$ where os $D_Q(r_0,s_0)$ are given recursively  by
\begin{equation}\label{eq3p12}
\begin{array}{c} D_{0,0}(r_0,s_0)=1,\\
\\
\begin{array}{c} P(q_1+r_0,q_2+s_0)D_Q(r_0,s_0)=-\sum^{q_1-1}_{i=0}\sum^{q_2}_{j=0}\big[(i+r_0)a_{q_1-i,q_2-j}+(j+s_0)b_{q_1-i,q_2-j}+c_{q_1-i,q_2-j}\big]D_{i,j}(r_0,s_0)\\ \\
-\sum^{q_2-1}_{j=0}\big[(q_1+r_0)a_{0,q_2-j}+(j+s_0)b_{0,q_2-j}+c_{0,q_2-j}\big]D_{q_1,j}(r_0,s_0),\;|Q|=1,2,\ldots\end{array}
\end{array}
\end{equation}
Observe that
$$P(q_1+r_0,q_2+s_0)=Aq_1^2+Bq_1q_2+Cq_2^2+2q_1A\big[r_0+\frac{B}{2A}s_0+\frac{a(0,0)-A}{2A}\big]+2q_2C\big[s_0+\frac{B}{2C}r_0+\frac{b(0,0)-C}{2C}\big]$$
therefrom
$$\Vert P(q_1+r_0,q_2+s_0)\Vert\geq  \Vert Aq_1^2+Bq_1q_2+Cq_2^2\Vert-2q_1\Vert A\Vert\left\Vert r_0+\frac{B}{2A}s_0+\frac{a(0,0)-A}{2A}\right\Vert-2q_2\Vert C\Vert\left\Vert s_0+\frac{B}{2C}r_0+\frac{b(0,0)-C}{2C}\right\Vert.
$$
Given that
$$\Vert Aq_1^2+Bq_1q_2+Cq_2^2\Vert^2=q_1^4\Vert A\Vert^2+2q_1^2q_2^2\big[\frac{\Vert B\Vert^2}{2}+\mbox{Re}(A\overline{C})\big]+q_2^4\Vert C\Vert^2+2q_1q_2\big[q_1^2\mbox{Re}(A\overline{B})+q_2^2\mbox{Re}(B\overline{C})\big]$$
be a $\alpha=\min\lbrace\Vert A\Vert^2,\frac{\Vert B\Vert^2}{2}+\mbox{Re}(A\overline{C}),\Vert C\Vert^2,\mbox{Re}(A\overline{B}),\mbox{Re}(B\overline{C})\rbrace>0$ we have
$$\Vert Aq_1^2+Bq_1q_2+Cq_2^2\Vert^2\geq  \alpha \big[q_1^4+2q_1^2q_2^2+q_2^4+2q_1q_2(q_1^2+q_2^2)\big]$$
$$\Vert Aq_1^2+Bq_1q_2+Cq_2^2\Vert^2\geq  \alpha \big[(q_1^2+q_2^2)^2+2q_1q_2(q_1^2+q_2^2)\big]$$
$$\Vert Aq_1^2+Bq_1q_2+Cq_2^2\Vert^2\geq  \alpha (q_1^2+q_2^2) \big[q_1^2+q_2^2+2q_1q_2\big]$$
$$\Vert Aq_1^2+Bq_1q_2+Cq_2^2\Vert^2\geq  \alpha (q_1^2+q_2^2)(q_1+q_2)^2\geq \frac{\alpha}{2}(q_1+q_2)^4$$
$$\Vert Aq_1^2+Bq_1q_2+Cq_2^2\Vert\geq \frac{\sqrt{\alpha}}{\sqrt{2}}(q_1+q_2)^2.$$
Let  $\theta=\max\lbrace\left\Vert r_0+\frac{B}{2A}s_0+\frac{a(0,0)-A}{2A}\right\Vert,\left\Vert s_0+\frac{B}{2C}r_0+\frac{b(0,0)-C}{2C}\right\Vert\rbrace$ and $\beta=\max\{\Vert A \Vert,\Vert C \Vert\}$ hence we have
$$\Vert P(q_1+r_0,q_2+s_0)\Vert \geq \frac{\sqrt{\alpha}}{\sqrt{2}}(q_1+q_2)^2-2\theta\beta(q_1+q_2).$$
Consequently:
\begin{equation}\label{eq3p13}
\Vert P(q_1+r_0,q_2+s_0)\Vert\geq \frac{\sqrt{\alpha}}{\sqrt{2}} (q_1+q_2)\big[(q_1+q_2)-\frac{2\sqrt{2}\theta\beta}{\sqrt{\alpha}} \big].
\end{equation}

Let $\rho$ be any number that satisfies the inequality $0<\rho<R$. Given that the series defined in (\ref{eq3p3}) are convergent for $(x,y)=(\rho,\rho)$ there exists a constant $M>0$ such that

\begin{equation}\label{eq3p14}
\Vert a_Q\Vert\rho^{|Q|}\leq M,\;\;\;\;\;\Vert b_Q\Vert\rho^{|Q|}\leq M\;\;\;\;\;\Vert c_Q\Vert\rho^{|Q|}\leq M\;\;\;\;\;|Q|=0,1,2,\ldots
\end{equation}
Using (\ref{eq3p13}) and (\ref{eq3p14}) in (\ref{eq3p12}) we obtain
\begin{equation}\label{eq3p15}
\begin{array}{c c l}\frac{\sqrt{\alpha}}{\sqrt{2}} |Q|\big[|Q|-\frac{2\sqrt{2}\theta\beta}{\sqrt{\alpha}} \big]\Vert D_Q(r_0,s_0)\Vert&\leq& M\sum^{q_1-1}_{j=0}\sum^{q_2}_{j=0}[i+j+\Vert r_0\Vert+\Vert s_0\Vert+1]\rho^{i+j-|Q|}\Vert D_{i,j}(r_0,s_0)\Vert\\&&\\&&+M\sum^{q_2-1}_{j=0}[q_1+\Vert r_0\Vert+\Vert s_0\Vert+1]\rho^{j-q_2}\Vert D_{q_1,j}(r_0,s_0)\Vert.\end{array}
\end{equation}
Summing up all terms of norm $|Q|=n$ in (\ref{eq3p15}) we have
\begin{equation}\label{eq3p16}
\frac{\sqrt{\alpha}}{\sqrt{2}} n\big[n-\frac{2\sqrt{2}\theta\beta}{\sqrt{\alpha}} \big]\sum_{|Q|=n}\Vert D_Q(r_0,s_0)\Vert \leq 2M\sum^{n-1}_{k=0}\big(k+\Vert r_0\Vert+\Vert s_0\Vert+1\big)\rho^{k-n}\big(\sum_{|Q|=k}\Vert D_Q(r_0,s_0)\Vert\big).
\end{equation}

Consider $\tilde{\psi}(x)=\sum^\infty_{|Q|=0}\Vert D_Q(r_0,s_0)\Vert x^{|Q|}$ the formal power series of nonnegative power in the variable $x$. We will show that $\tilde{\psi}$ converges in $D_R[0]$,
this will imply that $\psi$ converges in the bidisc $\Delta[(0,0),(R,R)]$. Let $d_n=\sum_{|Q|=n}\Vert D_Q(r_0,s_0)\Vert$ and then
$$\tilde{\psi}(x)=\sum^\infty_{n=0}\big(\sum_{|Q|=n}\Vert D_Q(r_0,s_0)\Vert \big)x^{n}=\sum^\infty_{n=0}d_nx^n.$$
Let $n_0$ be a natural number such that $n_0>\frac{2\sqrt{2}\theta\beta}{\sqrt{\alpha}}$ and let us define $g_0,g_1,\ldots$ of the following way:
$$g_0=\Vert D_{0,0}(r_0,s_0)\Vert=1,\,\;\;g_n=\sum_{|Q|=n}\Vert D_Q(r_0,s_0)\Vert,\;\;(n=1,2,\ldots,n_0-1)$$e
\begin{equation}\label{eq3p17}
\frac{\sqrt{\alpha}}{\sqrt{2}} n\big[n-\frac{2\sqrt{2}\theta\beta}{\sqrt{\alpha}} \big]g_n=2M\sum^{n-1}_{k=0}\big(k+\Vert r_0\Vert+\Vert s_0\Vert+1\big)\rho^{k-n}g_k
\end{equation}
for $n=n_0,n_0+1,\ldots$.
Then, comparing the definition of $g_n$ with (\ref{eq3p16}), we conclude that
\begin{equation}\label{eq3p18}
d_n\leq g_n,\;\;\;\;\;\;n=0,1,2,\ldots
\end{equation}
Thus we will show that the series
\begin{equation}\label{eq3p19}
\sum^\infty_{n=0}g_nx^n
\end{equation}
converges for $|x|<\rho$.

Replacing $n$ by $n+1$ in (\ref{eq3p17}) we have:
$$\rho\frac{\sqrt{\alpha}}{\sqrt{2}} (n+1)\big[n+1-\frac{2\sqrt{2}\theta\beta}{\sqrt{\alpha}} \big]g_{n+1}=\big(\frac{\sqrt{\alpha}}{\sqrt{2}} n\big[n-\frac{2\sqrt{2}\theta\beta}{\sqrt{\alpha}} \big]+2M(n+\Vert r_0\Vert+\Vert s_0\Vert+1)\big)g_n$$
for $n=n_0,n_0+1,\ldots$. Thus
$$\left|\frac{g_{n+1}x^{n+1}}{g_{n}x^{n}}\right|=\frac{\frac{\sqrt{\alpha}}{\sqrt{2}} n\big[n-\frac{2\sqrt{2}\theta\beta}{\sqrt{\alpha}} \big]+2M(n+\Vert r_0\Vert+\Vert s_0\Vert+1)}{\rho\frac{\sqrt{\alpha}}{\sqrt{2}} (n+1)\big[n+1-\frac{2\sqrt{2}\theta\beta}{\sqrt{\alpha}} \big]}|x|$$ converges to  $|x|/\rho$ as $n\to\infty$. Thus according to the quotient test, the series (\ref{eq3p19}) converges in $|x|<\rho$. Using (\ref{eq3p18}) and by the comparison criteria, we conclude that the series
$$\sum^\infty_{n=0}d_nx^n,\;\;\;\;d_0=1,$$ converges in   $|x|<\rho$.  But given that $\rho$ is any number that satisfies the inequality $0<\rho<R$, we have already proved that this series converges for $|x|<R$.
\end{proof}


\section{Examples: Bessel,  Airy, Hermite  partial differential equations and others}
In this section we introduce some classes of examples bringing to our framework classical
examples from the theory of ordinary differential equations. The strategy is based on Remark~\ref{Remark:restriction}. Indeed, we look for models which once restricted to
straight lines $y=tx, \, t \in \mathbb R$, are of the original classical model. For instance, the
candidate for being the Hermite partial equation must be such that once restricted to lines
$y=tx$ becomes a Hermite ordinary differential equation. The only exception is for the Bessel equation.
We shall start with:

\subsection{Bessel partial differential equations}

We have mentioned in the Introduction that one of the motivations of this work is to introduce in the
partial differential equations framework some classical models and therefore to obtain some of their
possible applications. This is the case of the
{\it Bessel equation} $x^2 y^{\prime \prime}  + x y ^\prime + (x^2 - \nu^2) y=0$, whose range of applications goes from
heat conduction, to the model of the hydrogen atom (\cite{BS,Gr}). This equation has  the origin $x=0$ as a singular point. Its solutions are called {\it Bessel functions} and the most important cases are when $\nu$ is an integer or half-integer. Bessel functions for integer $\nu$ are also known as {\it cylinder functions}  because they appear in the solution to Laplace's equation in cylindrical coordinates.
Let us now introduce what we shall understand as the equivalent of the  Bessel equation for the case of
PDEs.

\subsubsection{Bessel PDE type  I}
Consider the equation {\it Bessel PDE of type I} given by
\begin{equation}\label{eqbessel1}
x^2\frac{\partial^2 z}{\partial x^2}+2xy\frac{\partial^2 z}{\partial x\partial y}+y^2\frac{\partial^2 z}{\partial y^2}+x\frac{\partial z}{\partial x}+y\frac{\partial z}{\partial y}+ (x^2-\nu^2)z=0,
\end{equation}
where $\nu\in\mathbb{R}$.
This model is obtained by imposing that the restrictions to  lines $y=tx$ are  Bessel ordinary differential equations.
\subsubsection{Solution of Bessel type I}
Let us first consider the case $\nu=0$. Let $\varphi$ be a solution  of (\ref{eqbessel1}) of the form
\begin{equation}\label{eqbessel2}
\varphi(x,y)=x^ry^s\sum^{\infty}_{|Q|=0}d_QX^Q
\end{equation}
where $d_{0,0}\neq0$. Then
$$\frac{\partial\varphi}{\partial x}=\sum^\infty_{|Q|=0}(q_1+r)d_Qx^{q_1+r-1}y^{q_2+s} $$
$$\frac{\partial\varphi}{\partial y}=\sum^\infty_{|Q|=0}(q_2+s)d_Qx^{q_1+r}y^{q_2+s-1} $$
$$\frac{\partial^2\varphi}{\partial x^2}=\sum^\infty_{|Q|=0}(q_1+r)(q_1+r-1)d_Qx^{q_1+r-2}y^{q_2+s} $$
$$\frac{\partial^2\varphi}{\partial x \partial y}=\sum^\infty_{|Q|=0}(q_1+r)(q_2+s)d_Qx^{q_1+r-1}y^{q_2+s-1} $$
$$\frac{\partial^2\varphi}{\partial y^2}=\sum^\infty_{|Q|=0}(q_2+s)(q_2+s-1)d_Qx^{q_1+r}y^{q_2+s-2} $$

and from this we have
$$x^2\frac{\partial^2\varphi}{\partial x^2}=x^ry^s\sum^\infty_{|Q|=0}(q_1+r)(q_1+r-1)d_QX^Q$$
$$2xy\frac{\partial^2\varphi}{\partial x \partial y}=x^ry^s\sum^\infty_{|Q|=0}2(q_1+r)(q_2+s)d_QX^Q$$
$$y^2\frac{\partial^2\varphi}{\partial y^2}=x^ry^s\sum^\infty_{|Q|=0}(q_2+s)(q_2+s-1)d_QX^Q$$
$$x\frac{\partial\varphi}{\partial x}=x^ry^s\sum^\infty_{|Q|=0}(q_1+r)d_QX^Q$$
$$y\frac{\partial\varphi}{\partial y}=x^ry^s\sum^\infty_{|Q|=0}(q_2+s)d_QX^Q$$
$$x^2\varphi=x^ry^s\sum^{\infty}_{|Q|=0}d_Qx^{q_1+2}y^{q_2}=x^ry^s\sum^{\infty}_{|Q|=2}d_{q_1-2,q_2}X^Q.$$
Given that $\varphi$ is solution of (\ref{eqbessel1}) we have
$$x^2\frac{\partial^2 \varphi}{\partial x^2}+2xy\frac{\partial^2 \varphi}{\partial x\partial y}+y^2\frac{\partial^2 \varphi}{\partial y^2}+x\frac{\partial \varphi}{\partial x}+y\frac{\partial \varphi}{\partial y}+x^2\varphi=0.$$
Therefore
$$\begin{array}{c} x^ry^s\big[\sum^\infty_{|Q|=0}\big[(q_1+r)(q_1+r-1)+2(q_1+r)(q_2+s)+(q_2+s)(q_2+s-1)+(q_1+r)+(q_2+s)\big]d_QX^Q \\ \\ +\sum^{\infty}_{|Q|=2}d_{q_1-2,q_2}X^Q\big]=0
\end{array}$$

$$\sum^\infty_{|Q|=0}(q_1+q_2+r+s)^2d_QX^Q+\sum^{\infty}_{|Q|=2}d_{q_1-2,q_2}X^Q=0$$

$$\sum^\infty_{|Q|=0}(|Q|+r+s)^2d_QX^Q+\sum^{\infty}_{|Q|=2}d_{q_1-2,q_2}X^Q=0$$
$$\begin{array}{c}
(r+s)^2d_{0,0}+(1+r+s)^2\big[d_{1,0}x+d_{0,1}y\big]+\sum^{\infty}_{|Q|=2}\big[(|Q|+r+s)^2d_Q+d_{q_1-2,q_2}\big]X^Q=0\end{array}$$

and then $$(r+s)^2d_{0,0}=0,\;(1+r+s)^2d_{1,0}=0,\;(1+r+s)^2d_{0,1}=0$$
 and
$$(|Q|+r+s)^2d_Q+d_{q_1-2,q_2}=0,\;\mbox{ for every }|Q|=2,3,\ldots$$
Given that $d_{0,0}\neq0$ we have that
\begin{equation}\label{eqbessel3}
r+s=0.
\end{equation}
Let $(r,s)$ point of (\ref{eqbessel3}). Note that  $|Q|+r_0+s_0\neq 0$, for all $|Q|=1,2,\ldots$

Thus we have $d_{1,0}=d_{0,1}=0$  and
$$d_Q=-\frac{d_{q_1-2,q_2}}{|Q|^2},\;\mbox{ for every }|Q|=2,3\ldots.$$
Observe that
$$d_{1,1}=d_{0,2}=0\mbox{ and }d_{2,0}=-\frac{d_{0,0}}{2^2}$$

$$d_{1,2}=d_{0,3}=0\mbox{ and }d_{3,0}=-\frac{d_{1,0}}{3^2}=0,d_{2,1}=-\frac{d_{0,1}}{3^2}=0$$

$$\begin{array}{c} d_{1,3}=d_{0,4}=0\mbox{ and }d_{4,0}=-\frac{d_{2,0}}{4^2}=
(-1)^2\frac{d_{0,0}}{4^22^2}
,d_{3,1}=-\frac{d_{1,1}}{4^2}=0,d_{2,2}=-\frac{d_{0,2}}{4^2}=0\end{array}$$

$$\begin{array}{c} d_{1,4}=d_{0,5}=0\mbox{ and }d_{5,0}=-\frac{d_{3,0}}{5^2}=0,
d_{4,1}=-\frac{d_{2,1}}{5^2}=0,d_{3,2}=-\frac{d_{1,2}}{5^2}=0,d_{2,3}=-\frac{d_{0,3}}{5^2}=0\end{array}$$

$$\begin{array}{c}d_{1,5}=d_{0,6}=0\mbox{ and }d_{6,0}=-\frac{d_{4,0}}{6^2}=
(-1)^3\frac{d_{0,0}}{6^24^22^2},d_{5,1}=-\frac{d_{3,1}}{6^2}=0,d_{4,2}=-\frac{d_{2,2}}{6^2}=0,d_{3,3}=-\frac{d_{1,3}}{6^2}=0,
d_{2,4}=-\frac{d_{0,4}}{6^2}=0.
\end{array} $$

In general
$$d_{q_1,q_2}=0\;\;(q_1,q_2)\neq(2n,0)$$
and
$$d_{2n,0}=\frac{(-1)^n d_{0,0}}{2^{2n}(n!)^2}\;\;\;n=1,2,\ldots$$
Choosing $d_{0,0}=1$ we have that
\begin{equation}\label{eqbessel4}
\varphi(x,y)=x^{r_0}y^{s_0}\big[1+\sum^\infty_{n=1}\frac{(-1)^n x^{2n}}{2^{2n}(n!)^2}\big]
\end{equation}
 is solution of (\ref{eqbessel1}) where $(r_0,s_0)$ verifies (\ref{eqbessel3}). We observe that
 $y(x)=1+\sum^\infty_{n=1}\frac{(-1)^n x^{2n}}{2^{2n}(n!)^2}$ is exactly one of the solutions of the
 Bessel ordinary differential equation $x^2 y^{\prime \prime} + x y ^\prime + (x^2 - \nu^2)=0$ for
 $\nu=0$. Indeed, this is well-known as a first type Bessel function of order zero and denoted by
 $J_0$. Thus our solution is of the form $\vr(x,y)=x^{r_0}y^{s_0} J_0(x)$, which show that nothing striking comes from this first model, though it may be useful in applications.

\subsubsection{Bessel PDE type II}
Now we consider a more symmetric model for the Bessel PDE.
Consider the {\it Bessel PDE of type II} given by the equation
\begin{equation}\label{eq51}
x^2\frac{\partial^2 z}{\partial x^2}+y^2\frac{\partial^2 z}{\partial y^2}+x\frac{\partial z}{\partial x}+y\frac{\partial z}{\partial y}+(xy-\nu^2)z=0
\end{equation}

\subsubsection{Solution of Bessel type II}
Let us first consider the case $\nu=0$. Let $\varphi$ be a solution  of (\ref{eq51}) of the form
\begin{equation}\label{eq52}
\varphi(x,y)=x^ry^s\sum^{\infty}_{|Q|=0}d_QX^Q
\end{equation}
where $d_{0,0}\neq0$. Then
$$\frac{\partial\varphi}{\partial x}=\sum^\infty_{|Q|=0}(q_1+r)d_Qx^{q_1+r-1}y^{q_2+s} $$
$$\frac{\partial\varphi}{\partial y}=\sum^\infty_{|Q|=0}(q_2+s)d_Qx^{q_1+r}y^{q_2+s-1} $$
$$\frac{\partial^2\varphi}{\partial x^2}=\sum^\infty_{|Q|=0}(q_1+r)(q_1+r-1)d_Qx^{q_1+r-2}y^{q_2+s} $$
$$\frac{\partial^2\varphi}{\partial y^2}=\sum^\infty_{|Q|=0}(q_2+s)(q_2+s-1)d_Qx^{q_1+r}y^{q_2+s-2} $$
and from this we have
$$x^2\frac{\partial^2\varphi}{\partial x^2}=x^ry^s\sum^\infty_{|Q|=0}(q_1+r)(q_1+r-1)d_QX^Q$$
$$y^2\frac{\partial^2\varphi}{\partial y^2}=x^ry^s\sum^\infty_{|Q|=0}(q_2+s)(q_2+s-1)d_QX^Q$$
$$x\frac{\partial\varphi}{\partial x}=x^ry^s\sum^\infty_{|Q|=0}(q_1+r)d_QX^Q$$
$$y\frac{\partial\varphi}{\partial y}=x^ry^s\sum^\infty_{|Q|=0}(q_2+s)d_QX^Q$$
$$xy\varphi(x,y)=x^ry^s\sum^{\infty}_{|Q|=0}d_Qx^{q_1+1}y^{q_2+1}=x^ry^s\sum^{\infty}_{|Q|=2}d_{q_1-1,q_2-1}x^Q.$$
Given that $\varphi$ is solution of (\ref{eq51}) we have
$$x^2\frac{\partial^2 \varphi}{\partial x^2}+y^2\frac{\partial^2 \varphi}{\partial y^2}+x\frac{\partial \varphi}{\partial x}+y\frac{\partial \varphi}{\partial y}+xy\varphi=0$$
Therefore
$$x^ry^s\big[\sum^\infty_{|Q|=0}\big[(q_1+r)(q_1+r-1)+(q_2+s)(q_2+s-1)+(q_1+r)+(q_2+s)\big]d_QX^Q+\sum^{\infty}_{|Q|=2}d_{q_1-1,q_2-1}X^Q\big]=0 $$
$$\sum^\infty_{|Q|=0}\big[(q_1+r)^2+(q_2+s)^2\big]d_QX^Q+\sum^{\infty}_{|Q|=2}d_{q_1-1,q_2-1}X^Q=0$$
$$(r^2+s^2)d_{0,0}+\big[(r+1)^2+s^2)d_{1,0}x+(r^2+(s+1)^2)d_{0,1}y\big] +\sum^{\infty}_{|Q|=2}\big(\big[(q_1+r)^2+(q_2+s)^2\big]d_Q+d_{q_1-1,q_2-1}\big)X^Q=0$$
and then $$(r^2+s^2)d_{0,0}=0,\;\;\;((r+1)^2+s^2)d_{1,0}=0,\;\;\;(r^2+(s+1)^2)d_{0,1}=0$$ and
$$\big[(q_1+r)^2+(q_2+s)^2\big]d_Q+d_{q_1-1,q_2-1}=0,\;\mbox{ for every }|Q|=2,3,\ldots$$
Given that $d_{0,0}\neq0$ we have that
\begin{equation}\label{eq53}
r^2+s^2=0.
\end{equation}
Let $(r,s)$ point of (\ref{eq53}) such that
\begin{equation}\label{eq53'}
(r,s)\notin\{(r_1,s_1)\in\mathbb{C}^2;\;2q_1r_1+2q_2s_1+q_1^2+q_2^2=0,\;\mbox{ for some }|Q|=1,2,\ldots\}.
\end{equation}
Thus we have $d_{1,0}=d_{0,1}=0$ and
$$d_Q=-\frac{d_{q_1-1,q_2-1}}{2q_1r_0+2q_2s_0+q_1^2+q_2^2},\;\mbox{ for every }|Q|=2,3,\ldots.$$
Observe that
$$d_{2,0}=d_{0,2}=0\mbox{ and }d_{1,1}=-\frac{d_{0,0}}{2(r_0+s_0+1)}$$

$$d_{3,0}=d_{0,3}=0\mbox{ and }d_{2,1}=-\frac{d_{1,0}}{4r_0+2s_0+5}=0,d_{1,2}=-\frac{d_{0,1}}{2r_0+4s_0+5}=0$$
$$\begin{array}{c}d_{4,0}=d_{0,4}=0\mbox{ and }d_{3,1}=-\frac{d_{2,0}}{6r_0+2s_0+10}=0,d_{1,3}=-\frac{d_{0,2}}{2r_0+6s_0+10}=0,\\
\\
d_{2,2}=-\frac{d_{1,1}}{4(r_0+s_0+2)}=(-1)^2\frac{d_{0,0}}{2\cdot4(r_0+s_0+1)(r_0+s_0+2)}\end{array} $$

$$\begin{array}{c}d_{5,0}=d_{0,5}=0\mbox{ and }d_{4,1}=-\frac{d_{3,0}}{8r_0+2s_0+17}=0,d_{1,4}=-\frac{d_{0,3}}{2r_0+8s_0+17}=0,\\
\\
d_{3,2}=-\frac{d_{2,1}}{6r_0+4s_0+13}=0,d_{2,3}=-\frac{d_{1,2}}{4r_0+6s_0+13}=0\end{array} $$

$$\begin{array}{c}d_{6,0}=d_{0,6}=0\mbox{ and }d_{5,1}=-\frac{d_{4,0}}{10r_0+2s_0+26}=0,d_{1,5}=-\frac{d_{0,4}}{2r_0+10s_0+26}=0,\\
\\
d_{4,2}=-\frac{d_{3,1}}{8r_0+4s_0+20}=0,d_{2,4}=-\frac{d_{1,3}}{4r_0+8s_0+20}=0\\
\\
d_{3,3}=-\frac{d_{2,2}}{6(r_0+s_0+3)}=(-1)^3\frac{d_{0,0}}{2\cdot4\cdot 6(r_0+s_0+1)(r_0+s_0+2)(r_0+s_0+3)}
\end{array} $$
In general
$$d_{q_1,q_2}=0\;\;q_1\neq q_2$$ and
$$d_{n,n}=\frac{(-1)^nd_{0,0}}{2^n(n!)(r_0+s_0+1)(r_0+s_0+2)\ldots(r_0+s_0+n)}\;\;\;n=0,1,2,\ldots$$
Choosing $d_{0,0}=1$ we have that
\begin{equation}\label{eq54}
\varphi(x,y)=x^{r_0}y^{s_0}\big[1+\sum^\infty_{n=1}\frac{(-1)^n(xy)^n}{2^n(n!)(r_0+s_0+1)(r_0+s_0+2)\ldots(r_0+s_0+n)}\big]
\end{equation}
 is solution of (\ref{eq51}) where $(r_0,s_0)$ verifies (\ref{eq53}) and (\ref{eq53'}).
 Equation (\ref{eq51}) will be called \textit{Second order Bessel PDE}, due to the similarity of the  solution (\ref{eq54})  with the order zero Bessel functions for second order ordinary differential equations.

 \subsubsection{Bessel functions}

 Let us study more the solutions of the Bessel PDE of type II.
 Recall that the ordinary Bessel equation $x^2 y^{\prime \prime}  + x y ^\prime + (x^2 - \nu^2) y=0$
has two important cases: $\nu=0$ and $\nu=1/2$. The case $\nu=0$ has one solution of the form
$y_1(x)=a_0[1 +\sum\limits_{m=1}^\infty \frac{(-1)^m x^{2m}}{2^{2m} (m!)^2}]$ defined for $x>0$.

Let us take a sample of our solutions. For $r=s=0$ there is no resonance in the Bessel type II and we have a convergent solution
of the form
\begin{equation}\label{Besselsolution}
\varphi(x,y)=1+\sum^\infty_{n=1}\frac{(-1)^n(xy)^n}{2^n(n!)(1)(2)\ldots(n)}=
1+\sum^\infty_{n=1}\frac{(-1)^n(xy)^n}{2^n(n!)^2}
\end{equation}
Unlike in the case of type I, the solution of Bessel type II  is not straight connected to the Bessel function $J_0$, providing a different type of object for study.

\subsection{Airy partial differential equations}

In ordinary differential equations we have the {\it Airy equation} $y^{\prime \prime} - x y =0$.
This equation named after  George Biddell Airy British astronomer, appears
  quite frequently in problems connected with optics. The Airy equation describes some very interesting phenomena. It is well-known as the  simplest second-order linear differential equation with a {\it turning point} (a point where the character of the solutions changes from oscillatory to exponential). There are two linearly independent  solutions of the Airy equation and they  are called
  {\it Airy functions}. The Airy functions are  solutions to Schrödinger's equation for a particle confined within a triangular potential well and for a particle in a one-dimensional constant force field. They are also used in the approximation near a turning point in mathematical physics.
   Next we introduce an equivalent to Airy equation for partial differential equations and calculate its solutions.

\subsubsection{Type I Airy PDE} Consider the {\it Airy partial differential equation of type I} as given by
\begin{equation}\label{eqairy1}
x^2\frac{\partial^2 z}{\partial x^2}+2xy\frac{\partial^2 z}{\partial x\partial y}+ y^2\frac{\partial^2 z}{\partial y^2}-x^3z=0.
\end{equation}

 This is a parabolic
PDE with a regular singularity at the origin. Notice that for each line $y=\lambda x$ we have a copy of the Airy equation.  This seems quite surprising since the original Airy ordinary differential equation is not singular at the origin.

\subsubsection{Solution of Airy type I}
Let us now calculate the solutions of this equation.
Let $\varphi$ be a solution  of (\ref{eqairy1}) of the form
\begin{equation}\label{eqairy2}
\varphi(x,y)=x^ry^s\sum^{\infty}_{|Q|=0}d_QX^Q
\end{equation}
where $d_{0,0}\neq0$. Then
$$\frac{\partial^2\varphi}{\partial x^2}=\sum^\infty_{|Q|=0}(q_1+r)(q_1+r-1)d_Qx^{q_1+r-2}y^{q_2+s} $$
$$\frac{\partial^2\varphi}{\partial x \partial y}=\sum^\infty_{|Q|=0}(q_1+r)(q_2+s)d_Qx^{q_1+r-1}y^{q_2+s-1} $$
$$\frac{\partial^2\varphi}{\partial y^2}=\sum^\infty_{|Q|=0}(q_2+s)(q_2+s-1)d_Qx^{q_1+r}y^{q_2+s-2} $$
and from this we have
$$x^2\frac{\partial^2\varphi}{\partial x^2}=x^ry^s\sum^\infty_{|Q|=0}(q_1+r)(q_1+r-1)d_QX^Q$$
$$2xy\frac{\partial^2\varphi}{\partial x \partial y}=x^ry^s\sum^\infty_{|Q|=0}2(q_1+r)(q_2+s)d_QX^Q$$
$$y^2\frac{\partial^2\varphi}{\partial y^2}=x^ry^s\sum^\infty_{|Q|=0}(q_2+s)(q_2+s-1)d_QX^Q$$
$$x^3\varphi(x,y)=x^ry^s\sum^{\infty}_{|Q|=0}d_Qx^{q_1+3}y^{q_2}=x^ry^s\sum^{\infty}_{|Q|=3}d_{q_1-3,q_2}X^Q.$$
Given that $\varphi$ is solution of (\ref{eqairy2}) we have
$$x^2\frac{\partial^2 \varphi}{\partial x^2}+2xy\frac{\partial^2 \varphi}{\partial x\partial y}+ y^2\frac{\partial^2 \varphi}{\partial y^2}-x^3\varphi=0.$$
Therefore
$$x^ry^s\big[\sum^\infty_{|Q|=0}\big[(q_1+r)(q_1+r-1)+2(q_1+r)(q_2+s)+(q_2+s)(q_2+s-1)\big]d_QX^Q+\sum^{\infty}_{|Q|=3}d_{q_1-3,q_2}X^Q\big]=0 $$
$$\sum^\infty_{|Q|=0}\big[(q_1+q_2+r+s)^2-(q_1+q_2+r+s)\big]d_QX^Q+\sum^{\infty}_{|Q|=3}d_{q_1-3,q_2}X^Q=0$$
$$\sum^\infty_{|Q|=0}\big[(|Q|+r+s)^2-(|Q|+r+s)\big]d_QX^Q+\sum^{\infty}_{|Q|=3}d_{q_1-3,q_2}X^Q=0$$
$$\begin{array}{c}
(r+s)(r+s-1)d_{0,0}+(1+r+s)(r+s)\big[d_{1,0}x+d_{0,1}y\big]+(2+r+s)(1+r+s)\big[d_{2,0}x^2+d_{1,1}xy+d_{0,2}y^2\big]\\
\\+\sum^{\infty}_{|Q|=3}\big[(|Q|+r+s)(|Q|+r+s-1)d_Q+d_{q_1-3,q_2}\big]X^Q=0\end{array}$$
and then $$(r+s)(r+s-1)d_{0,0}=0,\;(1+r+s)(r+s)d_{1,0}=0,\;(1+r+s)(r+s)d_{0,1}=0,$$
$$(2+r+s)(1+r+s)d_{2,0}=0,\;(2+r+s)(1+r+s)d_{1,1}=0,\;(2+r+s)(1+r+s)d_{0,2}=0,$$
 and
$$(|Q|+r+s)(|Q|+r+s-1)d_Q+d_{q_1-3,q_2}=0,\;\mbox{ for every }|Q|=3,4,\ldots$$
Given that $d_{0,0}\neq0$ we have that
\begin{equation}\label{eqairy3}
(r+s)(r+s-1)=0.
\end{equation}

Let $(r,s)$ point of (\ref{eqairy3}) such that
\begin{equation}\label{eqairy4}
(r,s)\notin\mathcal C= \{(r_1,s_1)\in\mathbb{C}^2;\;(|Q|+r_1+s_1)(|Q|+r_1+s_1-1)=0,\;\mbox{ for some }|Q|=1,2,\ldots\}.
\end{equation}

Thus we have $d_{1,0}=d_{0,1}=0$, $d_{2,0}=d_{1,1}=d_{0,2}=0$  and
$$d_Q=\frac{d_{q_1-3,q_2}}{(|Q|+r+s)(|Q|+r+s-1)},\;\mbox{ for every }|Q|=3,4\ldots.$$
Observe that
$$d_{2,1}=d_{1,2}=d_{0,3}=0\mbox{ and }d_{3,0}=\frac{d_{0,0}}{(3+r+s)(2+r+s)}$$

$$d_{2,2}=d_{1,3}=d_{0,4}=0\mbox{ and }d_{4,0}=\frac{d_{1,0}}{(4+r+s)(3+r+s)}=0,d_{3,1}=\frac{d_{0,1}}{(4+r+s)(3+r+s)}=0$$

$$\begin{array}{c} d_{2,3}=d_{1,4}=d_{0,5}=0\mbox{ and }d_{5,0}=\frac{d_{2,0}}{(5+r+s)(4+r+s)}=0,\\
\\
d_{4,1}=\frac{d_{3,1}}{(5+r+s)(4+r+s)}=0,d_{3,2}=\frac{d_{0,2}}{(5+r+s)(4+r+s)}=0\end{array}$$

$$\begin{array}{c} d_{2,4}=d_{1,5}=d_{0,6}=0\mbox{ and }d_{6,0}=\frac{d_{3,0}}{(6+r+s)(5+r+s)}=\frac{d_{0,0}}{(6+r+s)(5+r+s)(3+r+s)(2+r+s)}
,\\
\\
d_{5,1}=\frac{d_{2,1}}{(6+r+s)(5+r+s)}=0,d_{4,2}=\frac{d_{1,2}}{(6+r+s)(5+r+s)}=0,d_{3,3}=\frac{d_{0,3}}{(6+r+s)(5+r+s)}=0\end{array}$$

$$\begin{array}{c}d_{2,5}=d_{1,6}=d_{0,7}=0\mbox{ and }d_{7,0}=\frac{d_{4,0}}{(7+r+s)(6+r+s)}=0,d_{6,1}=\frac{d_{3,1}}{(7+r+s)(6+r+s)}=0,\\
\\
d_{5,2}=\frac{d_{2,2}}{(7+r+s)(6+r+s)}=0,d_{4,3}=\frac{d_{1,3}}{(7+r+s)(6+r+s)}=0,
d_{3,4}=\frac{d_{0,4}}{(7+r+s)(6+r+s)}=0
\end{array} $$

$$\begin{array}{c}d_{2,6}=d_{1,7}=d_{0,8}=0\mbox{ and }d_{8,0}=\frac{d_{5,0}}{(8+r+s)(7+r+s)}=0,d_{7,1}=\frac{d_{4,1}}{(8+r+s)(7+r+s)}=0,\\
\\
d_{6,2}=\frac{d_{3,2}}{(8+r+s)(7+r+s)}=0,d_{5,3}=\frac{d_{2,3}}{(8+r+s)(7+r+s)}=0,
d_{4,4}=\frac{d_{1,4}}{(8+r+s)(7+r+s)}=0,d_{3,5}=\frac{d_{0,5}}{(8+r+s)(7+r+s)}=0
\end{array} $$

$$\begin{array}{c} d_{2,7}=d_{1,8}=d_{0,9}=0\mbox{ and }d_{9,0}=\frac{d_{6,0}}{(9+r+s)(8+r+s)}=\frac{d_{0,0}}{(9+r+s)(8+r+s)(6+r+s)(5+r+s)(3+r+s)(2+r+s)}
,\\
\\
d_{8,1}=\frac{d_{5,1}}{(9+r+s)(8+r+s)}=0,d_{7,2}=\frac{d_{4,2}}{(9+r+s)(8+r+s)}=0,d_{6,3}=\frac{d_{3,3}}{(9+r+s)(8+r+s)}=0\\
\\
d_{5,4}=\frac{d_{2,4}}{(9+r+s)(8+r+s)}=0,d_{4,5}=\frac{d_{1,5}}{(9+r+s)(8+r+s)}=0,d_{3,6}=\frac{d_{0,6}}{(9+r+s)(8+r+s)}=0.
\end{array}$$

In general
$$d_{q_1,q_2}=0\;\;(q_1,q_2)\neq(3n,0)$$
and
$$d_{3n,0}=\frac{d_{0,0}}{(2+r+s)(3+r+s)\cdots(3n-1+r+s)(3n+r+s)}\;\;\;n=1,2,\ldots$$
Choosing $d_{0,0}=1$ we have that
\begin{equation}\label{eqairy5}
\varphi(x,y)=x^{r}y^{s}\big[1+\sum^\infty_{n=1}\frac{x^{3n}}{(2+r+s)(3+r+s)\cdots(3n-1+r+s)(3n+r+s)}\big]
\end{equation}
 is solution of (\ref{eqairy1}) where $(r,s)$ verifies (\ref{eqairy3}) and (\ref{eqairy4}).

 \subsubsection{Type II Airy PDE}

Consider the {\it Airy partial differential equation of type II} as given by
\begin{equation}\label{2eqairy1}
x^2\frac{\partial^2 z}{\partial x^2}+2xy\frac{\partial^2 z}{\partial x\partial y}+ y^2\frac{\partial^2 z}{\partial y^2}-x^2yz=0.
\end{equation}

This is the second possible type (up to interchanging $x$ and $y$) for the Airy PDE in order to
have the restrictions to lines $y=tx$ also given by Airy equations.

\subsubsection{Solution of Airy type II}
Let $\varphi$ be a solution  of (\ref{2eqairy1}) of the form
\begin{equation}\label{2eqairy2}
\varphi(x,y)=x^ry^s\sum^{\infty}_{|Q|=0}d_QX^Q
\end{equation}
where $d_{0,0}\neq0$. Then
$$\frac{\partial^2\varphi}{\partial x^2}=\sum^\infty_{|Q|=0}(q_1+r)(q_1+r-1)d_Qx^{q_1+r-2}y^{q_2+s} $$
$$\frac{\partial^2\varphi}{\partial x \partial y}=\sum^\infty_{|Q|=0}(q_1+r)(q_2+s)d_Qx^{q_1+r-1}y^{q_2+s-1} $$
$$\frac{\partial^2\varphi}{\partial y^2}=\sum^\infty_{|Q|=0}(q_2+s)(q_2+s-1)d_Qx^{q_1+r}y^{q_2+s-2} $$
and from this we have
$$x^2\frac{\partial^2\varphi}{\partial x^2}=x^ry^s\sum^\infty_{|Q|=0}(q_1+r)(q_1+r-1)d_QX^Q$$
$$2xy\frac{\partial^2\varphi}{\partial x \partial y}=x^ry^s\sum^\infty_{|Q|=0}2(q_1+r)(q_2+s)d_QX^Q$$
$$y^2\frac{\partial^2\varphi}{\partial y^2}=x^ry^s\sum^\infty_{|Q|=0}(q_2+s)(q_2+s-1)d_QX^Q$$
$$x^2y\varphi(x,y)=x^ry^s\sum^{\infty}_{|Q|=0}d_Qx^{q_1+2}y^{q_2+1}=x^ry^s\sum^{\infty}_{|Q|=3}d_{q_1-2,q_2-1}X^Q.$$
Given that $\varphi$ is solution of (\ref{2eqairy2}) we have
$$x^2\frac{\partial^2 \varphi}{\partial x^2}+2xy\frac{\partial^2 \varphi}{\partial x\partial y}+ y^2\frac{\partial^2 \varphi}{\partial y^2}-x^2y\varphi=0.$$
Therefore
$$x^ry^s\big[\sum^\infty_{|Q|=0}\big[(q_1+r)(q_1+r-1)+2(q_1+r)(q_2+s)+(q_2+s)(q_2+s-1)\big]d_QX^Q-\sum^{\infty}_{|Q|=3}d_{q_1-2,q_2-1}X^Q\big]=0 $$
$$\sum^\infty_{|Q|=0}\big[(q_1+q_2+r+s)^2-(q_1+q_2+r+s)\big]d_QX^Q-\sum^{\infty}_{|Q|=3}d_{q_1-3,q_2-1}X^Q=0$$
$$\sum^\infty_{|Q|=0}\big[(|Q|+r+s)^2-(|Q|+r+s)\big]d_QX^Q-\sum^{\infty}_{|Q|=3}d_{q_1-2,q_2-1}X^Q=0$$
$$\begin{array}{c}
(r+s)(r+s-1)d_{0,0}+(1+r+s)(r+s)\big[d_{1,0}x+d_{0,1}y\big]+(2+r+s)(1+r+s)\big[d_{2,0}x^2+d_{1,1}xy+d_{0,2}y^2\big]\\
\\+\sum^{\infty}_{|Q|=3}\big[(|Q|+r+s)(|Q|+r+s-1)d_Q-d_{q_1-2,q_2-1}\big]X^Q=0\end{array}$$
and then $$(r+s)(r+s-1)d_{0,0}=0,\;(1+r+s)(r+s)d_{1,0}=0,\;(1+r+s)(r+s)d_{0,1}=0,$$
$$(2+r+s)(1+r+s)d_{2,0}=0,\;(2+r+s)(1+r+s)d_{1,1}=0,\;(2+r+s)(1+r+s)d_{0,2}=0,$$
 and
$$(|Q|+r+s)(|Q|+r+s-1)d_Q-d_{q_1-2,q_2-1}=0,\;\mbox{ for every }|Q|=3,4,\ldots$$
Given that $d_{0,0}\neq0$ we have that
\begin{equation}\label{2eqairy3}
(r+s)(r+s-1)=0.
\end{equation}
Let $(r,s)$ point of (\ref{2eqairy3}) such that
\begin{equation}\label{2eqairy4}
(r,s)\notin\{(r_1,s_1)\in\mathbb{C}^2;\;(|Q|+r_1+s_1)(|Q|+r_1+s_1-1)=0,\;\mbox{ for some }|Q|=1,2,\ldots\}.
\end{equation}
Thus we have $d_{1,0}=d_{0,1}=0$, $d_{2,0}=d_{1,1}=d_{0,2}=0$  and
$$d_Q=\frac{d_{q_1-2,q_2-1}}{(|Q|+r+s)(|Q|+r+s-1)},\;\mbox{ for every }|Q|=3,4\ldots.$$
Observe that
$$d_{3,0}=d_{1,2}=d_{0,3}=0\mbox{ and }d_{2,1}=\frac{d_{0,0}}{(3+r+s)(2+r+s)}$$

$$d_{4,0}=d_{0,4}=d_{1,3}= 0\mbox{ and }d_{3,1}=\frac{d_{1,0}}{(4+r+s)(3+r+s)}=0,d_{2,2}=\frac{d_{0,1}}{(4+r+s)(3+r+s)}=0$$

$$\begin{array}{c} d_{5,0}=d_{1,4}=d_{0,5}=0\mbox{ and }d_{4,1}=\frac{d_{2,0}}{(5+r+s)(4+r+s)}=0,\\
\\
d_{3,2}=\frac{d_{1,1}}{(5+r+s)(4+r+s)}=0,d_{2,3}=\frac{d_{0,2}}{(5+r+s)(4+r+s)}=0\end{array}$$

$$\begin{array}{c} d_{6,0}=d_{1,5}=d_{0,6}=0\mbox{ and }d_{4,2}=\frac{d_{2,1}}{(6+r+s)(5+r+s)}=\frac{d_{0,0}}{(6+r+s)(5+r+s)(3+r+s)(2+r+s)}
,\\
\\
d_{5,1}=\frac{d_{3,0}}{(6+r+s)(5+r+s)}=0,d_{3,3}=\frac{d_{1,2}}{(6+r+s)(5+r+s)}=0,d_{2,4}=\frac{d_{0,3}}{(6+r+s)(5+r+s)}=0\end{array}$$

$$\begin{array}{c}d_{7,0}=d_{1,6}=d_{0,7}=0\mbox{ and }d_{6,1}=\frac{d_{4,0}}{(7+r+s)(6+r+s)}=0,d_{5,2}=\frac{d_{3,1}}{(7+r+s)(6+r+s)}=0,\\
\\
d_{4,3}=\frac{d_{2,2}}{(7+r+s)(6+r+s)}=0,d_{3,4}=\frac{d_{1,3}}{(7+r+s)(6+r+s)}=0,
d_{2,5}=\frac{d_{0,4}}{(7+r+s)(6+r+s)}=0
\end{array} $$

$$\begin{array}{c}d_{8,0}=d_{1,7}=d_{0,8}=0\mbox{ and }d_{7,1}=\frac{d_{5,0}}{(8+r+s)(7+r+s)}=0,d_{6,2}=\frac{d_{4,1}}{(8+r+s)(7+r+s)}=0,\\
\\
d_{5,3}=\frac{d_{3,2}}{(8+r+s)(7+r+s)}=0,d_{4,4}=\frac{d_{2,3}}{(8+r+s)(7+r+s)}=0,
d_{3,5}=\frac{d_{1,4}}{(8+r+s)(7+r+s)}=0,d_{2,6}=\frac{d_{0,5}}{(8+r+s)(7+r+s)}=0
\end{array} $$

$$\begin{array}{c} d_{9,0}=d_{1,8}=d_{0,9}=0\mbox{ and }d_{6,3}=\frac{d_{4,2}}{(9+r+s)(8+r+s)}=\frac{d_{0,0}}{(9+r+s)(8+r+s)(6+r+s)(5+r+s)(3+r+s)(2+r+s)}
,\\
\\
d_{8,1}=\frac{d_{6,0}}{(9+r+s)(8+r+s)}=0,d_{7,2}=\frac{d_{5,1}}{(9+r+s)(8+r+s)}=0,d_{5,4}=\frac{d_{3,3}}{(9+r+s)(8+r+s)}=0\\
\\
d_{4,5}=\frac{d_{2,4}}{(9+r+s)(8+r+s)}=0,d_{3,6}=\frac{d_{1,5}}{(9+r+s)(8+r+s)}=0,d_{2,7}=\frac{d_{0,6}}{(9+r+s)(8+r+s)}=0.
\end{array}$$

In general
$$d_{q_1,q_2}=0\;\;(q_1,q_2)\neq(2n,n)$$
and
$$d_{2n,n}=\frac{d_{0,0}}{(2+r+s)(3+r+s)\cdots(3n-1+r+s)(3n+r+s)}\;\;\;n=1,2,\ldots$$
Choosing $d_{0,0}=1$ we have that
\begin{equation}\label{2eqairy5bis}
\varphi(x,y)=x^{r}y^{s}\big[1+\sum^\infty_{n=1}\frac{x^{2n}y^n}{(2+r+s)(3+r+s)\cdots(3n-1+r+s)(3n+r+s)}\big]
\end{equation}
 is solution of (\ref{2eqairy1}) where $(r_0,s_0)$ verifies (\ref{2eqairy3}) and (\ref{2eqairy4}).

\subsubsection{Airy functions}

 Let us first  investigate the non-resonance conditions \eqref{eqhermite3} and \eqref{eqhermite4} above. If $(r,s)\in \mathcal C$ then
$r+s=1$ or $r+s=0$. Assume that $r+s=0$. In this case the resonance condition becomes
$|Q|(|Q|-1)=0$ for some multi-index $Q=(q_1,q_2)$ with $q_1 + q_2 \geq 1$.
Therefore it becomes $|Q|=1$. Thus we may have resonance in this case.
On the other hand if $r+s=1$ then the resonance condition becomes
$(|Q|+1)|Q|=0$ and therefore it becomes $|Q|+1=0$. This does not occur and we do not have resonance in this case. For instance we can take $r=s=1/2$ and have no resonance.
In this case $r+s=1$ and we have a solution of the form
\begin{equation}\label{2eqairy6}
\varphi_2(x,y)=(xy)^{\frac{1}{2}}\big[1+\sum^\infty_{n=1}\frac{x^{2n}y^n}{(3)(4)\cdots(3n)(3n+1)}\big]
\end{equation}
This will called {\it Airy function of type II in two variables.}

Similarly we have the {\it Airy function of type I in two variables} given

\begin{equation}\label{eqairy5tri}
\varphi_{1}(x,y)=(xy)^{\frac{1}{2}}\big[1+\sum^\infty_{n=1}\frac{x^{3n}}{(3)(4)\cdots(3n)(3n+1)}\big]
\end{equation}
which  is a solution of (\ref{eqairy1}).

Recall that the ordinary Airy equation is $y^{\prime \prime } - xy=0$ which has
general solution  of the form $y= a_0 y_1 (x) + a_1 y_2(x)$ for
\[
y_1(x)=1 +\sum\limits_{n=1}^\infty \frac{x^{3n}}{2.3...(3n-1)(3n)}
\]
and
\[
y_2(x)=x+ \sum\limits_{n=1}^\infty \frac{x^{3n+1}}{3.4...(3n)(3n+1)}=x[1+ \sum\limits_{n=1}^\infty \frac{x^{3n}}{3.4...(3n)(3n+1)}]
\]

The functions $y_1(x), y_2(x)$ are called {\it first} and {\it second Airy functions}.

Let us write $\xi(t):=1+ \sum\limits_{n=1}^\infty \frac{t^{n}}{3.4...(3n)(3n+1)}$. Then
$y_2(x)=x.\xi(x^3)$ and $\xi(t)=\frac{y_2(t^{\frac{1}{3}})}{t^{\frac{1}{3}}}$. Then we can write
\[
\vr_1(x,y)=(xy)^{\frac{1}{2}}\xi(x^3)=(xy)^{\frac{1}{2}} \frac{y_2(x)}{x}= \sqrt{\frac{y}{x}}.y_2(x)
\]
 As for $\vr_2(x,y)$ we observe that
 $\vr_2(x,y)= \sqrt{xy}\xi(x^2y)$ and therefore
 \[
 \vr_2(x,y)=\sqrt{xy}\frac{y_2( (x^2y)^{\frac{1}{3}})}{(x^2y)^{\frac{1}{3}}}=(\frac{y}{x})^{\frac{1}{6}}y_2((x^2y)^{\frac{1}{3}}).
 \]
\subsection{Hermite equation}
 In ordinary differential equations the {\it Hermite equation} is
the second-order ordinary differential equation
$y^{\prime \prime} -2xy^{\prime} +\lambda y=0.$ The constant $\lambda$ is complex or real.
This equation has an irregular singularity at $x=\infty$. The solutions by series are defined
for $|x|<\infty$. If $\lambda=2n$ for some natural number $n\in \mathbb N$
 the Hermite equation has among its solutions the {\it Hermite polynomial} of degree $n$
 given by $H_n(x)=(-1)^ne^{x^2}\frac{d^n}{dz^n}(e^{-x^2})$.
 Hermite's differential equation appears during the solution of the Schrödinger
equation for the harmonic oscillator.
The Hermite polynomials are mutually orthogonal with respect to the density or weight function $e^{-x^2}$.
This allows the use of Hermite polynomials as generating functions in several problems in mathematical-physics. Also in probability theory the use of Hermite functions is quite important.

\subsubsection{Hermite PDE type I}
Let us now proceed to introduce what we shall understand as a first model for the  {\it Hermite partial differential equation}.
 Consider the equation
\begin{equation}\label{eqhermite1}
x^2\frac{\partial^2 z}{\partial x^2}+2xy\frac{\partial^2 z}{\partial x\partial y}+y^2\frac{\partial^2 z}{\partial y^2}-2x^3\frac{\partial z}{\partial x}-2x^2y\frac{\partial z}{\partial y}+\lambda x^2z=0,
\end{equation}
where $\lambda\in\mathbb{R}$.

The restriction of this equation to lines $y=tx$ is a Hermite ordinary differential equation.
Again, this is somehow unexpected, since the PDE has a singularity at the origin and the Hermite ODE
is non-singular at the origin.
\subsubsection{Solution of  Hermite  type I}
Let $\varphi$ be a solution  of (\ref{eqhermite1}) of the form
\begin{equation}\label{eqhermite2}
\varphi(x,y)=x^ry^s\sum^{\infty}_{|Q|=0}d_QX^Q
\end{equation}
where $d_{0,0}\neq0$.

Then
$$\frac{\partial\varphi}{\partial x}=\sum^\infty_{|Q|=0}(q_1+r)d_Qx^{q_1+r-1}y^{q_2+s} $$
$$\frac{\partial\varphi}{\partial y}=\sum^\infty_{|Q|=0}(q_2+s)d_Qx^{q_1+r}y^{q_2+s-1} $$
$$\frac{\partial^2\varphi}{\partial x^2}=\sum^\infty_{|Q|=0}(q_1+r)(q_1+r-1)d_Qx^{q_1+r-2}y^{q_2+s} $$
$$\frac{\partial^2\varphi}{\partial x \partial y}=\sum^\infty_{|Q|=0}(q_1+r)(q_2+s)d_Qx^{q_1+r-1}y^{q_2+s-1} $$
$$\frac{\partial^2\varphi}{\partial y^2}=\sum^\infty_{|Q|=0}(q_2+s)(q_2+s-1)d_Qx^{q_1+r}y^{q_2+s-2} $$

and from this we have
$$x^2\frac{\partial^2\varphi}{\partial x^2}=x^ry^s\sum^\infty_{|Q|=0}(q_1+r)(q_1+r-1)d_QX^Q$$
$$2xy\frac{\partial^2\varphi}{\partial x \partial y}=x^ry^s\sum^\infty_{|Q|=0}2(q_1+r)(q_2+s)d_QX^Q$$
$$y^2\frac{\partial^2\varphi}{\partial y^2}=x^ry^s\sum^\infty_{|Q|=0}(q_2+s)(q_2+s-1)d_QX^Q$$
$$-2x^3\frac{\partial\varphi}{\partial x}=x^ry^s\sum^\infty_{|Q|=0}-2(q_1+r)d_Qx^{q_1+2}y^{q_2}=x^ry^s\sum^{\infty}_{|Q|=2}-2(q_1+r-2)d_{q_1-2,q_2}X^Q $$
$$-2x^2y\frac{\partial\varphi}{\partial y}=x^ry^s\sum^\infty_{|Q|=0}(q_2+s)d_Qx^{q_1+2}y^{q_2}=x^ry^s\sum^{\infty}_{|Q|=2}-2(q_2+s)d_{q_1-2,q_2}X^Q $$
$$\lambda x^2\varphi=x^ry^s\sum^{\infty}_{|Q|=0}\lambda d_Qx^{q_1+2}y^{q_2}=x^ry^s\sum^{\infty}_{|Q|=2}\lambda d_{q_1-2,q_2}X^Q.$$
Given that $\varphi$ is solution of (\ref{eqhermite1}) we have
$$x^2\frac{\partial^2 \varphi}{\partial x^2}+2xy\frac{\partial^2 \varphi}{\partial x\partial y}+y^2\frac{\partial^2 \varphi}{\partial y^2}-2x^3\frac{\partial \varphi}{\partial x}-2x^2y\frac{\partial \varphi}{\partial y}+\lambda x^2\varphi=0.$$
Therefore
$$\begin{array}{c} x^ry^s\big[\sum^\infty_{|Q|=0}\big[(q_1+r)(q_1+r-1)+2(q_1+r)(q_2+s)+(q_2+s)(q_2+s-1)\big]d_QX^Q \\ \\ +\sum^{\infty}_{|Q|=2}[-2(q_1+r-2)-2(q_2+s)+\lambda]d_{q_1-2,q_2}X^Q\big]=0
\end{array}$$
$$\sum^\infty_{|Q|=0}\big[(q_1+q_2+r+s)^2-(q_1+q_2+r+s)\big]d_QX^Q+\sum^{\infty}_{|Q|=2}[-2(q_1+r-2)-2(q_2+s)+\lambda]d_{q_1-2,q_2}X^Q=0$$
$$\sum^\infty_{|Q|=0}\big[(|Q|+r+s)^2-(|Q|+r+s)\big]d_QX^Q+\sum^{\infty}_{|Q|=2}[-2(|Q|+r+s)+4+\lambda]d_{q_1-2,q_2}X^Q=0$$
$$\begin{array}{c}
(r+s)(r+s-1)d_{0,0}+(1+r+s)(r+s)\big[d_{1,0}x+d_{0,1}y\big]+\\
\\+\sum^{\infty}_{|Q|=2}\big[(|Q|+r+s)(|Q|+r+s-1)d_Q+[-2(|Q|+r+s)+4+\lambda]d_{q_1-2,q_2}\big]X^Q=0\end{array}$$
and then $$(r+s)(r+s-1)d_{0,0}=0,\;(1+r+s)(r+s)d_{1,0}=0,\;(1+r+s)(r+s)d_{0,1}=0$$
 and
$$(|Q|+r+s)(|Q|+r+s-1)d_Q+[-2(|Q|+r+s)+4+\lambda]d_{q_1-2,q_2}=0,\;\mbox{ for every }|Q|=2,3,\ldots$$
Given that $d_{0,0}\neq0$ we have that
\begin{equation}\label{eqhermite3}
(r+s)(r+s-1)=0.
\end{equation}
Let $(r,s)$ point of (\ref{eqhermite3}) such that
\begin{equation}\label{eqhermite4}
(r,s)\notin\mathcal C= \{(r_1,s_1)\in\mathbb{C}^2;\;(|Q|+r_1+s_1)(|Q|+r_1+s_1-1)=0,\;\mbox{ for some }|Q|=1,2,\ldots\}.
\end{equation}

Thus we have $d_{1,0}=d_{0,1}=0$  and
$$d_Q=-\frac{[-2(|Q|+r+s)+4+\lambda]d_{q_1-2,q_2}}{(|Q|+r+s)(|Q|+r+s-1)},\;\mbox{ for every }|Q|=2,3\ldots.$$
Observe that
$$d_{1,1}=d_{0,2}=0\mbox{ and }d_{2,0}=-\frac{[-2(r+s)+\lambda]d_{0,0}}{(2+r+s)(1+r+s)}$$

$$d_{1,2}=d_{0,3}=0\mbox{ and }d_{3,0}=-\frac{[-2(r+s)+\lambda-2]d_{1,0}}{(3+r+s)(2+r+s)}=0,d_{2,1}=-\frac{[-2(r+s)+\lambda-2]d_{0,1}}{(3+r+s)(2+r+s)}=0$$

$$\begin{array}{c} d_{1,3}=d_{0,4}=0\mbox{ and }d_{4,0}=-\frac{[-2(r+s)+\lambda-4]d_{2,0}}{(4+r+s)(3+r+s)}=
(-1)^2\frac{[-2(r+s)+\lambda-4][-2(r+s)+\lambda]d_{0,0}}{(4+r+s)(3+r+s)(2+r+s)(1+r+s)}
,\\
\\
d_{3,1}=-\frac{[-2(r+s)+\lambda-4]d_{1,1}}{(4+r+s)(3+r+s)}=0,d_{2,2}=-\frac{[-2(r+s)+\lambda-4]d_{0,2}}{(4+r+s)(3+r+s)}=0\end{array}$$

$$\begin{array}{c} d_{1,4}=d_{0,5}=0\mbox{ and }d_{5,0}=-\frac{[-2(r+s)+\lambda-6]d_{3,0}}{(5+r+s)(4+r+s)}=0,\\
\\
d_{4,1}=-\frac{[-2(r+s)+\lambda-6]d_{2,1}}{(5+r+s)(4+r+s)}=0,d_{3,2}=-\frac{[-2(r+s)+\lambda-6]d_{1,2}}{(5+r+s)(4+r+s)}=0,d_{2,3}=-\frac{[-2(r+s)+\lambda-6]d_{0,3}}{(5+r+s)(4+r+s)}=0\end{array}$$

$$\begin{array}{c}d_{1,5}=d_{0,6}=0\mbox{ and }d_{6,0}=-\frac{[-2(r+s)+\lambda-8]d_{4,0}}{(6+r+s)(5+r+s)}=
(-1)^3\frac{[-2(r+s)+\lambda-8][-2(r+s)+\lambda-4][-2(r+s)+\lambda]d_{0,0}}{(6+r+s)(5+r+s)(4+r+s)(3+r+s)(2+r+s)(1+r+s)},\\
\\
d_{5,1}=-\frac{[-2(r+s)+\lambda-8]d_{3,1}}{(6+r+s)(5+r+s)}=0,d_{4,2}=-\frac{[-2(r+s)+\lambda-8]d_{2,2}}{(6+r+s)(5+r+s)}=0,\\
\\d_{3,3}=-\frac{[-2(r+s)+\lambda-8]d_{1,3}}{(6+r+s)(5+r+s)}=0,
d_{2,4}=-\frac{[-2(r+s)+\lambda-8]d_{0,4}}{(6+r+s)(5+r+s)}=0.
\end{array} $$

In general
$$d_{q_1,q_2}=0\;\;(q_1,q_2)\neq(2n,0)$$
and
$$d_{2n,0}=\frac{(-1)^n [-2(r+s)+\lambda-4(n-1)]\cdots[-2(r+s)+\lambda] d_{0,0}}{(2n+r+s)\cdots(1+r+s)}\;\;\;n=1,2,\ldots$$
Choosing $d_{0,0}=1$ we have that
\begin{equation}\label{eqhermite5}
\varphi_{(r,s),\lambda}(x,y)=x^{r}y^{s}\big[1+\sum^\infty_{n=1}\frac{(-1)^n [-2(r+s)+\lambda-4(n-1)]\cdots[-2(r+s)+\lambda] x^{2n}}{(2n+r+s)\cdots(1+r+s)}\big]
\end{equation}
 is solution of (\ref{eqhermite1}) where $(r,s)$ verifies (\ref{eqhermite3}) and (\ref{eqhermite4}).

 \subsubsection{Hermite polynomials}

 Let us first  investigate the non-resonance conditions \eqref{eqhermite3} and \eqref{eqhermite4} above. If $(r,s)\in \mathcal C$ then
$r+s=1$ or $r+s=0$. Assume that $r+s=0$. In this case the resonance condition becomes
$|Q|(|Q|-1)=0$ for some multi-index $Q=(q_1,q_2)$ with $q_1 + q_2 \geq 1$.
Therefore it becomes $|Q|=1$. Thus we may have resonance in this case.
On the other hand if $r+s=1$ then the resonance condition becomes
$(|Q|+1)|Q|=0$ and therefore it becomes $|Q|+1=0$. This does not occur and we do not have resonance in this case. For instance we can take $r=s=1/2$ and have no resonance.
In this case $r+s=1$ and we have a solution of the form
\begin{equation}\label{eqhermite6}
\varphi_{(\frac{1}{2},\frac{1}{2}),\lambda}(x,y)=x^{1/2}y^{1/2}\big[1+\sum^\infty_{n=1}\frac{(-1)^n [-2+\lambda-4(n-1)]\cdots[-2+\lambda] x^{2n}}{(2n+2)\cdots(2)}\big]
\end{equation}
  Thus, for $\lambda=2(2m+1)$ we have a   solution $H_m(x,y)=(xy)^{\frac{1}{2}}P_m(x,y)$
  where $P_m(x,y)$ is a polynomial of degree $m$, satisfying $P_m(0,0)=1$. Given
  any $r,s$ with $r+s=1$ and $\lambda = 2(2m+1)$ as above, we have a solution of the form
  $\vr(x,y)=x^r y^s P_m(x,y)$. We shall call $P_m(x,y)$  {\it Hermite polynomial} of degree $m$.

\subsubsection{Hermite equation II}
A second possible model for the Hermite equation in the PDE framework would be as follows:
Consider the equation
\begin{equation}\label{2eqhermite1}
x^2\frac{\partial^2 z}{\partial x^2}+2xy\frac{\partial^2 z}{\partial x\partial y}+y^2\frac{\partial^2 z}{\partial y^2}-2x^3\frac{\partial z}{\partial x}-2y^3\frac{\partial z}{\partial y}+\lambda xyz=0,
\end{equation}
where $\lambda\in\mathbb{R}$.
This second model though symmetric with respect to $x$ and $y$ and though it is parabolic, it does
restrict to lines $y=tx$ as a Hermite ordinary differential equation.

\subsubsection{Solution of Hermite PDE model II}
Let $\varphi$ be a solution  of (\ref{2eqhermite1}) of the form
\begin{equation}\label{2eqhermite2}
\varphi(x,y)=x^ry^s\sum^{\infty}_{|Q|=0}d_QX^Q
\end{equation}
where $d_{0,0}\neq0$. Then
$$\frac{\partial\varphi}{\partial x}=\sum^\infty_{|Q|=0}(q_1+r)d_Qx^{q_1+r-1}y^{q_2+s} $$
$$\frac{\partial\varphi}{\partial y}=\sum^\infty_{|Q|=0}(q_2+s)d_Qx^{q_1+r}y^{q_2+s-1} $$
$$\frac{\partial^2\varphi}{\partial x^2}=\sum^\infty_{|Q|=0}(q_1+r)(q_1+r-1)d_Qx^{q_1+r-2}y^{q_2+s} $$
$$\frac{\partial^2\varphi}{\partial x \partial y}=\sum^\infty_{|Q|=0}(q_1+r)(q_2+s)d_Qx^{q_1+r-1}y^{q_2+s-1} $$
$$\frac{\partial^2\varphi}{\partial y^2}=\sum^\infty_{|Q|=0}(q_2+s)(q_2+s-1)d_Qx^{q_1+r}y^{q_2+s-2} $$

and from this we have
$$x^2\frac{\partial^2\varphi}{\partial x^2}=x^ry^s\sum^\infty_{|Q|=0}(q_1+r)(q_1+r-1)d_QX^Q$$
$$2xy\frac{\partial^2\varphi}{\partial x \partial y}=x^ry^s\sum^\infty_{|Q|=0}2(q_1+r)(q_2+s)d_QX^Q$$
$$y^2\frac{\partial^2\varphi}{\partial y^2}=x^ry^s\sum^\infty_{|Q|=0}(q_2+s)(q_2+s-1)d_QX^Q$$
$$-2x^3\frac{\partial\varphi}{\partial x}=x^ry^s\sum^\infty_{|Q|=0}-2(q_1+r)d_Qx^{q_1+2}y^{q_2}=x^ry^s\sum^{\infty}_{|Q|=2}-2(q_1+r-2)d_{q_1-2,q_2}X^Q $$
$$-2y^3\frac{\partial\varphi}{\partial y}=x^ry^s\sum^\infty_{|Q|=0}-2(q_2+s)d_Qx^{q_1}y^{q_2+2}=x^ry^s\sum^{\infty}_{|Q|=2}-2(q_2+s-2)d_{q_1,q_2-2}X^Q$$
$$\lambda xy\varphi=x^ry^s\sum^{\infty}_{|Q|=0}\lambda d_Qx^{q_1+1}y^{q_2+1}=x^ry^s\sum^{\infty}_{|Q|=2}\lambda d_{q_1-1,q_2-1}X^Q.$$
Given that $\varphi$ is solution of (\ref{2eqhermite1}) we have
$$x^2\frac{\partial^2 \varphi}{\partial x^2}+2xy\frac{\partial^2 \varphi}{\partial x\partial y}+y^2\frac{\partial^2 \varphi}{\partial y^2}-2x^3\frac{\partial \varphi}{\partial x}-2y^3\frac{\partial \varphi}{\partial y}+\lambda x^2\varphi=0.$$
Therefore
$$\begin{array}{c} x^ry^s\big[\sum^\infty_{|Q|=0}\big[(q_1+r)(q_1+r-1)+2(q_1+r)(q_2+s)+(q_2+s)(q_2+s-1)\big]d_QX^Q \\ \\ +\sum^{\infty}_{|Q|=2}[-2(q_1+r-2)]d_{q_1-2,q_2}X^Q+\sum^{\infty}_{|Q|=2}[-2(q_2+s-2)]d_{q_1,q_2-2}X^Q+\sum^{\infty}_{|Q|=2}[\lambda d_{q_1-1,q_2-1}]X^Q\big]=0
\end{array}$$
$$\begin{array}{c} \sum^\infty_{|Q|=0}\big[(q_1+q_2+r+s)^2-(q_1+q_2+r+s)\big]d_QX^Q \\ \\ +\sum^{\infty}_{|Q|=2}[-2(q_1+r-2)d_{q_1-2,q_2}-2(q_2+s-2)d_{q_1,q_2-2}+\lambda d_{q_1-1,q_2-1}]X^Q=0
\end{array}$$
$$\begin{array}{c} \sum^\infty_{|Q|=0}\big[(|Q|+r+s)^2-(|Q|+r+s)\big]d_QX^Q \\ \\ +\sum^{\infty}_{|Q|=2}[-2(q_1+r-2)d_{q_1-2,q_2}-2(q_2+s-2)d_{q_1,q_2-2}+\lambda d_{q_1-1,q_2-1}]X^Q=0
\end{array}$$

$$\begin{array}{c}
(r+s)(r+s-1)d_{0,0}+(1+r+s)(r+s)\big[d_{1,0}x+d_{0,1}y\big]+\\
\\\sum^{\infty}_{|Q|=2}[(|Q|+r+s)(|Q|+r+s-1)d_Q-2(q_1+r-2)d_{q_1-2,q_2}-2(q_2+s-2)d_{q_1,q_2-2}+\lambda d_{q_1-1,q_2-1}]X^Q=0 \end{array}$$
and then $$(r+s)(r+s-1)d_{0,0}=0,\;(1+r+s)(r+s)d_{1,0}=0,\;(1+r+s)(r+s)d_{0,1}=0$$
 and
$$(|Q|+r+s)(|Q|+r+s-1)d_Q-2(q_1+r-2)d_{q_1-2,q_2}-2(q_2+s-2)d_{q_1,q_2-2}+\lambda d_{q_1-1,q_2-1}=0,\;\mbox{ for every }|Q|=2,3,\ldots$$
Given that $d_{0,0}\neq0$ we have that
\begin{equation}\label{2eqhermite3}
(r+s)(r+s-1)=0.
\end{equation}
Let $(r,s)$ point of (\ref{2eqhermite3}) such that
\begin{equation}\label{2eqhermite4}
(r,s)\notin\{(r_1,s_1)\in\mathbb{C}^2;\;(|Q|+r+s)(|Q|+r_1+s_1-1)=0,\;\mbox{ for some }|Q|=1,2,\ldots\}.
\end{equation}
Thus we have $d_{1,0}=d_{0,1}=0$  and
\begin{equation}\label{2eqhermite4'}
d_Q=\frac{2(q_1+r-2)d_{q_1-2,q_2}+2(q_2+s-2)d_{q_1,q_2-2}-\lambda d_{q_1-1,q_2-1}}{(|Q|+r+s)(|Q|+r+s-1)},\;\mbox{ for every }|Q|=2,3\ldots.
\end{equation}
Observe that
$$d_{2,0}=\frac{2rd_{0,0}}{(2+r+s)(1+r+s)},d_{1,1}=\frac{-\lambda d_{0,0}}{(2+r+s)(1+r+s)},d_{0,2}=\frac{2sd_{0,0}}{(2+r+s)(1+r+s)}$$

$$\begin{array}{c}
d_{3,0}=\frac{2(r+1)d_{1,0}}{(3+r+s)(2+r+s)}=0,d_{2,1}=\frac{2rd_{0,1}-\lambda d_{1,0}}{(3+r+s)(2+r+s)}=0,d_{1,2}=\frac{2sd_{1,0}-\lambda d_{0,1}}{(3+r+s)(2+r+s)}=0,d_{0,3}=\frac{2(s+1)d_{0,1}}{(3+r+s)(2+r+s)}=0 \end{array}$$

$$\begin{array}{c}
d_{4,0}=\frac{2(r+2)d_{2,0}}{(4+r+s)(3+r+s)}=\frac{2^2(r+2)rd_{0,0}}{(4+r+s)(3+r+s)(2+r+s)(1+r+s)},d_{3,1}=\frac{2(r+1)d_{1,1}-\lambda d_{2,0}}{(4+r+s)(3+r+s)}=\frac{-2\lambda(2r+1)d_{0,0}}{(4+r+s)(3+r+s)(2+r+s)(1+r+s)}\\
\\
d_{2,2}=\frac{2rd_{0,2}+2sd_{2,0}-\lambda d_{1,1}}{(4+r+s)(3+r+s)}=\frac{(8rs+\lambda^2)d_{0,0}}{(4+r+s)(3+r+s)(2+r+s)(1+r+s)},d_{1,3}=\frac{2(s+1)d_{1,1}-\lambda d_{0,2}}{(4+r+s)(3+r+s)}=\frac{-2\lambda(2s+1)d_{0,0}}{(4+r+s)(3+r+s)(2+r+s)(1+r+s)}\\
\\
d_{0,4}=\frac{2(s+2)d_{0,2}}{(4+r+s)(3+r+s)}=\frac{2^2(s+2)sd_{0,0}}{(4+r+s)(3+r+s)(2+r+s)(1+r+s)}
 \end{array}$$

$$\begin{array}{c}
d_{5,0}=\frac{2(r+3)d_{3,0}}{(5+r+s)(4+r+s)}=0,d_{4,1}=\frac{2(r+2)d_{2,1}-\lambda d_{3,0}}{(5+r+s)(4+r+s)}=0,d_{3,2}=\frac{2(r+1)d_{1,2}+2sd_{3,0}-\lambda d_{2,1}}{(5+r+s)(4+r+s)}=0,\\
\\d_{2,3}=\frac{2rd_{0,3}+2(s+1)d_{2,1}-\lambda d_{1,2}}{(5+r+s)(4+r+s)}=0,d_{1,4}=\frac{2(s+2)d_{1,2}-\lambda d_{0,3}}{(5+r+s)(4+r+s)}=0, d_{0,5}=\frac{2(s+3)d_{0,3}}{(5+r+s)(4+r+s)}=0
 \end{array}$$

$$\begin{array}{c}
d_{6,0}=\frac{2(r+4)d_{4,0}}{(6+r+s)(5+r+s)}=\frac{2^3(r+4)(r+2)rd_{0,0}}{(6+r+s)(5+r+s)(4+r+s)(3+r+s)(2+r+s)(1+r+s)},\\
\\d_{5,1}=\frac{2(r+3)d_{3,1}-\lambda d_{4,0}}{(6+r+s)(5+r+s)}=\frac{-2^2\lambda(3r^2+9r+3)d_{0,0}}{(6+r+s)(5+r+s)(4+r+s)(3+r+s)(2+r+s)(1+r+s)}\\
\\
d_{4,2}=\frac{2(r+2)d_{2,2}+2sd_{4,0}-\lambda d_{3,1}}{(6+r+s)(5+r+s)}=\frac{[24(r+2)rs+6\lambda^2(r+1)]d_{0,0}}{(6+r+s)(5+r+s)(4+r+s)(3+r+s)(2+r+s)(1+r+s)}\\
\\
d_{3,3}=\frac{2(r+1)d_{1,3}+2(s+1)d_{3,1}-\lambda d_{2,2}}{(6+r+s)(5+r+s)}=\frac{-\lambda[24rs+12r+12s+8+\lambda^2]d_{0,0}}{(6+r+s)(5+r+s)(4+r+s)(3+r+s)(2+r+s)(1+r+s)}\\
\\
d_{2,4}=\frac{2rd_{0,4}+2(s+2)d_{2,2}-\lambda d_{1,3}}{(6+r+s)(5+r+s)}=\frac{[24(s+2)rs+6\lambda^2(s+1)]d_{0,0}}{(6+r+s)(5+r+s)(4+r+s)(3+r+s)(2+r+s)(1+r+s)}\\
\\d_{1,5}=\frac{2(s+3)d_{1,3}-\lambda d_{0,4}}{(6+r+s)(5+r+s)}=\frac{-2^2\lambda(3s^2+9s+3)d_{0,0}}{(6+r+s)(5+r+s)(4+r+s)(3+r+s)(2+r+s)(1+r+s)}\\
\\
d_{0,6}=\frac{2(s+4)d_{0,4}}{(6+r+s)(5+r+s)}=\frac{2^3(s+4)(s+2)sd_{0,0}}{(6+r+s)(5+r+s)(4+r+s)(3+r+s)(2+r+s)(1+r+s)}.
 \end{array}$$
In general
$$d_{q_1,q_2}=0\;\;|Q|=2n-1$$
and
$$d_{q_1,q_2}\neq0\;\;\;|Q|=2n.$$
Choosing $d_{0,0}=1$ we have that
\begin{equation}\label{2eqhermite5}
\varphi(x,y)=x^{r}y^{s}\big[1+\sum^\infty_{|Q|\;\mbox{even}}d_QX^Q\big]
\end{equation}
where $d_Q$ given by the recurrence (\ref{2eqhermite4'}), is solution of (\ref{2eqhermite1}) where $(r_0,s_0)$ verifies (\ref{2eqhermite3}) and (\ref{2eqhermite4}).

It is visible that the recurrence in this Hermite PDE model II is much more complicate than the one for the first model.

\subsection{Legendre equation}
The Legendre ordinary differential equation writes
\[
(1-x^2) y^{\prime \prime}  - 2xy^\prime + \lambda(\lambda+1)y=0
\]
where $\lambda$ is a constant.

The Legendre differential equation shows up when applying the  separation of variables solution of second order linear elliptic, hyperbolic and parabolic partial differential equations in spherical coordinates, especially the Helmholtz equation, Laplace's equation, and the Schrödinger equation. These equations are important  in electrostatics, electromagnetic wave propagation (e.g. antenna theory), and the solution of the Hydrogen atom wave functions in single-particle quantum mechanics. Their solutions form the polar angle part of the spherical harmonics basis for the multi pole expansion, which is used in both electromagnetic and gravitational statics.

\subsubsection{Legendre PDE I}
Consider the equation
\begin{equation}\label{eqlegendre1}
(1-x^2)x^2\frac{\partial^2 z}{\partial x^2}+(1-x^2)2xy\frac{\partial^2 z}{\partial x\partial y}+(1-x^2)y^2\frac{\partial^2 z}{\partial y^2}-2x^3\frac{\partial z}{\partial x}-2x^2y\frac{\partial z}{\partial y}+\lambda(\lambda+1) x^2z=0,
\end{equation}
where $\lambda\in\mathbb{R}$. Let $\varphi$ be a solution  of (\ref{eqlegendre1}) of the form
\begin{equation}\label{eqlegendre2}
\varphi(x,y)=x^ry^s\sum^{\infty}_{|Q|=0}d_QX^Q
\end{equation}
where $d_{0,0}\neq0$. Then
$$\frac{\partial\varphi}{\partial x}=\sum^\infty_{|Q|=0}(q_1+r)d_Qx^{q_1+r-1}y^{q_2+s} $$
$$\frac{\partial\varphi}{\partial y}=\sum^\infty_{|Q|=0}(q_2+s)d_Qx^{q_1+r}y^{q_2+s-1} $$
$$\frac{\partial^2\varphi}{\partial x^2}=\sum^\infty_{|Q|=0}(q_1+r)(q_1+r-1)d_Qx^{q_1+r-2}y^{q_2+s} $$
$$\frac{\partial^2\varphi}{\partial x \partial y}=\sum^\infty_{|Q|=0}(q_1+r)(q_2+s)d_Qx^{q_1+r-1}y^{q_2+s-1} $$
$$\frac{\partial^2\varphi}{\partial y^2}=\sum^\infty_{|Q|=0}(q_2+s)(q_2+s-1)d_Qx^{q_1+r}y^{q_2+s-2} $$

and from this we have
$$\begin{array}{c l }(1-x^2)x^2\dfrac{\partial^2\varphi}{\partial x^2}&=(1-x^2)x^ry^s\sum^\infty_{|Q|=0}(q_1+r)(q_1+r-1)d_QX^Q\\&
\\&=x^ry^s\sum^\infty_{|Q|=0}(q_1+r)(q_1+r-1)d_QX^Q-x^ry^s\sum^\infty_{|Q|=0}(q_1+r)(q_1+r-1)d_Qx^{q_1+2}y^{q_2}\\
&\\
&=x^ry^s\sum^\infty_{|Q|=0}(q_1+r)(q_1+r-1)d_QX^Q-x^ry^s\sum^\infty_{|Q|=2}(q_1+r-2)(q_1+r-3)d_{q_1-2,q_2}X^Q \end{array}$$

$$\begin{array}{cl}(1-x^2)2xy\dfrac{\partial^2\varphi}{\partial x \partial y}&=(1-x^2)x^ry^s\sum^\infty_{|Q|=0}2(q_1+r)(q_2+s)d_QX^Q\\&
\\&=x^ry^s\sum^\infty_{|Q|=0}2(q_1+r)(q_2+s)d_QX^Q-x^ry^s\sum^\infty_{|Q|=0}2(q_1+r)(q_2+s)d_Qx^{q_1+2}y^{q_2}\\
&\\
&=x^ry^s\sum^\infty_{|Q|=0}2(q_1+r)(q_2+s)d_QX^Q-x^ry^s\sum^\infty_{|Q|=2}2(q_1+r-2)(q_2+s)d_{q_1-2,q_2}X^Q
\end{array}$$

$$\begin{array}{cl}(1-x^2)y^2\dfrac{\partial^2\varphi}{\partial y^2}&=(1-x^2)x^ry^s\sum^\infty_{|Q|=0}(q_2+s)(q_2+s-1)d_QX^Q\\&\\
&=x^ry^s\sum^\infty_{|Q|=0}(q_2+s)(q_2+s-1)d_QX^Q-x^ry^s\sum^\infty_{|Q|=0}(q_2+s)(q_2+s-1)d_Qx^{q_1+2}y^{q_2}\\&\\
&=x^ry^s\sum^\infty_{|Q|=0}(q_2+s)(q_2+s-1)d_QX^Q-x^ry^s\sum^\infty_{|Q|=2}(q_2+s)(q_2+s-1)d_{q_1-2,q_2}X^Q
\end{array}$$

$$-2x^3\frac{\partial\varphi}{\partial x}=x^ry^s\sum^\infty_{|Q|=0}-2(q_1+r)d_Qx^{q_1+2}y^{q_2}=x^ry^s\sum^{\infty}_{|Q|=2}-2(q_1+r-2)d_{q_1-2,q_2}X^Q $$
$$-2x^2y\frac{\partial\varphi}{\partial y}=x^ry^s\sum^\infty_{|Q|=0}(q_2+s)d_Qx^{q_1+2}y^{q_2}=x^ry^s\sum^{\infty}_{|Q|=2}-2(q_2+s)d_{q_1-2,q_2}X^Q $$
$$\lambda(\lambda+1) x^2\varphi=x^ry^s\sum^{\infty}_{|Q|=0}\lambda(\lambda+1) d_Qx^{q_1+2}y^{q_2}=x^ry^s\sum^{\infty}_{|Q|=2}\lambda(\lambda+1) d_{q_1-2,q_2}X^Q.$$
Given that $\varphi$ is solution of (\ref{eqlegendre1}) we have
$$(1-x^2)x^2\frac{\partial^2 \varphi}{\partial x^2}+(1-x^2)2xy\frac{\partial^2 \varphi}{\partial x\partial y}+(1-x^2)y^2\frac{\partial^2 \varphi}{\partial y^2}-2x^3\frac{\partial \varphi}{\partial x}-2x^2y\frac{\partial \varphi}{\partial y}+\lambda(\lambda+1) x^2\varphi=0.$$
Therefore
$$\begin{array}{c} x^ry^s\big[\sum^\infty_{|Q|=0}\big[(q_1+r)(q_1+r-1)+2(q_1+r)(q_2+s)+(q_2+s)(q_2+s-1)\big]d_QX^Q \\ \\ +\sum^{\infty}_{|Q|=2}[-(q_1+r-2)(q_1+r-3)-2(q_1+r-2)(q_2+s)-(q_2+s)(q_2+s-1)\\
\\-2(q_1+r-2)-2(q_2+s)+\lambda(\lambda+1)]d_{q_1-2,q_2}X^Q\big]=0
\end{array}$$

$$\begin{array}{c} \sum^\infty_{|Q|=0}\big[(q_1+q_2+r+s)^2-(q_1+q_2+r+s)\big]d_QX^Q \\ \\ +\sum^{\infty}_{|Q|=2}[-(q_1+q_2+r+s-2)^2+(q_1+q_2+r+s-2)+\lambda(\lambda+1)]d_{q_1-2,q_2}X^Q=0
\end{array}$$
$$\sum^\infty_{|Q|=0}\big[(|Q|+r+s)^2-(|Q|+r+s)\big]d_QX^Q+\sum^{\infty}_{|Q|=2}[-(|Q|+r+s-2)^2+(|Q|+r+s-2)+\lambda(\lambda+1)]d_{q_1-2,q_2}X^Q=0$$
$$\begin{array}{c}
(r+s)(r+s-1)d_{0,0}+(1+r+s)(r+s)\big[d_{1,0}x+d_{0,1}y\big]+\\
\\\sum^{\infty}_{|Q|=2}\big[(|Q|+r+s)(|Q|+r+s-1)d_Q-[(|Q|+r+s-2)(|Q|+r+s-3)-\lambda(\lambda+1)]d_{q_1-2,q_2}\big]X^Q=0\end{array}$$
and then $$(r+s)(r+s-1)d_{0,0}=0,\;(1+r+s)(r+s)d_{1,0}=0,\;(1+r+s)(r+s)d_{0,1}=0$$
 and
$$(|Q|+r+s)(|Q|+r+s-1)d_Q-[(|Q|+r+s-2)(|Q|+r+s-3)-\lambda(\lambda+1)]d_{q_1-2,q_2}=0,\;\mbox{ for every }|Q|=2,3,\ldots$$
Given that $d_{0,0}\neq0$ we have that
\begin{equation}\label{eqlegendre3}
(r+s)(r+s-1)=0.
\end{equation}
Let $(r,s)$ point of (\ref{eqlegendre3}) such that
\begin{equation}\label{eqlegendre4}
(r,s)\notin\{(r_1,s_1)\in\mathbb{C}^2;\;(|Q|+r+s)(|Q|+r_1+s_1-1)=0,\;\mbox{ for some }|Q|=1,2,\ldots\}.
\end{equation}
Thus we have $d_{1,0}=d_{0,1}=0$  and
$$d_Q=\frac{[(|Q|+r+s-2)(|Q|+r+s-3)-\lambda(\lambda+1)]d_{q_1-2,q_2}}{(|Q|+r+s)(|Q|+r+s-1)},\;\mbox{ for every }|Q|=2,3\ldots.$$
Observe that
$$d_{1,1}=d_{0,2}=0\mbox{ and }d_{2,0}=\frac{[(r+s)(r+s-1)-\lambda(\lambda+1)]d_{0,0}}{(2+r+s)(1+r+s)}$$

$$d_{1,2}=d_{0,3}=0\mbox{ and }d_{3,0}=\frac{[(r+s+1)(r+s)-\lambda(\lambda+1)]d_{1,0}}{(3+r+s)(2+r+s)}=0,d_{2,1}=\frac{[(r+s+1)(r+s)-\lambda(\lambda+1)]d_{0,1}}{(3+r+s)(2+r+s)}=0$$

$$\begin{array}{c} d_{1,3}=d_{0,4}=0\mbox{ and }
d_{3,1}=\frac{[(r+s+2)(r+s+1)-\lambda(\lambda+1)]d_{1,1}}{(4+r+s)(3+r+s)}=0,d_{2,2}=\frac{[(r+s+2)(r+s+1)-\lambda(\lambda+1)]d_{0,2}}{(4+r+s)(3+r+s)}=0
,\\
\\
d_{4,0}=\frac{[(r+s+2)(r+s+1)-\lambda(\lambda+1)]d_{2,0}}{(4+r+s)(3+r+s)}=
\frac{[(r+s+2)(r+s+1)-\lambda(\lambda+1)][(r+s)(r+s-1)-\lambda(\lambda+1)]d_{0,0}}{(4+r+s)(3+r+s)(2+r+s)(1+r+s)}
\end{array}$$

$$\begin{array}{c} d_{1,4}=d_{0,5}=0\mbox{ and }d_{5,0}=\frac{[(r+s+3)(r+s+2)-\lambda(\lambda+1)]d_{3,0}}{(5+r+s)(4+r+s)}=0,
d_{4,1}=\frac{[(r+s+3)(r+s+2)-\lambda(\lambda+1)]d_{2,1}}{(5+r+s)(4+r+s)}=0,\\
\\d_{3,2}=\frac{[(r+s+3)(r+s+2)-\lambda(\lambda+1)]d_{1,2}}{(5+r+s)(4+r+s)}=0,d_{2,3}=\frac{[(r+s+3)(r+s+2)-\lambda(\lambda+1)]d_{0,3}}{(5+r+s)(4+r+s)}=0\end{array}$$

$$\begin{array}{c}d_{1,5}=d_{0,6}=0\mbox{ and }d_{5,1}=\frac{[(r+s+4)(r+s+3)-\lambda(\lambda+1)]d_{3,1}}{(6+r+s)(5+r+s)}=0,d_{4,2}=\frac{[(r+s+4)(r+s+3)-\lambda(\lambda+1)]d_{2,2}}{(6+r+s)(5+r+s)}=0,\\
\\d_{3,3}=\frac{[(r+s+4)(r+s+3)-\lambda(\lambda+1)]d_{1,3}}{(6+r+s)(5+r+s)}=0,
d_{2,4}=\frac{[(r+s+4)(r+s+3)-\lambda(\lambda+1)]d_{0,4}}{(6+r+s)(5+r+s)}=0\\
\\
\begin{array}{cl}d_{6,0}&=\frac{[(r+s+4)(r+s+3)-\lambda(\lambda+1)]d_{4,0}}{(6+r+s)(5+r+s)}\\
&\\&=\frac{[(r+s+4)(r+s+3)-\lambda(\lambda+1)][(r+s+2)(r+s+1)-\lambda(\lambda+1)][(r+s)(r+s-1)-\lambda(\lambda+1)]d_{0,0}}{(6+r+s)(5+r+s)(4+r+s)(3+r+s)(2+r+s)(1+r+s)}\end{array}.
\end{array}$$

In general
$$d_{q_1,q_2}=0\;\;(q_1,q_2)\neq(2n,0)$$
and
$$d_{2n,0}=\dfrac{[(r+s+2(n-1))(r+s+2(n-1)-1)-\lambda(\lambda+1)]\cdots[(r+s)(r+s-1)-\lambda(\lambda+1)] d_{0,0}}{(2n+r+s)\cdots(1+r+s)}$$for every $n=1,2,\ldots$.

Choosing $d_{0,0}=1$ we have that
\begin{equation}\label{eqlegendre5}
\begin{array}{cl}\varphi(x,y)&=x^{r}y^{s}+\\&\\&
x^{r}y^{s}\sum^\infty_{n=1}\dfrac{[(r+s+2(n-1))(r+s+2(n-1)-1)-\lambda(\lambda+1)]\cdots[(r+s)(r+s-1)-\lambda(\lambda+1)]x^{2n}}{(2n+r+s)\cdots(1+r+s)}\end{array}
\end{equation}
 is solution of (\ref{eqlegendre1}) where $(r,s)$ verifies (\ref{eqlegendre3}) and (\ref{eqlegendre4}).

\subsubsection{Legendre PDE II}
Consider the equation
\begin{equation}\label{2eqlegendre1}
(1-xy)x^2\frac{\partial^2 z}{\partial x^2}+(1-xy)2xy\frac{\partial^2 z}{\partial x\partial y}+(1-xy)y^2\frac{\partial^2 z}{\partial y^2}-2x^3\frac{\partial z}{\partial x}-2y^3\frac{\partial z}{\partial y}+\lambda(\lambda+1) xyz=0,
\end{equation}
where $\lambda\in\mathbb{R}$. Let $\varphi$ be a solution  of (\ref{2eqlegendre1}) of the form
\begin{equation}\label{2eqlegendre2}
\varphi(x,y)=x^ry^s\sum^{\infty}_{|Q|=0}d_QX^Q
\end{equation}
where $d_{0,0}\neq0$. Then
$$\frac{\partial\varphi}{\partial x}=\sum^\infty_{|Q|=0}(q_1+r)d_Qx^{q_1+r-1}y^{q_2+s} $$
$$\frac{\partial\varphi}{\partial y}=\sum^\infty_{|Q|=0}(q_2+s)d_Qx^{q_1+r}y^{q_2+s-1} $$
$$\frac{\partial^2\varphi}{\partial x^2}=\sum^\infty_{|Q|=0}(q_1+r)(q_1+r-1)d_Qx^{q_1+r-2}y^{q_2+s} $$
$$\frac{\partial^2\varphi}{\partial x \partial y}=\sum^\infty_{|Q|=0}(q_1+r)(q_2+s)d_Qx^{q_1+r-1}y^{q_2+s-1} $$
$$\frac{\partial^2\varphi}{\partial y^2}=\sum^\infty_{|Q|=0}(q_2+s)(q_2+s-1)d_Qx^{q_1+r}y^{q_2+s-2} $$

and from this we have
$$\begin{array}{c l }(1-xy)x^2\dfrac{\partial^2\varphi}{\partial x^2}&=(1-x^2)x^ry^s\sum^\infty_{|Q|=0}(q_1+r)(q_1+r-1)d_QX^Q\\&
\\&=x^ry^s\sum^\infty_{|Q|=0}(q_1+r)(q_1+r-1)d_QX^Q-x^ry^s\sum^\infty_{|Q|=0}(q_1+r)(q_1+r-1)d_Qx^{q_1+1}y^{q_2+1}\\
&\\
&=x^ry^s\sum^\infty_{|Q|=0}(q_1+r)(q_1+r-1)d_QX^Q-x^ry^s\sum^\infty_{|Q|=2}(q_1+r-1)(q_1+r-2)d_{q_1-1,q_2-1}X^Q \end{array}$$

$$\begin{array}{cl}(1-xy)2xy\dfrac{\partial^2\varphi}{\partial x \partial y}&=(1-x^2)x^ry^s\sum^\infty_{|Q|=0}2(q_1+r)(q_2+s)d_QX^Q\\&
\\&=x^ry^s\sum^\infty_{|Q|=0}2(q_1+r)(q_2+s)d_QX^Q-x^ry^s\sum^\infty_{|Q|=0}2(q_1+r)(q_2+s)d_Qx^{q_1+1}y^{q_2+1}\\
&\\
&=x^ry^s\sum^\infty_{|Q|=0}2(q_1+r)(q_2+s)d_QX^Q-x^ry^s\sum^\infty_{|Q|=2}2(q_1+r-1)(q_2+s-1)d_{q_1-1,q_2-1}X^Q
\end{array}$$

$$\begin{array}{cl}(1-xy)y^2\dfrac{\partial^2\varphi}{\partial y^2}&=(1-xy)x^ry^s\sum^\infty_{|Q|=0}(q_2+s)(q_2+s-1)d_QX^Q\\&\\
&=x^ry^s\sum^\infty_{|Q|=0}(q_2+s)(q_2+s-1)d_QX^Q-x^ry^s\sum^\infty_{|Q|=0}(q_2+s)(q_2+s-1)d_Qx^{q_1+1}y^{q_2+1}\\&\\
&=x^ry^s\sum^\infty_{|Q|=0}(q_2+s)(q_2+s-1)d_QX^Q-x^ry^s\sum^\infty_{|Q|=2}(q_2+s-1)(q_2+s-2)d_{q_1-1,q_2-1}X^Q
\end{array}$$

$$-2x^3\frac{\partial\varphi}{\partial x}=x^ry^s\sum^\infty_{|Q|=0}-2(q_1+r)d_Qx^{q_1+2}y^{q_2}=x^ry^s\sum^{\infty}_{|Q|=2}-2(q_1+r-2)d_{q_1-2,q_2}X^Q $$
$$-2y^3\frac{\partial\varphi}{\partial y}=x^ry^s\sum^\infty_{|Q|=0}-2(q_2+s)d_Qx^{q_1}y^{q_2+2}=x^ry^s\sum^{\infty}_{|Q|=2}-2(q_2+s-2)d_{q_1,q_2-2}X^Q $$
$$\lambda(\lambda+1) xy\varphi=x^ry^s\sum^{\infty}_{|Q|=0}\lambda(\lambda+1) d_Qx^{q_1+1}y^{q_2+1}=x^ry^s\sum^{\infty}_{|Q|=2}\lambda(\lambda+1) d_{q_1-1,q_2-1}X^Q.$$
Given that $\varphi$ is solution of (\ref{2eqlegendre1}) we have
$$(1-xy)x^2\frac{\partial^2 \varphi}{\partial x^2}+(1-xy)2xy\frac{\partial^2 \varphi}{\partial x\partial y}+(1-xy)y^2\frac{\partial^2 \varphi}{\partial y^2}-2x^3\frac{\partial \varphi}{\partial x}-2y^3\frac{\partial \varphi}{\partial y}+\lambda(\lambda+1) xy\varphi=0.$$
Therefore
$$\begin{array}{c} x^ry^s\big[\sum^\infty_{|Q|=0}\big[(q_1+r)(q_1+r-1)+2(q_1+r)(q_2+s)+(q_2+s)(q_2+s-1)\big]d_QX^Q \\ \\ +\sum^{\infty}_{|Q|=2}\big[[-(q_1+r-1)(q_1+r-2)-2(q_1+r-1)(q_2+s-1)-(q_2+s-1)(q_2+s-2)+\lambda(\lambda+1)]d_{q_1-1,q_2-1}\\ \\+ [-2(q_1+r-2)]d_{q_1-2,q_2}+[-2(q_2+s-2)]d_{q_1,q_2-2}\big]  X^Q\big]=0
\end{array}$$

$$\begin{array}{c} \sum^\infty_{|Q|=0}\big[(q_1+q_2+r+s)^2-(q_1+q_2+r+s)\big]d_QX^Q \\ \\ +\sum^{\infty}_{|Q|=2}\big[[-(q_1+q_2+r+s-2)^2+(q_1+q_2+r+s-2)+\lambda(\lambda+1)]d_{q_1-1,q_2-1}\\ \\+ [-2(q_1+r-2)]d_{q_1-2,q_2}+[-2(q_2+s-2)]d_{q_1,q_2-2}\big]X^Q=0
\end{array}$$
$$\begin{array}{c} \sum^\infty_{|Q|=0}\big[(|Q|+r+s)^2-(|Q|+r+s)\big]d_QX^Q \\ \\ +\sum^{\infty}_{|Q|=2}\big[[-(|Q|+r+s-2)^2+(|Q|+r+s-2)+\lambda(\lambda+1)]d_{q_1-1,q_2-1}\\ \\+ [-2(q_1+r-2)]d_{q_1-2,q_2}+[-2(q_2+s-2)]d_{q_1,q_2-2}\big]X^Q=0
\end{array}$$

$$\begin{array}{c}
(r+s)(r+s-1)d_{0,0}+(1+r+s)(r+s)\big[d_{1,0}x+d_{0,1}y\big]+\sum^{\infty}_{|Q|=2}\big[
(|Q|+r+s)(|Q|+r+s-1)d_Q
\\ \\ -[(|Q|+r+s-2)(|Q|+r+s-3)-\lambda(\lambda+1)]d_{q_1-1,q_2-1}\\ \\-2(q_1+r-2)d_{q_1-2,q_2}-2(q_2+s-2)d_{q_1,q_2-2}\big]X^Q=0 \end{array}$$
and then $$(r+s)(r+s-1)d_{0,0}=0,\;(1+r+s)(r+s)d_{1,0}=0,\;(1+r+s)(r+s)d_{0,1}=0$$
 and
$$(|Q|+r+s)(|Q|+r+s-1)d_Q-[(|Q|+r+s-2)(|Q|+r+s-3)-\lambda(\lambda+1)]d_{q_1-1,q_2-1}-2(q_1+r-2)d_{q_1-2,q_2}-2(q_2+s-2)d_{q_1,q_2-2}=0$$
for every $|Q|=2,3,\ldots$.

Given that $d_{0,0}\neq0$ we have that
\begin{equation}\label{2eqlegendre3}
(r+s)(r+s-1)=0.
\end{equation}
Let $(r,s)$ point of (\ref{2eqlegendre3}) such that
\begin{equation}\label{2eqlegendre4}
(r,s)\notin\{(r_1,s_1)\in\mathbb{C}^2;\;(|Q|+r_1+s_1)(|Q|+r_1+s_1-1)=0,\;\mbox{ for some }|Q|=1,2,\ldots\}.
\end{equation}
Thus we have $d_{1,0}=d_{0,1}=0$  and
\begin{equation}\label{2eqlegendre4'}
d_Q=\frac{[(|Q|+r+s-2)(|Q|+r+s-3)-\lambda(\lambda+1)]d_{q_1-1,q_2-1}+2(q_1+r-2)d_{q_1-2,q_2}+2(q_2+s-2)d_{q_1,q_2-2}}{(|Q|+r+s)(|Q|+r+s-1)},
\end{equation}
for every $|Q|=2,3,\ldots$.

Observe that
$$d_{2,0}=\frac{2rd_{0,0}}{(2+r+s)(1+r+s)},d_{1,1}=\frac{[(r+s)(r+s-1)-\lambda(\lambda+1)]d_{0,0}}{(2+r+s)(1+r+s)},d_{0,2}=\frac{2sd_{0,0}}{(2+r+s)(1+r+s)}$$

$$\begin{array}{c}
d_{3,0}=\frac{2(r+1)d_{1,0}}{(3+r+s)(2+r+s)}=0,d_{2,1}=\frac{2rd_{0,1}+[(r+s+1)(r+s)-\lambda(\lambda+1)]d_{1,0}}{(3+r+s)(2+r+s)}=0,\\ \\ d_{1,2}=\frac{2sd_{1,0}+[(r+s+1)(r+s)-\lambda(\lambda+1)]d_{0,1}}{(3+r+s)(2+r+s)}=0,d_{0,3}=\frac{2(s+1)d_{0,1}}{(3+r+s)(2+r+s)}=0 \end{array}$$

$$\begin{array}{c}
d_{4,0}=\frac{2(r+2)d_{2,0}}{(4+r+s)(3+r+s)}=\frac{2^2(r+2)rd_{0,0}}{(4+r+s)(3+r+s)(2+r+s)(1+r+s)},d_{0,4}=\frac{2(s+2)d_{0,2}}{(4+r+s)(3+r+s)}=\frac{2^2(s+2)sd_{0,0}}{(4+r+s)(3+r+s)(2+r+s)(1+r+s)},\\
\\d_{3,1}=\frac{2(r+1)d_{1,1}+[(r+s+2)(r+s+1)-\lambda(\lambda+1)]d_{2,0}}{(4+r+s)(3+r+s)}=\frac{[2(2r+1)[(r+s)(r+s+1)-\lambda(\lambda+1)]-4s]d_{0,0}}{(4+r+s)(3+r+s)(2+r+s)(1+r+s)}\\
\\
d_{2,2}=\frac{2rd_{0,2}+2sd_{2,0}+[(r+s+2)(r+s+1)-\lambda(\lambda+1)]d_{1,1}}{(4+r+s)(3+r+s)}=\frac{[8rs++[(r+s+2)(r+s+1)-\lambda(\lambda+1)][(r+s)(r+s-1)-\lambda(\lambda+1)]]d_{0,0}}{(4+r+s)(3+r+s)(2+r+s)(1+r+s)},\\ \\
d_{1,3}=\frac{2(s+1)d_{1,1}+[(r+s+2)(r+s+1)-\lambda(\lambda+1)]d_{0,2}}{(4+r+s)(3+r+s)}=\frac{[2(2s+1)[(r+s)(r+s+1)-\lambda(\lambda+1)]-4r]d_{0,0}}{(4+r+s)(3+r+s)(2+r+s)(1+r+s)} \end{array}$$
In general
$$d_{q_1,q_2}=0\;\;|Q|=2n-1$$
and
$$d_{q_1,q_2}\neq0\;\;\;|Q|=2n.$$
Choosing $d_{0,0}=1$ we have that
\begin{equation}\label{2eqlegendre5}
\varphi(x,y)=x^{r}y^{s}\big[1+\sum^\infty_{|Q|\;\mbox{even}}d_QX^Q\big]
\end{equation}
where $d_Q$ given by the recurrence (\ref{2eqlegendre4'}), is solution of (\ref{2eqlegendre1}) where $(r,s)$ verifies (\ref{2eqlegendre3}) and (\ref{2eqlegendre4}).

\subsection{Chebyshev equation}

Chebyshev's equation is the second order equation 
\[
(1-x^2)y^{\prime \prime} - x y^\prime + p^2 y=0
\]
where $p\in \mathbb K$. 
Solutions of Chebyshev's equation are of importance in numerical analysis such as solution to partial differential equations, smoothing of data etc.
Chebyshev's equation can be used to generate polynomials that could serve as mathematical model to approximate some observed physical phenomena.
Let us give a word about it. 

The solutions of the Chebyshev equation are of the form 
$y=\sum\limits_{n=0}^\infty a_n x^n$ where the coefficients $a_n$ are given by 
\[
a_{n+2}=\frac{(n-p)(n+p)}{(n+1)(n+2)} a_n
\]
These solutions are convergent for $|x|<1$. 
By choosing $a_0=1, \, a_1=0$ we obtain a solution $\mathcal T_1(x)$ and by choosing 
$a_0=0, \, a_1=1$ we obtain a solution $\mathcal T_2(x)$. The remarkable fact is:
if $p$ is even then $\mathcal T_1(x)$ is polynomial of degree $p$. If $p$ is odd then 
$\mathcal T_2(x)$ is polynomial of degree $p$. These give rise to the so called {\it Chebyshev polynomials}.

\subsubsection{Chebyshev PDE of type I}
We present now a PDE model for Chebyshev equation, with the property that its restriction 
to straight lines is a Chebyshev ODE. 
Consider the equation
\begin{equation}\label{eqchebyshev1}
(1-x^2)x^2\frac{\partial^2 z}{\partial x^2}+(1-x^2)2xy\frac{\partial^2 z}{\partial x\partial y}+(1-x^2)y^2\frac{\partial^2 z}{\partial y^2}-x^3\frac{\partial z}{\partial x}-x^2y\frac{\partial z}{\partial y}+p^2x^2z=0,
\end{equation}
where $p\in\mathbb{K}$. Let $\varphi$ be a solution  of (\ref{eqchebyshev1}) of the form
\begin{equation}\label{eqchebyshev2}
\varphi(x,y)=x^ry^s\sum^{\infty}_{|Q|=0}d_QX^Q
\end{equation}
where $d_{0,0}\neq0$. Then
$$\frac{\partial\varphi}{\partial x}=\sum^\infty_{|Q|=0}(q_1+r)d_Qx^{q_1+r-1}y^{q_2+s} $$
$$\frac{\partial\varphi}{\partial y}=\sum^\infty_{|Q|=0}(q_2+s)d_Qx^{q_1+r}y^{q_2+s-1} $$
$$\frac{\partial^2\varphi}{\partial x^2}=\sum^\infty_{|Q|=0}(q_1+r)(q_1+r-1)d_Qx^{q_1+r-2}y^{q_2+s} $$
$$\frac{\partial^2\varphi}{\partial x \partial y}=\sum^\infty_{|Q|=0}(q_1+r)(q_2+s)d_Qx^{q_1+r-1}y^{q_2+s-1} $$
$$\frac{\partial^2\varphi}{\partial y^2}=\sum^\infty_{|Q|=0}(q_2+s)(q_2+s-1)d_Qx^{q_1+r}y^{q_2+s-2} $$

and from this we have
$$\begin{array}{c l }(1-x^2)x^2\dfrac{\partial^2\varphi}{\partial x^2}&=(1-x^2)x^ry^s\sum^\infty_{|Q|=0}(q_1+r)(q_1+r-1)d_QX^Q\\&
\\&=x^ry^s\sum^\infty_{|Q|=0}(q_1+r)(q_1+r-1)d_QX^Q-x^ry^s\sum^\infty_{|Q|=0}(q_1+r)(q_1+r-1)d_Qx^{q_1+2}y^{q_2}\\
&\\
&=x^ry^s\sum^\infty_{|Q|=0}(q_1+r)(q_1+r-1)d_QX^Q-x^ry^s\sum^\infty_{|Q|=2}(q_1+r-2)(q_1+r-3)d_{q_1-2,q_2}X^Q \end{array}$$

$$\begin{array}{cl}(1-x^2)2xy\dfrac{\partial^2\varphi}{\partial x \partial y}&=(1-x^2)x^ry^s\sum^\infty_{|Q|=0}2(q_1+r)(q_2+s)d_QX^Q\\&
\\&=x^ry^s\sum^\infty_{|Q|=0}2(q_1+r)(q_2+s)d_QX^Q-x^ry^s\sum^\infty_{|Q|=0}2(q_1+r)(q_2+s)d_Qx^{q_1+2}y^{q_2}\\
&\\
&=x^ry^s\sum^\infty_{|Q|=0}2(q_1+r)(q_2+s)d_QX^Q-x^ry^s\sum^\infty_{|Q|=2}2(q_1+r-2)(q_2+s)d_{q_1-2,q_2}X^Q
\end{array}$$

$$\begin{array}{cl}(1-x^2)y^2\dfrac{\partial^2\varphi}{\partial y^2}&=(1-x^2)x^ry^s\sum^\infty_{|Q|=0}(q_2+s)(q_2+s-1)d_QX^Q\\&\\
&=x^ry^s\sum^\infty_{|Q|=0}(q_2+s)(q_2+s-1)d_QX^Q-x^ry^s\sum^\infty_{|Q|=0}(q_2+s)(q_2+s-1)d_Qx^{q_1+2}y^{q_2}\\&\\
&=x^ry^s\sum^\infty_{|Q|=0}(q_2+s)(q_2+s-1)d_QX^Q-x^ry^s\sum^\infty_{|Q|=2}(q_2+s)(q_2+s-1)d_{q_1-2,q_2}X^Q
\end{array}$$

$$-x^3\frac{\partial\varphi}{\partial x}=x^ry^s\sum^\infty_{|Q|=0}-(q_1+r)d_Qx^{q_1+2}y^{q_2}=x^ry^s\sum^{\infty}_{|Q|=2}-(q_1+r-2)d_{q_1-2,q_2}X^Q $$
$$-x^2y\frac{\partial\varphi}{\partial y}=x^ry^s\sum^\infty_{|Q|=0}-(q_2+s)d_Qx^{q_1+2}y^{q_2}=x^ry^s\sum^{\infty}_{|Q|=2}-(q_2+s)d_{q_1-2,q_2}X^Q $$
$$p^2 x^2\varphi=x^ry^s\sum^{\infty}_{|Q|=0}p^2d_Qx^{q_1+2}y^{q_2}=x^ry^s\sum^{\infty}_{|Q|=2}p^2 d_{q_1-2,q_2}X^Q.$$
Given that $\varphi$ is solution of (\ref{eqchebyshev1}) we have
$$(1-x^2)x^2\frac{\partial^2 \varphi}{\partial x^2}+(1-x^2)2xy\frac{\partial^2 \varphi}{\partial x\partial y}+(1-x^2)y^2\frac{\partial^2 \varphi}{\partial y^2}-x^3\frac{\partial \varphi}{\partial x}-x^2y\frac{\partial \varphi}{\partial y}+p^2 x^2\varphi=0.$$
Therefore
$$\begin{array}{c} x^ry^s\big[\sum^\infty_{|Q|=0}\big[(q_1+r)(q_1+r-1)+2(q_1+r)(q_2+s)+(q_2+s)(q_2+s-1)\big]d_QX^Q \\ \\ +\sum^{\infty}_{|Q|=2}[-(q_1+r-2)(q_1+r-3)-2(q_1+r-2)(q_2+s)-(q_2+s)(q_2+s-1)\\
\\-(q_1+r-2)-(q_2+s)+p^2]d_{q_1-2,q_2}X^Q\big]=0
\end{array}$$

$$\begin{array}{c} \sum^\infty_{|Q|=0}\big[(q_1+q_2+r+s)^2-(q_1+q_2+r+s)\big]d_QX^Q \\ \\ +\sum^{\infty}_{|Q|=2}[-(q_1+q_2+r+s-2)^2+p^2]d_{q_1-2,q_2}X^Q=0
\end{array}$$
$$\sum^\infty_{|Q|=0}\big[(|Q|+r+s)^2-(|Q|+r+s)\big]d_QX^Q+\sum^{\infty}_{|Q|=2}[-(|Q|+r+s-2)^2+p^2]d_{q_1-2,q_2}X^Q=0$$
$$\begin{array}{c}
(r+s)(r+s-1)d_{0,0}+(1+r+s)(r+s)\big[d_{1,0}x+d_{0,1}y\big]+\\
\\\sum^{\infty}_{|Q|=2}\big[(|Q|+r+s)(|Q|+r+s-1)d_Q-(|Q|+r+s-2-p)(|Q|+r+s-2+p)d_{q_1-2,q_2}\big]X^Q=0\end{array}$$
and then $$(r+s)(r+s-1)d_{0,0}=0,\;(1+r+s)(r+s)d_{1,0}=0,\;(1+r+s)(r+s)d_{0,1}=0$$
 and
$$(|Q|+r+s)(|Q|+r+s-1)d_Q-(|Q|+r+s-2-p)(|Q|+r+s-2+p)d_{q_1-2,q_2}=0,\;\mbox{ for every }|Q|=2,3,\ldots$$
Given that $d_{0,0}\neq0$ we have that
\begin{equation}\label{eqchebyshev3}
(r+s)(r+s-1)=0.
\end{equation}
Let $(r,s)$ point of (\ref{eqchebyshev3}) such that
\begin{equation}\label{eqchebyshev4}
(r,s)\notin\{(r_1,s_1)\in\mathbb{C}^2;\;(|Q|+r_1+s_1)(|Q|+r_1+s_1-1)=0,\;\mbox{ for some }|Q|=1,2,\ldots\}.
\end{equation}
Thus we have $d_{1,0}=d_{0,1}=0$  and
$$d_Q=\frac{(|Q|+r+s-2-p)(|Q|+r+s-2+p)d_{q_1-2,q_2}}{(|Q|+r+s)(|Q|+r+s-1)},\;\mbox{ for every }|Q|=2,3\ldots.$$
Observe that
$$d_{1,1}=d_{0,2}=0\mbox{ and }d_{2,0}=\frac{(r+s-p)(r+s+p)d_{0,0}}{(2+r+s)(1+r+s)}$$

$$d_{1,2}=d_{0,3}=0\mbox{ and }d_{3,0}=\frac{(r+s+1-p)(r+s+1+p)d_{1,0}}{(3+r+s)(2+r+s)}=0,d_{2,1}=\frac{(r+s+1-p)(r+s+1+p)d_{0,1}}{(3+r+s)(2+r+s)}=0$$

$$\begin{array}{c} d_{1,3}=d_{0,4}=0\mbox{ and }
d_{3,1}=\frac{(r+s+2-p)(r+s+2+p)d_{1,1}}{(4+r+s)(3+r+s)}=0,d_{2,2}=\frac{(r+s+2-p)(r+s+2+p)d_{0,2}}{(4+r+s)(3+r+s)}=0
,\\
\\
d_{4,0}=\frac{(r+s+2-p)(r+s+2+p)d_{2,0}}{(4+r+s)(3+r+s)}=
\frac{(r+s+2-p)(r+s+2+p)(r+s-p)(r+s+p)d_{0,0}}{(4+r+s)(3+r+s)(2+r+s)(1+r+s)}
\end{array}$$

$$\begin{array}{c} d_{1,4}=d_{0,5}=0\mbox{ and }d_{5,0}=\frac{(r+s+3-p)(r+s+3+p)d_{3,0}}{(5+r+s)(4+r+s)}=0,
d_{4,1}=\frac{(r+s+3-p)(r+s+3+p)d_{2,1}}{(5+r+s)(4+r+s)}=0,\\
\\d_{3,2}=\frac{(r+s+3-p)(r+s+3+p)d_{1,2}}{(5+r+s)(4+r+s)}=0,d_{2,3}=\frac{(r+s+3-p)(r+s+3+p)d_{0,3}}{(5+r+s)(4+r+s)}=0\end{array}$$

$$\begin{array}{c}d_{1,5}=d_{0,6}=0\mbox{ and }d_{5,1}=\frac{(r+s+4-p)(r+s+4+p)d_{3,1}}{(6+r+s)(5+r+s)}=0,d_{4,2}=\frac{(r+s+4-p)(r+s+4+p)d_{2,2}}{(6+r+s)(5+r+s)}=0,\\
\\d_{3,3}=\frac{(r+s+4-p)(r+s+4+p)d_{1,3}}{(6+r+s)(5+r+s)}=0,
d_{2,4}=\frac{(r+s+4-p)(r+s+4+p)d_{0,4}}{(6+r+s)(5+r+s)}=0\\
\\
\begin{array}{cl}d_{6,0}&=\frac{(r+s+4-p)(r+s+4+p)d_{4,0}}{(6+r+s)(5+r+s)}\\
&\\&=\frac{(r+s+4-p)(r+s+4+p)(r+s+2-p)(r+s+2+p)(r+s-p)(r+s+p) d_{0,0}}{(6+r+s)(5+r+s)(4+r+s)(3+r+s)(2+r+s)(1+r+s)}\end{array}.
\end{array}$$

In general
$$d_{q_1,q_2}=0\;\;(q_1,q_2)\neq(2n,0)$$
and
$$d_{2n,0}=\dfrac{(r+s+2(n-1)-p)(r+s+2(n-1)+p)\cdots(r+s-p)(r+s+p)d_{0,0}}{(2n+r+s)\cdots(1+r+s)}$$for every $n=1,2,\ldots$.

Choosing $d_{0,0}=1$ we have that
\begin{equation}\label{eqchebyshev5}
\begin{array}{cl}\varphi(x,y)&=x^{r}y^{s}+\\&\\&
x^{r}y^{s}\sum^\infty_{n=1}\dfrac{(r+s+2(n-1)-p)(r+s+2(n-1)+p)\cdots(r+s-p)(r+s+p)x^{2n}}{(2n+r+s)\cdots(1+r+s)}\end{array}
\end{equation}
 is solution of (\ref{eqchebyshev1}) where $(r,s)$ verifies (\ref{eqchebyshev3}) and (\ref{eqchebyshev4}).

\subsubsection{Chebyshev polynomials  in two variables}
By choosing in \eqref{eqchebyshev5} $r=s=1/2$ we obtain a solution 
\begin{equation}\label{eqchebyshev5bis}
\begin{array}{cl}\varphi(x,y)&=(xy)^{\frac{1}{2}}\big[ 1+\\&\\&
\sum^\infty_{n=1}\dfrac{(1+2(n-1)-p)(1+2(n-1)+p)\cdots(1-p)(1+p)x^{2n}}{(2n+1)\cdots(2)}\big]\end{array}
\end{equation}

Thus for odd natural values of the parameter $p=2n-1\in \mathbb N$ we have  solutions of the form
$\vr(x,y)=(xy)^{\frac{1}{2}}\mathcal T_p(x,y)$, where $\mathcal T_p(x,y)$ are degree $2n=p+1$ playing the role of
Chebyshev polynomials in the case of ODE. Similarly, and surprisingly, when $-p=2n-1\in \mathbb N$ is odd also have a solution of the form $\vr(x,y)=(xy)^{\frac{1}{2}} \tilde {\mathcal T_p}(x,y)$ for some polynomial
$\tilde {\mathcal T_p} (x,y)$ of degree $2n=-p+1$.  
\subsubsection{Chebyshev equation II}
Now we introduce the model for Chebyshev PDE based on symmetries. Consider the equation
\begin{equation}\label{2eqChebyshev1}
(1-xy)x^2\frac{\partial^2 z}{\partial x^2}+(1-xy)2xy\frac{\partial^2 z}{\partial x\partial y}+(1-xy)y^2\frac{\partial^2 z}{\partial y^2}-x^3\frac{\partial z}{\partial x}-y^3\frac{\partial z}{\partial y}+p^2 xyz=0,
\end{equation}
where $p\in\mathbb{K}$. Let $\varphi$ be a solution  of (\ref{2eqChebyshev1}) of the form
\begin{equation}\label{2eqChebyshev2}
\varphi(x,y)=x^ry^s\sum^{\infty}_{|Q|=0}d_QX^Q
\end{equation}
where $d_{0,0}\neq0$. Then
$$\frac{\partial\varphi}{\partial x}=\sum^\infty_{|Q|=0}(q_1+r)d_Qx^{q_1+r-1}y^{q_2+s} $$
$$\frac{\partial\varphi}{\partial y}=\sum^\infty_{|Q|=0}(q_2+s)d_Qx^{q_1+r}y^{q_2+s-1} $$
$$\frac{\partial^2\varphi}{\partial x^2}=\sum^\infty_{|Q|=0}(q_1+r)(q_1+r-1)d_Qx^{q_1+r-2}y^{q_2+s} $$
$$\frac{\partial^2\varphi}{\partial x \partial y}=\sum^\infty_{|Q|=0}(q_1+r)(q_2+s)d_Qx^{q_1+r-1}y^{q_2+s-1} $$
$$\frac{\partial^2\varphi}{\partial y^2}=\sum^\infty_{|Q|=0}(q_2+s)(q_2+s-1)d_Qx^{q_1+r}y^{q_2+s-2} $$

and from this we have
$$\begin{array}{c l }(1-xy)x^2\dfrac{\partial^2\varphi}{\partial x^2}&=(1-x^2)x^ry^s\sum^\infty_{|Q|=0}(q_1+r)(q_1+r-1)d_QX^Q\\&
\\&=x^ry^s\sum^\infty_{|Q|=0}(q_1+r)(q_1+r-1)d_QX^Q-x^ry^s\sum^\infty_{|Q|=0}(q_1+r)(q_1+r-1)d_Qx^{q_1+1}y^{q_2+1}\\
&\\
&=x^ry^s\sum^\infty_{|Q|=0}(q_1+r)(q_1+r-1)d_QX^Q-x^ry^s\sum^\infty_{|Q|=2}(q_1+r-1)(q_1+r-2)d_{q_1-1,q_2-1}X^Q \end{array}$$

$$\begin{array}{cl}(1-xy)2xy\dfrac{\partial^2\varphi}{\partial x \partial y}&=(1-x^2)x^ry^s\sum^\infty_{|Q|=0}2(q_1+r)(q_2+s)d_QX^Q\\&
\\&=x^ry^s\sum^\infty_{|Q|=0}2(q_1+r)(q_2+s)d_QX^Q-x^ry^s\sum^\infty_{|Q|=0}2(q_1+r)(q_2+s)d_Qx^{q_1+1}y^{q_2+1}\\
&\\
&=x^ry^s\sum^\infty_{|Q|=0}2(q_1+r)(q_2+s)d_QX^Q-x^ry^s\sum^\infty_{|Q|=2}2(q_1+r-1)(q_2+s-1)d_{q_1-1,q_2-1}X^Q
\end{array}$$

$$\begin{array}{cl}(1-xy)y^2\dfrac{\partial^2\varphi}{\partial y^2}&=(1-xy)x^ry^s\sum^\infty_{|Q|=0}(q_2+s)(q_2+s-1)d_QX^Q\\&\\
&=x^ry^s\sum^\infty_{|Q|=0}(q_2+s)(q_2+s-1)d_QX^Q-x^ry^s\sum^\infty_{|Q|=0}(q_2+s)(q_2+s-1)d_Qx^{q_1+1}y^{q_2+1}\\&\\
&=x^ry^s\sum^\infty_{|Q|=0}(q_2+s)(q_2+s-1)d_QX^Q-x^ry^s\sum^\infty_{|Q|=2}(q_2+s-1)(q_2+s-2)d_{q_1-1,q_2-1}X^Q
\end{array}$$

$$-x^3\frac{\partial\varphi}{\partial x}=x^ry^s\sum^\infty_{|Q|=0}-(q_1+r)d_Qx^{q_1+2}y^{q_2}=x^ry^s\sum^{\infty}_{|Q|=2}-(q_1+r-2)d_{q_1-2,q_2}X^Q $$
$$-y^3\frac{\partial\varphi}{\partial y}=x^ry^s\sum^\infty_{|Q|=0}-(q_2+s)d_Qx^{q_1}y^{q_2+2}=x^ry^s\sum^{\infty}_{|Q|=2}-(q_2+s-2)d_{q_1,q_2-2}X^Q $$
$$p^2 xy\varphi=x^ry^s\sum^{\infty}_{|Q|=0}p^2 d_Qx^{q_1+1}y^{q_2+1}=x^ry^s\sum^{\infty}_{|Q|=2}p^2 d_{q_1-1,q_2-1}X^Q.$$
Given that $\varphi$ is solution of (\ref{2eqChebyshev1}) we have
$$(1-xy)x^2\frac{\partial^2 \varphi}{\partial x^2}+(1-xy)2xy\frac{\partial^2 \varphi}{\partial x\partial y}+(1-xy)y^2\frac{\partial^2 \varphi}{\partial y^2}-x^3\frac{\partial \varphi}{\partial x}-y^3\frac{\partial \varphi}{\partial y}+p^2 xy\varphi=0.$$
Therefore
$$\begin{array}{c} x^ry^s\big[\sum^\infty_{|Q|=0}\big[(q_1+r)(q_1+r-1)+2(q_1+r)(q_2+s)+(q_2+s)(q_2+s-1)\big]d_QX^Q \\ \\ +\sum^{\infty}_{|Q|=2}\big[[-(q_1+r-1)(q_1+r-2)-2(q_1+r-1)(q_2+s-1)-(q_2+s-1)(q_2+s-2)+p^2]d_{q_1-1,q_2-1}\\ \\+ [-(q_1+r-2)]d_{q_1-2,q_2}+[-(q_2+s-2)]d_{q_1,q_2-2}\big]  X^Q\big]=0
\end{array}$$

$$\begin{array}{c} \sum^\infty_{|Q|=0}\big[(q_1+q_2+r+s)^2-(q_1+q_2+r+s)\big]d_QX^Q \\ \\ +\sum^{\infty}_{|Q|=2}\big[[-(q_1+q_2+r+s-2)^2+(q_1+q_2+r+s-2)+p^2]d_{q_1-1,q_2-1}\\ \\+ [-(q_1+r-2)]d_{q_1-2,q_2}+[-(q_2+s-2)]d_{q_1,q_2-2}\big]X^Q=0
\end{array}$$
$$\begin{array}{c} \sum^\infty_{|Q|=0}\big[(|Q|+r+s)^2-(|Q|+r+s)\big]d_QX^Q \\ \\ +\sum^{\infty}_{|Q|=2}\big[[-(|Q|+r+s-2)^2+(|Q|+r+s-2)+p^2]d_{q_1-1,q_2-1}\\ \\+ [-(q_1+r-2)]d_{q_1-2,q_2}+[-(q_2+s-2)]d_{q_1,q_2-2}\big]X^Q=0
\end{array}$$

$$\begin{array}{c}
(r+s)(r+s-1)d_{0,0}+(1+r+s)(r+s)\big[d_{1,0}x+d_{0,1}y\big]+\sum^{\infty}_{|Q|=2}\big[
(|Q|+r+s)(|Q|+r+s-1)d_Q
\\ \\ -[(|Q|+r+s-2)(|Q|+r+s-3)-p^2]d_{q_1-1,q_2-1}\\ \\-(q_1+r-2)d_{q_1-2,q_2}-(q_2+s-2)d_{q_1,q_2-2}\big]X^Q=0 \end{array}$$
and then $$(r+s)(r+s-1)d_{0,0}=0,\;(1+r+s)(r+s)d_{1,0}=0,\;(1+r+s)(r+s)d_{0,1}=0$$
 and
$$(|Q|+r+s)(|Q|+r+s-1)d_Q-[(|Q|+r+s-2)(|Q|+r+s-3)-p^2]d_{q_1-1,q_2-1}-(q_1+r-2)d_{q_1-2,q_2}-(q_2+s-2)d_{q_1,q_2-2}=0$$
for every $|Q|=2,3,\ldots$.

Given that $d_{0,0}\neq0$ we have that
\begin{equation}\label{2eqChebyshev3}
(r+s)(r+s-1)=0.
\end{equation}
Let $(r,s)$ point of (\ref{2eqChebyshev3}) such that
\begin{equation}\label{2eqChebyshev4}
(r,s)\notin\{(r_1,s_1)\in\mathbb{C}^2;\;(|Q|+r_1+s_1)(|Q|+r_1+s_1-1)=0,\;\mbox{ for some }|Q|=1,2,\ldots\}.
\end{equation}
Thus we have $d_{1,0}=d_{0,1}=0$  and
\begin{equation}\label{2eqChebyshev4'}
d_Q=\frac{[(|Q|+r+s-2)(|Q|+r+s-3)-p^2]d_{q_1-1,q_2-1}+(q_1+r-2)d_{q_1-2,q_2}+(q_2+s-2)d_{q_1,q_2-2}}{(|Q|+r+s)(|Q|+r+s-1)},
\end{equation}
for every $|Q|=2,3,\ldots$.

Observe that
$$d_{2,0}=\frac{rd_{0,0}}{(2+r+s)(1+r+s)},d_{1,1}=\frac{[(r+s)(r+s-1)-p^2]d_{0,0}}{(2+r+s)(1+r+s)},d_{0,2}=\frac{sd_{0,0}}{(2+r+s)(1+r+s)}$$

$$\begin{array}{c}
d_{3,0}=\frac{(r+1)d_{1,0}}{(3+r+s)(2+r+s)}=0,d_{2,1}=\frac{rd_{0,1}+[(r+s+1)(r+s)-p^2]d_{1,0}}{(3+r+s)(2+r+s)}=0,\\ \\ d_{1,2}=\frac{sd_{1,0}+[(r+s+1)(r+s)-p^2]d_{0,1}}{(3+r+s)(2+r+s)}=0,d_{0,3}=\frac{(s+1)d_{0,1}}{(3+r+s)(2+r+s)}=0 \end{array}$$

$$\begin{array}{c}
d_{4,0}=\frac{(r+2)d_{2,0}}{(4+r+s)(3+r+s)}=\frac{(r+2)rd_{0,0}}{(4+r+s)(3+r+s)(2+r+s)(1+r+s)},d_{0,4}=\frac{(s+2)d_{0,2}}{(4+r+s)(3+r+s)}=\frac{(s+2)sd_{0,0}}{(4+r+s)(3+r+s)(2+r+s)(1+r+s)},\\
\\d_{3,1}=\frac{(r+1)d_{1,1}+[(r+s+2)(r+s+1)-p^2]d_{2,0}}{(4+r+s)(3+r+s)}=\frac{[(2r+1)[(r+s)(r+s+1)-p^2]-2s]d_{0,0}}{(4+r+s)(3+r+s)(2+r+s)(1+r+s)}\\
\\
d_{2,2}=\frac{rd_{0,2}+sd_{2,0}+[(r+s+2)(r+s+1)-p^2]d_{1,1}}{(4+r+s)(3+r+s)}=\frac{[2rs+[(r+s+2)(r+s+1)-p^2][(r+s)(r+s-1)-p^2]]d_{0,0}}{(4+r+s)(3+r+s)(2+r+s)(1+r+s)},\\ \\
d_{1,3}=\frac{(s+1)d_{1,1}+[(r+s+2)(r+s+1)-p^2]d_{0,2}}{(4+r+s)(3+r+s)}=\frac{[(2s+1)[(r+s)(r+s+1)-p^2]-2r]d_{0,0}}{(4+r+s)(3+r+s)(2+r+s)(1+r+s)} \end{array}$$
In general
$$d_{q_1,q_2}=0\;\;|Q|=2n-1$$
and
$$d_{q_1,q_2}\neq0\;\;\;|Q|=2n.$$
Choosing $d_{0,0}=1$ we have that
\begin{equation}\label{2eqChebyshev5}
\varphi(x,y)=x^{r}y^{s}\big[1+\sum^\infty_{|Q|\;\mbox{even}}d_QX^Q\big]
\end{equation}
where $d_Q$ given by the recurrence (\ref{2eqChebyshev4'}), is solution of (\ref{2eqChebyshev1}) where $(r,s)$ verifies (\ref{2eqChebyshev3}) and (\ref{2eqChebyshev4}).

\subsection{Laguerre equation}

As already mentioned in the Introduction 
the {\it Laguerre equation}  is given by $x y ^{\prime \prime} + (\nu +1 -x) y^\prime + \lambda y=0$ where
$\lambda, \nu \in \mathbb R$ are parameters. This equation is quite relevant  in quantum mechanics, since
it appears in the modern quantum mechanical description of the hydrogen atom.
Let us consider the case $\nu=0$ i.e., the simplified Laguerre equation given by 
$x y ^{\prime \prime} + (1 -x) y^\prime + \lambda y=0$ for $\lambda \in \mathbb K$. 
It is well-known that the solutions by series are polynomial when $\lambda$ is a nonnegative integer. This is the so called {\it Laguerre polynomial}.

\subsubsection{Laguerre PDE of type I}
As before, based on the condition that the PDE must have its restrictions to lines $y=tx$ as given by 
Laguerre EDOs we obtain the following model. 
Consider the equation
\begin{equation}\label{eqlaguerre1}
x^2\frac{\partial^2 z}{\partial x^2}+2xy\frac{\partial^2 z}{\partial x\partial y}+y^2\frac{\partial^2 z}{\partial y^2}+(1-x)x\frac{\partial z}{\partial x}+(1-x)y\frac{\partial z}{\partial y}+\lambda xz=0,
\end{equation}
where $\lambda\in\mathbb{R}$. Let $\varphi$ be a solution  of (\ref{eqlaguerre1}) of the form
\begin{equation}\label{eqlaguerre2}
\varphi(x,y)=x^ry^s\sum^{\infty}_{|Q|=0}d_QX^Q
\end{equation}
where $d_{0,0}\neq0$. Then
$$\frac{\partial\varphi}{\partial x}=\sum^\infty_{|Q|=0}(q_1+r)d_Qx^{q_1+r-1}y^{q_2+s} $$
$$\frac{\partial\varphi}{\partial y}=\sum^\infty_{|Q|=0}(q_2+s)d_Qx^{q_1+r}y^{q_2+s-1} $$
$$\frac{\partial^2\varphi}{\partial x^2}=\sum^\infty_{|Q|=0}(q_1+r)(q_1+r-1)d_Qx^{q_1+r-2}y^{q_2+s} $$
$$\frac{\partial^2\varphi}{\partial x \partial y}=\sum^\infty_{|Q|=0}(q_1+r)(q_2+s)d_Qx^{q_1+r-1}y^{q_2+s-1} $$
$$\frac{\partial^2\varphi}{\partial y^2}=\sum^\infty_{|Q|=0}(q_2+s)(q_2+s-1)d_Qx^{q_1+r}y^{q_2+s-2} $$

and from this we have
$$x^2\frac{\partial^2\varphi}{\partial x^2}=x^ry^s\sum^\infty_{|Q|=0}(q_1+r)(q_1+r-1)d_QX^Q$$
$$2xy\frac{\partial^2\varphi}{\partial x \partial y}=x^ry^s\sum^\infty_{|Q|=0}2(q_1+r)(q_2+s)d_QX^Q$$
$$y^2\frac{\partial^2\varphi}{\partial y^2}=x^ry^s\sum^\infty_{|Q|=0}(q_2+s)(q_2+s-1)d_QX^Q$$
$$\begin{array}{cl}(1-x)x\dfrac{\partial\varphi}{\partial x}&=(1-x)x^ry^s\sum^\infty_{|Q|=0}(q_1+r)d_QX^Q\\&
\\&=x^ry^s\sum^\infty_{|Q|=0}(q_1+r)d_QX^Q-x^ry^s\sum^\infty_{|Q|=0}(q_1+r)d_Qx^{q_1+1}y^{q_2}\\
&\\
&=x^ry^s\sum^\infty_{|Q|=0}(q_1+r)d_QX^Q-x^ry^s\sum^\infty_{|Q|=1}(q_1+r-1)d_{q_1-1,q_2}X^Q
\end{array}$$
$$\begin{array}{cl}(1-x)y\dfrac{\partial\varphi}{\partial y}&=(1-x)x^ry^s\sum^\infty_{|Q|=0}(q_2+s)d_QX^Q\\&
\\&=x^ry^s\sum^\infty_{|Q|=0}(q_2+s)d_QX^Q-x^ry^s\sum^\infty_{|Q|=0}(q_2+s)d_Qx^{q_1+1}y^{q_2}\\
&\\
&=x^ry^s\sum^\infty_{|Q|=0}(q_2+s)d_QX^Q-x^ry^s\sum^\infty_{|Q|=1}(q_2+s)d_{q_1-1,q_2}X^Q
\end{array}$$
$$\lambda x\varphi=x^ry^s\sum^{\infty}_{|Q|=0}\lambda d_Qx^{q_1+1}y^{q_2}=x^ry^s\sum^{\infty}_{|Q|=1}\lambda d_{q_1-1,q_2}X^Q.$$
Given that $\varphi$ is solution of (\ref{eqlaguerre1}) we have
$$x^2\frac{\partial^2 \varphi}{\partial x^2}+2xy\frac{\partial^2 \varphi}{\partial x\partial y}+y^2\frac{\partial^2 \varphi}{\partial y^2}+(1-x)x\frac{\partial \varphi}{\partial x}+(1-x)y\frac{\partial \varphi}{\partial y}+\lambda x\varphi=0.$$
Therefore
$$\begin{array}{c} x^ry^s\big[\sum^\infty_{|Q|=0}\big[(q_1+r)(q_1+r-1)+2(q_1+r)(q_2+s)+(q_2+s)(q_2+s-1)+(q_1+r)+(q_2+s)\big]d_QX^Q \\ \\ +\sum^{\infty}_{|Q|=1}[-(q_1+r-1)-(q_2+s)+\lambda]d_{q_1-1,q_2}X^Q\big]=0
\end{array}$$
$$\sum^\infty_{|Q|=0}(q_1+q_2+r+s)^2d_QX^Q-\sum^{\infty}_{|Q|=1}[q_1+q_2+r+s-(1+\lambda)]d_{q_1-1,q_2}X^Q=0$$
$$\sum^\infty_{|Q|=0}(|Q|+r+s)^2d_QX^Q-\sum^{\infty}_{|Q|=1}[|Q|+r+s-(1+\lambda)]d_{q_1-1,q_2}X^Q=0$$
$$\begin{array}{c}
(r+s)^2d_{0,0}+\sum^{\infty}_{|Q|=1}\big[(|Q|+r+s)^2d_Q-[|Q|+r+s-(1+\lambda)]d_{q_1-1,q_2}\big]X^Q=0\end{array}$$
and then $(r+s)^2d_{0,0}=0$  and
$$(|Q|+r+s)^2d_Q-[|Q|+r+s-(1+\lambda)]d_{q_1-1,q_2}=0,\;\mbox{ for every }|Q|=1,2,3,\ldots$$
Given that $d_{0,0}\neq0$ we have that
\begin{equation}\label{eqlaguerre3}
r+s=0.
\end{equation}
Let $(r,s)$ point of (\ref{eqlaguerre3}). Note that $|Q|+r+s\neq0$ for all $|Q|=1,2,\ldots$

Thus we have
$$d_Q=\frac{[|Q|-(1+\lambda)]d_{q_1-1,q_2}}{|Q|^2},\;\mbox{ for every }|Q|=1,2,\ldots.$$
Observe that
$$d_{0,1}=0\mbox{ and }d_{1,0}=\frac{-\lambda d_{0,0}}{1^2}=-\lambda d_{0,0}$$

$$d_{0,2}=0\mbox{ and }d_{1,1}=\frac{[1-\lambda]d_{0,1}}{2^2}=0,d_{2,0}=\frac{[1-\lambda]d_{1,0}}{2^2}=\frac{[1-\lambda][-\lambda]d_{0,0}}{2^2}$$

$$\begin{array}{c} d_{0,3}=0\mbox{ and }d_{2,1}=\frac{[2-\lambda]d_{1,1}}{3^2}=0,d_{1,2}=\frac{[2-\lambda]d_{0,2}}{3^2}=0, d_{3,0}=\frac{[2-\lambda]d_{2,0}}{3^2}=\frac{[2-\lambda][1-\lambda][-\lambda]d_{0,0}}{3^22^2}\end{array}$$

In general
$$d_{q_1,q_2}=0\;\;(q_1,q_2)\neq(n,0)$$
and
$$d_{n,0}=\frac{[n-1-\lambda]\cdots[-\lambda]d_{0,0}}{n^2\cdots1^2}=\frac{[n-1-\lambda]\cdots[-\lambda]d_{0,0}}{(n!)^2}\;\;\;n=1,2,\ldots$$
Choosing $d_{0,0}=1$ we have that
\begin{equation}\label{eqlaguerre5}
\varphi(x,y)=x^{r}y^{s}\big[1+\sum^\infty_{n=1}\frac{[n-1-\lambda]\cdots[-\lambda]x^n}{(n!)^2}\big]
\end{equation}
 is solution of (\ref{eqlaguerre1}) where $(r,s)$ verifies (\ref{eqlaguerre3}).

\subsubsection{Laguerre equation II}
As for the symmetric model of Laguerre PDE we consider the equation
\begin{equation}\label{2eqlaguerre1}
x^2\frac{\partial^2 z}{\partial x^2}+2xy\frac{\partial^2 z}{\partial x\partial y}+y^2\frac{\partial^2 z}{\partial y^2}+(1-xy)x\frac{\partial z}{\partial x}+(1-xy)y\frac{\partial z}{\partial y}+\lambda xyz=0,
\end{equation}
where $\lambda\in\mathbb{R}$. Let $\varphi$ be a solution  of (\ref{2eqlaguerre1}) of the form
\begin{equation}\label{2eqlaguerre2}
\varphi(x,y)=x^ry^s\sum^{\infty}_{|Q|=0}d_QX^Q
\end{equation}
where $d_{0,0}\neq0$. Then
$$\frac{\partial\varphi}{\partial x}=\sum^\infty_{|Q|=0}(q_1+r)d_Qx^{q_1+r-1}y^{q_2+s} $$
$$\frac{\partial\varphi}{\partial y}=\sum^\infty_{|Q|=0}(q_2+s)d_Qx^{q_1+r}y^{q_2+s-1} $$
$$\frac{\partial^2\varphi}{\partial x^2}=\sum^\infty_{|Q|=0}(q_1+r)(q_1+r-1)d_Qx^{q_1+r-2}y^{q_2+s} $$
$$\frac{\partial^2\varphi}{\partial x \partial y}=\sum^\infty_{|Q|=0}(q_1+r)(q_2+s)d_Qx^{q_1+r-1}y^{q_2+s-1} $$
$$\frac{\partial^2\varphi}{\partial y^2}=\sum^\infty_{|Q|=0}(q_2+s)(q_2+s-1)d_Qx^{q_1+r}y^{q_2+s-2} $$

and from this we have
$$x^2\frac{\partial^2\varphi}{\partial x^2}=x^ry^s\sum^\infty_{|Q|=0}(q_1+r)(q_1+r-1)d_QX^Q$$
$$2xy\frac{\partial^2\varphi}{\partial x \partial y}=x^ry^s\sum^\infty_{|Q|=0}2(q_1+r)(q_2+s)d_QX^Q$$
$$y^2\frac{\partial^2\varphi}{\partial y^2}=x^ry^s\sum^\infty_{|Q|=0}(q_2+s)(q_2+s-1)d_QX^Q$$
$$\begin{array}{cl}(1-xy)x\dfrac{\partial\varphi}{\partial x}&=(1-xy)x^ry^s\sum^\infty_{|Q|=0}(q_1+r)d_QX^Q\\&
\\&=x^ry^s\sum^\infty_{|Q|=0}(q_1+r)d_QX^Q-x^ry^s\sum^\infty_{|Q|=0}(q_1+r)d_Qx^{q_1+1}y^{q_2+1}\\
&\\
&=x^ry^s\sum^\infty_{|Q|=0}(q_1+r)d_QX^Q-x^ry^s\sum^\infty_{|Q|=2}(q_1+r-1)d_{q_1-1,q_2-1}X^Q
\end{array}$$
$$\begin{array}{cl}(1-xy)y\dfrac{\partial\varphi}{\partial y}&=(1-xy)x^ry^s\sum^\infty_{|Q|=0}(q_2+s)d_QX^Q\\&
\\&=x^ry^s\sum^\infty_{|Q|=0}(q_2+s)d_QX^Q-x^ry^s\sum^\infty_{|Q|=0}(q_2+s)d_Qx^{q_1+1}y^{q_2+1}\\
&\\
&=x^ry^s\sum^\infty_{|Q|=0}(q_2+s)d_QX^Q-x^ry^s\sum^\infty_{|Q|=2}(q_2+s-1)d_{q_1-1,q_2-1}X^Q
\end{array}$$
$$\lambda xy\varphi=x^ry^s\sum^{\infty}_{|Q|=0}\lambda d_Qx^{q_1+1}y^{q_2+1}=x^ry^s\sum^{\infty}_{|Q|=2}\lambda d_{q_1-1,q_2-1}X^Q.$$Given that $\varphi$ is solution of (\ref{2eqlaguerre1}) we have
$$x^2\frac{\partial^2 \varphi}{\partial x^2}+2xy\frac{\partial^2 \varphi}{\partial x\partial y}+y^2\frac{\partial^2 \varphi}{\partial y^2}+(1-xy)x\frac{\partial \varphi}{\partial x}+(1-xy)y\frac{\partial \varphi}{\partial y}+\lambda xy\varphi=0.$$
Therefore
$$\begin{array}{c} x^ry^s\big[\sum^\infty_{|Q|=0}\big[(q_1+r)(q_1+r-1)+2(q_1+r)(q_2+s)+(q_2+s)(q_2+s-1)+(q_1+r)+(q_2+s)\big]d_QX^Q \\ \\ -\sum^{\infty}_{|Q|=2}[(q_1+r-1)+(q_2+s-1)-\lambda]d_{q_1-1,q_2-1}X^Q\big]=0
\end{array}$$
$$\begin{array}{c} \sum^\infty_{|Q|=0}(q_1+q_2+r+s)^2d_QX^Q-\sum^{\infty}_{|Q|=2}[q_1+q_2+r+s-(2+\lambda)]d_{q_1-1,q_2-1}X^Q=0
\end{array}$$
$$\begin{array}{c} \sum^\infty_{|Q|=0}(|Q|+r+s)^2d_QX^Q-\sum^{\infty}_{|Q|=2}[|Q|+r+s-(2+\lambda)]d_{q_1-1,q_2-1}X^Q=0
\end{array}$$

$$\begin{array}{c}
(r+s)^2d_{0,0}+(1+r+s)^2\big[d_{1,0}x+d_{0,1}y\big]+\\
\\\sum^{\infty}_{|Q|=2}[(|Q|+r+s)^2d_Q-[|Q|+r+s-(2+\lambda)]d_{q_1-1,q_2-1}]X^Q=0 \end{array}$$
and then $$(r+s)^2d_{0,0}=0,\;(1+r+s)^2d_{1,0}=0,\;(1+r+s)^2d_{0,1}=0$$
 and
$$(|Q|+r+s)^2d_Q-[|Q|+r+s-(2+\lambda)]d_{q_1-1,q_2-1}=0,\;\mbox{ for every }|Q|=2,3,\ldots$$
Given that $d_{0,0}\neq0$ we have that
\begin{equation}\label{2eqlaguerre3}
r+s=0.
\end{equation}
Let $(r,s)$ point of (\ref{2eqlaguerre3}). Note that $|Q|+r_0+s_0\neq0$ for all $|Q|=1,2,\ldots$.

Thus we have $d_{1,0}=d_{0,1}=0$  and
$$d_Q=\dfrac{[|Q|-(2+\lambda)]d_{q_1-1,q_2-1}}{|Q|^2},\;\mbox{ for every }|Q|=2,3,\ldots$$
Observe that
$$d_{0,2}=d_{2,0}=0\mbox{ and }d_{1,1}=\frac{[-\lambda]d_{0,0}}{2^2}$$

$$\begin{array}{c} d_{0,3}=d_{3,0}=0\mbox{ and }d_{2,1}=\frac{[1-\lambda]d_{1,0}}{3^2}=0,d_{1,2}=\frac{[1-\lambda]d_{0,1}}{3^2}=0\end{array}$$

$$\begin{array}{c}d_{0,4}=d_{4,0}=0\mbox{ and }d_{3,1}=\frac{[2-\lambda]d_{2,0}}{4^2}=0,\\
\\d_{1,3}=\frac{[2-\lambda]d_{0,2}}{4^2}=0,d_{2,2}=\frac{[2-\lambda]d_{1,1}}{4^2}=
\frac{[2-\lambda][-\lambda]d_{0,0}}{4^22^2}
\end{array}$$

$$\begin{array}{c} d_{0,5}=d_{5,0}=0\mbox{ and }d_{4,1}=\frac{[3-\lambda]d_{3,0}}{5^2}=0,d_{3,2}=\frac{[3-\lambda]d_{2,1}}{5^2}=0\\ \\
d_{2,3}=\frac{[3-\lambda]d_{1,2}}{5^2}=0,d_{1,4}=\frac{[3-\lambda]d_{0,3}}{5^2}=0
\end{array}$$

$$\begin{array}{c}d_{0,6}=d_{6,0}=0\mbox{ and }d_{5,1}=\frac{[4-\lambda]d_{4,0}}{6^2}=0,\\
\\d_{4,2}=\frac{[4-\lambda]d_{3,1}}{6^2}=0,d_{2,4}=\frac{[4-\lambda][-\lambda]d_{1,3}}{6^2}=0,d_{1,5}=\frac{[4-\lambda]d_{0,4}}{6^2}=0\\
\\  d_{3,3}=\frac{[4-\lambda]d_{2,2}}{6^2}=\frac{[4-\lambda][2-\lambda][-\lambda]d_{0,0}}{6^24^22^2}.
\end{array}$$

In general
$$d_{q_1,q_2}=0\;\;(q_1,q_2)\neq(n,n)$$
and
$$d_{n,n}=\frac{[2(n-1)-\lambda]\cdots[-\lambda]d_{0,0}}{2^{2n}(n!)^2}\;\;\;n=1,2,\ldots$$
Choosing $d_{0,0}=1$ we have that
\begin{equation}\label{2eqlaguerre5}
\varphi(x,y)=x^{r}y^{s}\big[1+\sum^\infty_{n=1}\frac{[2(n-1)-\lambda]\cdots[-\lambda](xy)^n}{2^{2n}(n!)^2}\big]
\end{equation}
 is solution of (\ref{2eqlaguerre1}) where $(r,s)$ verifies (\ref{2eqlaguerre3}).

\subsubsection{Laguerre polynomials in two variables}
Both models Laguerre type I and II above exhibit solutions giving rise to polynomials  
for suitable values of the parameter $\lambda$. Indeed, for instance for $r_0=s_0=1/2$ and $\lambda=n$ 
we have the first model giving rise to solutions of  the form 
$\vr(x,y)=(xy)^{\frac{1}{2}}L_n(x)$ for some polynomial $L_n(x)$ of degree $n$.
This is however a one-variable polynomial. 
On the other hand, the second model Laguerre type II gives, for  $r_0=s_0=1/2$ and even $\lambda=2n-2$ 
a solution of the form  $\vr(x,y)=(xy)^{\frac{1}{2}}\tilde L_n(xy)$ for some polynomial $L(t)$ of degree $n$.

\subsection{Disturbed heat equation}

The next example shows the efficacy or our methods when compared to the classical separation of variables methods. It is an example that cannot have separated variables, but fits into our approach. It is based on the heat diffusion equation in a two dimension plaque. A perturbation is included in the equation as
result of some external influence as a source of heat for instance:
\begin{Example}[Disturbed heat equation]{\rm  We shall now apply our techniques in a PDE
that cannot be solved by the  usual method of separation of variables. Let us consider the following perturbed heat diffusion equation
\begin{equation}\label{eqheatdisturbed1}
     a^2\frac{\partial^2 z}{\partial x^2} =\frac{\partial z}{\partial y}+\alpha(x,y)\frac{\partial z}{\partial x},\;\;x>0,y>0
     \end{equation}
where $a$ is a positive constant and $\alpha$ is analytic.
Notice that putting $Z(x,y)=X(x)Y(y)$ and substituting in the PDE we obtain
\[
a^2 \frac{X^{\prime\prime}(x)}{X(x)}= \frac{Y^\prime(y)}{Y(y)} + \frac{\alpha(x,y)}{Y(y)}
 \]
 and therefore the PDE is not always separable variables equation.
Let us now solve this equation by our methods in some concrete examples.

 Making the change $x=\ln u, y=\ln v$ we transform the equation (\ref{eqheatdisturbed1}) in
\begin{equation}\label{eqheatdisturbed2}
a^2u^2\frac{\partial^2 \tilde{z}}{\partial u^2}+u(a^2-\alpha(\ln u,\ln v))\frac{\partial \tilde{z}}{\partial u}-v\frac{\partial \tilde{z}}{\partial v}=0.
\end{equation}
For example, if we consider $\alpha(x,y)=e^{x+y}$ in (\ref{eqheatdisturbed2}) we obtain
\begin{equation}\label{eqheatdisturbed3}
a^2u^2\frac{\partial^2 \tilde{z}}{\partial u^2}+u(a^2-uv)\frac{\partial \tilde{z}}{\partial u}-v\frac{\partial \tilde{z}}{\partial v}=0.
\end{equation}
Note that equation (\ref{eqheatdisturbed3}) is a parabolic equation with regular singularity.

Let $\tilde{\varphi}$ be a solution  of (\ref{eqheatdisturbed3}) of the form
\begin{equation}\label{eqheatdisturbed4}
\tilde{\varphi}(u,v)=u^rv^s\sum^{\infty}_{|Q|=0}d_Qu^{q_1}v^{q_2}
\end{equation}
where $d_{0,0}\neq0$. Then
$$\frac{\partial\tilde{\varphi}}{\partial u}=\sum^\infty_{|Q|=0}(q_1+r)d_Qu^{q_1+r-1}v^{q_2+s} $$
$$\frac{\partial\tilde{\varphi}}{\partial v}=\sum^\infty_{|Q|=0}(q_2+s)d_Qu^{q_1+r}v^{q_2+s-1} $$
$$\frac{\partial^2\tilde{\varphi}}{\partial u^2}=\sum^\infty_{|Q|=0}(q_1+r)(q_1+r-1)d_Qu^{q_1+r-2}v^{q_2+s} $$
and from this we have
$$a^2u^2\frac{\partial^2\tilde{\varphi}}{\partial u^2}=u^rv^s\sum^\infty_{|Q|=0}a^2(q_1+r)(q_1+r-1)d_Qu^{q_1}v^{q_2}$$

$$\begin{array}{c c l}u(a^2-uv)\frac{\partial\tilde{\varphi}}{\partial u}&=&(a^2-uv)u^rv^s\sum^\infty_{|Q|=0}(q_1+r)d_Qu^{q_1}v^{q_2}\\ &&
\\&=&u^rv^s\left(\sum^\infty_{|Q|=0}a^2(q_1+r)d_Qu^{q_1}v^{q_2}-\sum^\infty_{|Q|=2}(q_1-1+r)d_{q_1-1,q_2-1}u^{q_1}v^{q_2}\right)\end{array}$$
$$-v\frac{\partial\tilde{\varphi}}{\partial v}=u^rv^s\sum^\infty_{|Q|=0}-(q_2+s)d_Qu^{q_1}v^{q_2}.$$
Given that $\tilde{\varphi}$ is solution of (\ref{eqheatdisturbed3}) we have
$$a^2u^2\frac{\partial^2 \tilde{\varphi}}{\partial u^2}+u(a^2-uv)\frac{\partial \tilde{\varphi}}{\partial u}-v\frac{\partial \tilde{\varphi}}{\partial v}=0$$
Therefore
$$u^rv^s\big[\sum^\infty_{|Q|=0}\big[a^2(q_1+r)(q_1+r-1)+a^2(q_1+r)-(q_2+s)\big]d_Qu^{q_1}v^{q_2}+\sum^{\infty}_{|Q|=2}d_{q_1-1,q_2-1}u^{q_1}v^{q_2}\big]=0 $$
$$\sum^\infty_{|Q|=0}\big[a^2(q_1+r)^2-(q_2+s)\big]d_Qu^{q_1}v^{q_2}+\sum^{\infty}_{|Q|=2}d_{q_1-1,q_2-1}u^{q_1}v^{q_2}=0$$
$$(r^2-s)d_{0,0}+\big[(r+1)^2-s)d_{1,0}u+(r^2-(s+1))d_{0,1}v\big] +\sum^{\infty}_{|Q|=2}\big(\big[a^2(q_1+r)^2-(q_2+s)\big]d_Q+d_{q_1-1,q_2-1}\big)u^{q_1}v^{q_2}=0$$
and then $$(a^2r^2-s)d_{0,0}=0,\;\;\;(a^2(r+1)^2-s)d_{1,0}=0,\;\;\;(a^2r^2-(s+1))d_{0,1}=0$$ and
$$\big[a^2(q_1+r)^2-(q_2+s)\big]d_Q+d_{q_1-1,q_2-1}=0,\;\mbox{ for every }|Q|=2,3,\ldots$$
Given that $d_{0,0}\neq0$ we have that
\begin{equation}\label{eqheatdisturbed5}
a^2r^2-s=0.
\end{equation}
Let $(r,s)$ root of (\ref{eqheatdisturbed5}) such that
\begin{equation}\label{eqheatdisturbed6}
(r,s)\notin\{(r_1,s_1)\in\mathbb{C}^2;\;a^2(q_1+r_1)^2-(q_2+s_1)=0,\;\mbox{ for some }|Q|=1,2,\ldots\}.
\end{equation}
Thus we have $d_{1,0}=d_{0,1}=0$ and
$$d_Q=-\frac{d_{q_1-1,q_2-1}}{a^2(q_1+r)^2-(q_2+s)},\;\mbox{ for every }|Q|=2,3,\ldots.$$
Observe that
$$d_{2,0}=d_{0,2}=0\mbox{ and }d_{1,1}=-\frac{d_{0,0}}{a^2(1+r)^2-(s+1)}$$

$$d_{3,0}=d_{0,3}=0\mbox{ and }d_{2,1}=-\frac{d_{1,0}}{a^2(2+r)^2-(1+s)}=0,d_{1,2}=-\frac{d_{0,1}}{a^2(1+r)^2-(2+s)}=0$$
$$\begin{array}{c}d_{4,0}=d_{0,4}=0\mbox{ and }d_{3,1}=-\frac{d_{2,0}}{a^2(3+r)^2-(1+s)}=0,d_{1,3}=-\frac{d_{0,2}}{a^2(1+r)^2-(3+s)}=0,\\
\\
d_{2,2}=-\frac{d_{1,1}}{a^2(2+r)^2-(2+s)}=(-1)^2\frac{d_{0,0}}{(a^2(1+r)^2-(1+s))(a^2(2+r)^2-(2+s))}\end{array} $$

$$\begin{array}{c}d_{5,0}=d_{0,5}=0\mbox{ and }d_{4,1}=-\frac{d_{3,0}}{a^2(4+r)^2-(1+s)}=0,d_{1,4}=-\frac{d_{0,3}}{a^2(1+r)^2-(4+s)}=0,\\
\\
d_{3,2}=-\frac{d_{2,1}}{a^2(3+r)^2-(2+s)}=0,d_{2,3}=-\frac{d_{1,2}}{a^2(2+r)^2-(3+s)}=0\end{array} $$

$$\begin{array}{c}d_{6,0}=d_{0,6}=0\mbox{ and }d_{5,1}=-\frac{d_{4,0}}{a^2(5+r)^2-(1+s)}=0,d_{1,5}=-\frac{d_{0,4}}{a^2(1+r)^2-(5+s)}=0,\\
\\
d_{4,2}=-\frac{d_{3,1}}{a^2(4+r)^2-(2+s)}=0,d_{2,4}=-\frac{d_{1,3}}{a^2(2+r)^2-(4+s)}=0\\
\\
d_{3,3}=-\frac{d_{2,2}}{a^2(3+r)^2-(3+s)}=(-1)^3\frac{d_{0,0}}{(a^2(1+r)^2-(1+s))(a^2(2+r)^2-(2+s))(a^2(3+r)^2-(3+s))}.
\end{array} $$
In general
$$d_{q_1,q_2}=0\;\;q_1\neq q_2$$
and
$$d_{n,n}=\frac{(-1)^nd_{0,0}}{(a^2(1+r)^2-(1+s))(a^2(2+r)^2-(2+s))\cdots(a^2(n+r)^2-(n+s))}\;\;\;n=1,2,\ldots$$
Choosing $d_{0,0}=1$ we have that
\begin{equation}\label{eqheatdisturbed7}
\tilde{\varphi}(u,v)=u^{r}v^{s}\big[1+\sum^\infty_{n=1}\frac{(-1)^n(uv)^n}{(a^2(1+r)^2-(1+s))(a^2(2+r)^2-(2+s))\cdots(a^2(n+r)^2-(n+s))}\big]
\end{equation}
is solution of (\ref{eqheatdisturbed3}) where $(r_0,s_0)$ verifies (\ref{eqheatdisturbed5}) and (\ref{eqheatdisturbed6}). Therefore
$$\varphi(x,y)=\tilde{\varphi}(e^x,e^y)=e^{xr+ys}\big[1+\sum^\infty_{n=1}\frac{(-1)^ne^{n(x+y)}}{(a^2(1+r)^2-(1+s))(a^2(2+r)^2-(2+s))\cdots(a^2(n+r)^2-(n+s))}\big]$$
is solution of
$$     a^2\frac{\partial^2 z}{\partial x^2} =\frac{\partial z}{\partial y}+e^{x+y}\frac{\partial z}{\partial x}$$
 where $(r,s)$ verifies (\ref{eqheatdisturbed5}) and (\ref{eqheatdisturbed6}).

}
\end{Example}

\bibliographystyle{amsalpha}

\end{document}